\documentclass[reqno,11pt,letterpaper]{amsart}
\usepackage{amsmath,amssymb,amsthm,graphicx,mathrsfs,url}
\usepackage[usenames,dvipsnames]{color}
\usepackage[colorlinks=true,linkcolor=Red,citecolor=Green]{hyperref}
\usepackage{amsxtra}

\usepackage{bbm}
\usepackage{epstopdf}

\usepackage[dvipsnames]{xcolor}

\usepackage{verbatim}
\usepackage{graphicx}
\newcommand\smallO{
  \mathchoice
    {{\scriptstyle\mathcal{O}}}
    {{\scriptstyle\mathcal{O}}}
   {{\scriptscriptstyle\mathcal{O}}}
    {\scalebox{.7}{$\scriptscriptstyle\mathcal{O}$}}
 }

\def\?[#1]{\textbf{[#1]}\marginpar{\Large{\textbf{??}}}}

\numberwithin{equation}{section}

\let\epsilon=\varepsilon 

\usepackage{etoolbox}
\makeatletter
\patchcmd{\@setaddresses}
  {\par\addvspace\bigskipamount\indent}
  {\par\addvspace\smallskipamount\indent}
  {}{}
\makeatother

\allowdisplaybreaks
\usepackage[
left=0.65in, right=0.65in]{geometry}

\usepackage{amsxtra}
\usepackage{epstopdf}
\newtheorem{theorem}{Theorem}
\newtheorem{proposition}{Proposition}[section]

\newtheorem{lemma}[proposition]{Lemma}

\newtheorem{remark}{Remark}

\usepackage{tikz,amsmath,amssymb}
\usetikzlibrary{patterns,arrows.meta}
\usepackage{graphicx}

\title{The Inhomogeneous Kinetic FPUT Equation near Rayleigh-Jeans Equilibria}

\author{Angeliki Menegaki}  
\address{Department of Mathematics, Huxley building, South Kensington campus, Imperial College London, London SW7 2AZ, United Kingdom}
\address{ Department of Mathematics, University of Athens, Panepistimiopolis 157 84, Athens, Greece}
\email{a.menegaki@imperial.ac.uk}

\author{Haoling Xiang} 
\address{Department of Mathematics, Huxley building, South Kensington campus, Imperial College London, London SW7 2AZ, United Kingdom}
\email{h.xiang23@imperial.ac.uk}

\begin{document}
\maketitle

\begin{abstract}
We study the nonlinear dynamics of the spatially inhomogeneous
kinetic wave equation associated with the FPUT-$\beta$ problem.
We prove global-in-time well-posedness near nonsingular
Rayleigh-Jeans equilibria and their asymptotic stability in particular in 
weighted $L^\infty$ norms in space and momentum, for small
initial perturbations orthogonal to the collision kernel
at almost every spatial point. 
We describe the spectrum near the origin at small spatial
frequencies, identifying two simple eigenvalues with
fractional and diffusive orders $|\xi|^{8/5}$ and $\xi^2$,
respectively. 
Combining this spectral analysis with energy estimates
taking into account the transport effect, we establish 
algebraic decay for the linear semigroup despite the absence of a spectral gap for the collision operator. We then use these estimates to obtain nonlinear stability. 
\end{abstract}

\tableofcontents

\section{Introduction}
\label{sec:introduction}

We study the spatially inhomogeneous phonon Boltzmann equation, also
called the Boltzmann-Peierls equation, arising in the kinetic limit from the weakly
anharmonic microscopic Fermi-Pasta-Ulam-Tsingou-$\beta$ (FPUT-$\beta$) chain. 
We prove global-in-time well-posedness and polynomial decay for small
perturbations of a nonsingular Rayleigh-Jeans thermodynamic equilibrium. 
To our knowledge, this is the first global stability result near such equilibria for the spatially inhomogeneous kinetic FPUT equation.

The FPUT-$\beta$ system is a chain of anharmonic oscillators on the lattice, where each oscillator is tracked by both its position $q_i$ and momentum $p_i$. Its Hamiltonian energy is formally given by
$$ H(p,q)=\sum_{i\in\mathbb{Z}}
 \left[ \frac{p_i^2}{2} + F_\lambda(q_{i+1}-q_i) \right],
 \quad \text{where }\quad F_\lambda(r)=\frac{r^2}{2}+\lambda\frac{r^4}{4},
 \qquad \lambda>0.
$$
It consists of the kinetic energy and the  nearest-neighbour interactions. 

Such chains have a long history and go back to the pioneering numerical experiments at
Los Alamos in the $1950$s. Despite the nonlinearity in the dynamics, these numerical experiments revealed recurrent behaviour in which the energy distribution returned close to its initial state rather than being rapidly equipartitioned among the modes, as Fermi and his team had expected. 
This is the FPUT paradox, which stimulated extensive work on nonlinear Hamiltonian dynamics, with connections to KAM theory,
solitons and integrable systems, and wave turbulence, see for example  
\cite{fiftyFPU,OLDC2023,GLM26review} and the references therein. 

These observations raised fundamental questions about thermalisation in anharmonic oscillator chains. 
A related question concerns their ability to transport energy
and whether Fourier's law holds. 
The phonon Boltzmann equation describes their energy transfer through resonant interactions between normal modes in the weakly anharmonic regime. Understanding its long-time behaviour is therefore an important part of the kinetic approach to thermal conductivity.  

We first present the equation and our results, and later discuss their relation to this physical background and to previous mathematical work. 

\subsection{The equation and the perturbation}

We consider the inhomogeneous kinetic FPUT equation
\begin{equation} \label{eq:model}
 \partial_t F + \omega'(p)\partial_xF =\mathcal{C}[F], 
 \qquad (t,x,p)\in\mathbb{R}_+
 \times\mathbb{R} \times \mathbb{T},
\end{equation}
where $\mathbb{T} = 
\mathbb{R}/(2\pi\mathbb{Z})$. We use normalised units in which the dispersion relation is 
$$ \omega(p) = \left|\sin \frac{p}{2} \right|,
 \qquad v(p) :=\omega'(p).
$$

The four-phonon collision operator has the formal expression
\begin{equation}\label{eq:collision_op}
\begin{split}
 \mathcal{C}[F](p_0)
 &= \frac{1}{2} \int_{\mathbb{T}^3}
 \delta_{\mathbb{T}}(p_0+p_1-p_2-p_3)
 \delta(\omega_0+\omega_1-\omega_2-\omega_3)
 \prod_{\ell=0}^3\omega_\ell \prod_{\ell=0}^3 F_\ell 
 \\
 &\hspace{6cm}\times
 \left[\frac{1}{F_1} + \frac{1}{F_0} - \frac{1}{F_3} - \frac{1}{F_2}
 \right] 
  dp_1 dp_2 d p_3,
\end{split}
\end{equation}
where $F_j=F(t,x,p_j)$ and $\omega_j=\omega(p_j)$. 
Here the two Dirac distributions impose the resonance conditions
$p_0+p_1=p_2+p_3$ modulo $2\pi$ and
$\omega(p_0)+\omega(p_1)=\omega(p_2)+\omega(p_3)$.
These conditions define the resonant set on which the
interactions take place. They express conservation of wave number (modulo $2\pi$), and energy. 

The operator is understood through the parametrisation of 
the nontrivial resonant branch, 
\cite[Section 1.3]{GLM26} and \cite{LukkSpohn}. 
For fixed $p_0,p_2$, the nontrivial
solution to the resonant constraints is given by 
\begin{equation} \label{eq:parametrisation_nontrivial_branch}
\begin{split} 
 &p_1=h(p_0,p_2) \pmod{2\pi}, \qquad
p_3=p_0+p_1-p_2\pmod{2\pi},
 \\ &h(p_0,p_2) :=\frac{p_2-p_0}{2}
 +2\arcsin\left(
\tan\frac{|p_2-p_0|}{4}
 \cos\frac{p_2+p_0}{4} \right).
 \end{split}
\end{equation} 
The corresponding collision kernel is
$$
\sigma_{02}:=\frac{\omega_0\omega_1\omega_2\omega_3}{\sqrt{F_+(p_0,p_2)}}, \qquad F_+(p_0,p_2)
    =\left(\cos\frac{p_0}{2}+\cos\frac{p_2}{2}\right)^2
      +4\sin\frac{p_0}{2}\sin\frac{p_2}{2}.
$$

Writing $F_j=F(p_j)$, the collision operator on this branch is then 
$$ \mathcal{C}[F](p_0) = \int_0^{2\pi}\sigma_{02} \big[
 F_1F_2F_3+F_0F_2F_3 - F_0F_1F_3-F_0F_1F_2
\big] dp_2.
$$

The equation admits the spatially homogeneous stationary solutions 
$$
\mathfrak{f}(p)=(\beta\omega(p)+\gamma)^{-1},
    \qquad \beta,\gamma>0,
$$
called nonsingular Rayleigh- Jeans (RJ) equilibria.  Let us also note that $\mathfrak{f}$ is the unique maximiser of the  entropy
$$
H(F)=\int_{\mathbb{T} } \log F(p) dp
$$ 
among distributions with the same mass and energy as 
$\mathfrak{f}$. For a general characterisation of entropy
maximisers for kinetic wave equations on tori, we refer to
\cite{EGLM25}.  

We fix one such equilibrium and look at solutions of the form 
$$ F=\mathfrak{f} (1+g).$$ 

The perturbation $g$ satisfies 
\begin{equation}\label{eq:nonlinear_eq}
 \partial_t g + v(p)\partial_x g  
 = Lg + \Gamma(g,g) + Q(g,g,g),
 \qquad g(0)=g_0,
\end{equation} 
where $\Gamma$ and $Q$ are the symmetric quadratic and cubic
terms in the collision expansion. Their explicit formulas are given in the nonlinear stability Section \ref{sec:nonlinear}. 

The linearised collision operator $L$ is self-adjoint and nonpositive on $L^2_p$, 
and has the form
$$ L= -a +K ,\qquad a(p)=a_0\omega(p)^{5/3}
+\smallO \left(\omega(p)^{5/3}\right), \quad a_0>0,\qquad \text{as } p\to 0^+,\ 2\pi^- $$ 
with $K$ compact, $a$ a multiplication operator and $\ker L=\operatorname{span}\{\mathfrak f,\omega\mathfrak f\}$. This follows from the characterisation of the
collisional invariants in \cite{LukkSpohn}, see also \cite[Section 7]{GLM26review}. In particular, we see that the collision frequency $a$ vanishes at the momentum edges: as $p \to 0^+, 2\pi^-$.

The linear equation is
\begin{equation}\label{eq:linearised_eq}
 \partial_tg+v(p)\partial_xg=Lg.
\end{equation}

We use the \emph{spatial Fourier transform} 
$\widehat{h}(\xi,p)=\int_{\mathbb{R}}e^{-ix\xi}h(x,p) dx$, 
with inverse factor $(2\pi)^{-1}$. Then the linear equation for $\widehat{g} = \widehat{g}(t,\xi,p)$ is  
\begin{equation}\label{eq:linear_fourier}
 \partial_t\widehat{g} = B(\xi)\widehat{g},
 \qquad B(\xi):=L-i\xi v.
\end{equation}

If we denote the physical-space semigroup by $S(t)$, the solution
with initial data $g_0$ is $g(t) := S(t)g_0$, where 
\begin{equation}
\label{eq:physicalsemigroup_intro}
\widehat{g}(t,\xi,p) = \widehat{S(t)g_0}(\xi,p) = e^{tB(\xi)} \widehat{g}_0(\xi,p).
\end{equation}

\subsection{Main results}
To state the decay estimates, we use the following weighted $L^\infty_pL^r_\xi$ norms, with the $p$-norm taken first: 
$$
 X_{\mu,r}[h]
 :=\left\|\|\omega^\mu\widehat h(\xi)\|_{L^\infty_p}
       \right\|_{L^r_\xi},
 \quad \text{ for }\quad  r=1,2 \text{ and } \ \mu \in \{-1/2, 1/6, 1/2\}.  
$$
We also denote by 
$\|h\|_{\mathcal{X}} := X_{1/6,1}[h] + X_{1/6,2}[h]$
and $\mathscr{Y} := L^2(a^{-1} dp)$. 

We write $\langle t \rangle = 10+t$, and denote by $\Pi$ the
$L^2_p$-orthogonal projection onto $\ker L$.

We first describe in detail the low spatial-frequency spectrum of the operator $B(\xi)$, which is the content of Section \ref{sec: lowfrequency_spectrum}. 

\begin{theorem}[Low-Frequency Spectrum] \label{theo: low_freq_spectrum_intro} 
There exist $\xi_*,\vartheta_*>0$ such that, for  $0<|\xi|<\xi_*$, the spectrum of $B(\xi)$ in the closed disk
$|z|\leq \vartheta_*|\xi|$ consists of exactly two algebraically
simple, real, strictly negative eigenvalues.  They satisfy
$$
 z_{\rm fr}(\xi)=-c_{\rm fr}|\xi|^{8/5}+ \smallO(|\xi|^{8/5}),
 \qquad
 z_{\rm df}(\xi)=-c_{\rm df}\xi^2+ \smallO(\xi^2),
$$
where $c_{\rm fr},c_{\rm df}>0$. 
\end{theorem}

The subscripts ${\rm fr}$ and ${\rm df}$ refer to these fractional and diffusive orders. They reflect the different
edge-behaviour, as $p\to 0^+, 2\pi^-$, of the two kernel components: one, $\mathfrak{f}$,  remains lower bounded at the edges, while the other, $\omega\mathfrak{f}$ vanishes like $\omega$. Figure 1 illustrates these two spectral scales of the above \ref{fig:lowfrequency_spectrum}. 

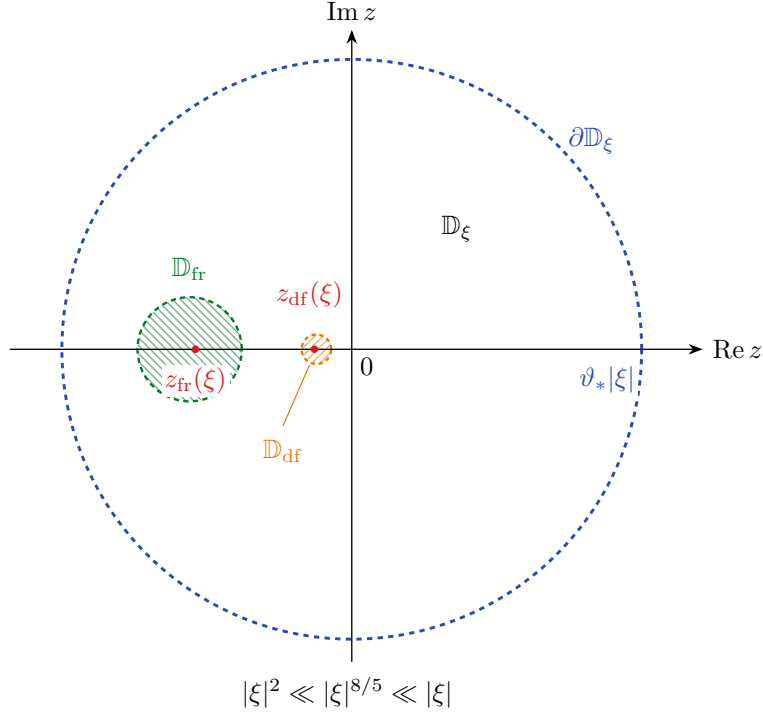
\begin{figure}[htbp]
    \centering
\begingroup
\definecolor{outerorange}{RGB}{235,119,0}
\definecolor{fractiongreen}{RGB}{0,123,42}
\definecolor{diffusionblue}{RGB}{42,73,177}
\definecolor{eigenred}{RGB}{221,25,25}
\begin{tikzpicture}[x=1.3cm,y=1.3cm,>=Stealth,font=\small]
  \path[use as bounding box] (-3.65,-3.62) rectangle (4.55,3.55);
  \path[pattern=north west lines,pattern color=fractiongreen!65]
    (-1.65,0) circle[radius=.53];
  \path[pattern=north east lines,pattern color=outerorange!65]
    (-.36,0) circle[radius=.15];
  \draw[diffusionblue,line width=1.1pt,dash pattern=on 2pt off 2pt]
    (0,0) circle[radius=2.95];
  \draw[fractiongreen,line width=.9pt,dash pattern=on 2pt off 1.6pt]
    (-1.65,0) circle[radius=.53];
  \draw[outerorange,line width=.9pt,dash pattern=on 1.6pt off 1.4pt]
    (-.36,0) circle[radius=.15];
  \draw[->,line width=.5pt] (-3.48,0)--(3.57,0)
    node[right] {$\operatorname{Re}z$};
  \draw[->,line width=.5pt] (0,-3.18)--(0,3.25)
    node[above] {$\operatorname{Im}z$};
  \node[below right,inner sep=3pt] at (0,0) {$0$};
  \node[diffusionblue,anchor=west] at (2.10,2.12)
    {$\partial \mathbb{D}_{\xi}$};
  \node[diffusionblue,fill=white,inner sep=1.2pt,anchor=north]
    at (2.60,-.13) {$\vartheta_*|\xi|$};
  \draw[diffusionblue,line width=.6pt] (2.95,-.055)--(2.95,.055);
  \node at (1.07,1.23) {$\mathbb{D}_{\xi}$};
  \node[fractiongreen] at (-1.65,.82) {$\mathbb{D}_{\rm fr}$};
  \draw[outerorange,line width=.45pt] (-.43,-.16)--(-.70,-.78);
  \node[outerorange,anchor=north] at (-.70,-.82)
    {$\mathbb{D}_{\rm df}$};
  \fill[eigenred] (-1.59,0) circle[radius=1.4pt];
  \fill[eigenred] (-.38,0) circle[radius=1.4pt];
  \node[eigenred,fill=white,inner sep=1pt,anchor=north]
    at (-1.59,-.15) {$z_{\rm fr}(\xi)$};
  \node[eigenred,anchor=south] at (-.43,.32)
    {$z_{\rm df}(\xi)$};
  \node[align=center] at (0,-3.49)
    {$|\xi|^2\ll |\xi|^{8/5}\ll |\xi|$
};
\end{tikzpicture}
\caption{Low-frequency spectrum for sufficiently small $0<|\xi|$. Each hatched disk contains exactly one simple real eigenvalue. The two disks are disjoint and contained in the blue  central disk 
$\mathbb{D}_{\xi} =\{|z|<\vartheta_*|\xi|\}$, which contains no other spectrum.} 
    \label{fig:lowfrequency_spectrum}
\endgroup
\end{figure}

Having described the spectrum  near the origin for sufficiently small $\xi$, we continue, in Section \ref{sec:linear} by providing decay estimates for the linear semigroup $S(t)$. We summarise our findings in the theorem below: 

\begin{theorem}[Linear decay]\label{thm:linear_intro}
Let $\delta>0$ be small enough, and assume that
$\Pi g_0(x,\cdot)=0$ for almost every $x$. 
For $r=1,2$ and all $t \geq 0$, whenever the right-hand sides
are finite, we have
\begin{align*}
 X_{1/6,r}[S(t)g_0] &\lesssim
 \langle t \rangle^{-2/5}
 \left( X_{-1/2,r}[g_0]
 +\|\widehat{g_0}\|_{L^\infty_\xi \mathscr{Y}}
 \right),\quad \text{ and } \\
 X_{1/2,r}[S(t)g_0]
 &\lesssim
 \langle t\rangle^{-3/5+\delta}
 \left(
 X_{-1/2,r}[g_0]
 +\|\widehat{g_0} \|_{L^\infty_\xi\mathscr{Y}}
 \right).
\end{align*}
Moreover,
$$ \|S(t)g_0\|_{L^2_{x,p}} \lesssim_\delta
 \langle t\rangle^{-3/5+\delta}
 \left( X_{-1/2,2}[g_0]
 +\|\widehat{g}_0\|_{L^\infty_\xi\mathscr{Y}}
 \right).
$$
\end{theorem}

By Fourier inversion, we may transfer back to the physical space, where the estimates with $r=1$ also give decay in the weighted $L^\infty_{x,p}$ norms, with rates  $\langle t\rangle^{-2/5}$ and
$\langle t\rangle^{-3/5+\delta}$ for the weights
$\omega^{1/6}$ and $\omega^{1/2}$, respectively.
\\


Finally in Section \ref{sec:nonlinear} we pass perturbatively to the nonlinear problem. In particular, we control the nonlinearities through the careful definition of a new norm that takes into account time, space and momentum. 

\begin{theorem}[Nonlinear stability Theorem]\label{thm:nonlinear_intro}
Let $0<\delta<1/10$. There exists $\varepsilon_\delta>0$
such that, if $ \Pi g_0(x,\cdot)= 0$ for almost every $x$, and
$$\varepsilon:= X_{-1/2,1}[g_0]+X_{-1/2,2}[g_0]
+\|\widehat{g}_0\|_{L^\infty_\xi\mathscr{Y}} 
    \leq \varepsilon_\delta,
$$ 
then the nonlinear equation \eqref{eq:nonlinear_eq} admits a unique global mild solution
$ g\in C([0,\infty);\mathcal{X} )$.  In particular, 
for $r=1,2$ and all $t\geq 0$, it decays as: 
\begin{align*}  
X_{1/6,r}[g(t)]
\lesssim_\delta \varepsilon\langle t \rangle^{-2/5},\qquad
X_{1/2,r}[g(t)]  \lesssim_\delta
 \varepsilon\langle t\rangle^{-3/5+\delta}.
\end{align*}
Moreover, in the physical variables we have 
$$
\|g(t)\|_{L^2_{x,p}}  \lesssim_\delta
\varepsilon\langle t \rangle^{-3/5+\delta}. 
$$
\end{theorem}

Note that by Fourier inversion, the estimates with $r=1$ also yield the weighted pointwise decay in the physical variables: 
$$ \|\omega^{1/6}g(t)\|_{L^\infty_{x,p}}
\lesssim_\delta\varepsilon\langle t\rangle^{-2/5},
\qquad \|\omega^{1/2}g(t)\|_{L^\infty_{x,p}}  \lesssim_\delta \varepsilon \langle t\rangle^{-3/5+\delta}. 
$$

We stress here that in both linear and nonlinear theorems, we impose the orthogonality assumption $\Pi g_0(x,\cdot)=0$, for almost every spatial point $x$.  This orthogonality is imposed only on the initial
data and it is generally not preserved by the linear flow, due to the effect of the transport term. This is in contrast with the spatially homogeneous problem, where the orthogonality is preserved and it is a substantial ingredient in the global well-posedness proof in  \cite{GLM26}.

\subsection{Difficulties and strategy of the proof} 

One of the main difficulties is the degeneracy of the collision frequency $a$ at the momentum edges: $a(p)\sim \omega(p)^{5/3}$
vanishes as $p\to0^+,2\pi^-$.  As a consequence, the collision operator does not possess a spectral gap in the flat $L^2_p$, and instead, a weighted Poincar\'e inequality controls the microscopic component in the weighted geometry $L^2_a$.

This feature is already present in the homogeneous kinetic FPUT model  \cite{GLM26}, corresponding to $\xi=0$, where the essential spectrum touches the origin. The degeneracy at the edges plays a central role in the algebraic relaxation. A similar degeneracy of the collision frequency occurs at
large velocities for the classical Boltzmann equation with
soft potentials in near-equilibrium regimes \cite{SG2006}. Here, however, the momentum variable belongs to the torus, and the collision frequency vanishes at the
momentum edges.

The second major difficulty is that orthogonality to $\ker L$ in
$L^2_p$ is generally not preserved when $\xi\neq0$, due to the
transport term. Preservation of this orthogonality is used essentially in the proof of Rayleigh-Jeans stability in the homogeneous setting \cite{GLM26}.  The challenge here is therefore to deal
with transport and the degeneracy of the collision frequency simultaneously.

A further feature of the inhomogeneous problem is the coexistence of two distinct relaxation scales: the two-dimensional $\ker L$ give rise, at low spatial frequencies,  
to a fractional mode and a diffusive mode. We therefore need estimates that account for both effects, with sufficient control of their dependence on the spatial frequency to obtain 
decay after Fourier inversion.  
The resulting decay must  also be strong enough to close the nonlinear stability argument. 


To overcome these difficulties, we develop an approach combining a precise analysis of the low-frequency spectrum with energy estimates adapted to the degeneracy of the collision operator.

\noindent
\textbf{(i) For the low-frequency spectrum.}
To determine the spectral points for small $|\xi|$, we first work with the new variable $w=a^{1/2}u$, rewrite the eigenvalue
problem for $B(\xi)$ in terms of $w$ and we decompose the space 
into $\mathscr{N} = a^{1/2}\ker L$ and its orthogonal complement. Through a Schur complement argument we then reduce the original problem to a two-by-two matrix
equation. A careful asymptotic analysis of its entries reveals the two different scales appearing in Theorem \ref{theo: low_freq_spectrum_intro}:  $|\xi|^{8/5}$ and $\xi^2$, around which we locate the two eigenvalues by Rouch\'e's theorem.

\noindent
\textbf{(ii) For the linear decay and nonlinear stability.}  
Since we cannot rely on preservation of the orthogonality condition along the flow, at \emph{small spatial frequencies} we 
separate the two eigenmodes and work with the remaining spectral subspace, which is invariant under the flow. On this subspace, we are able to control the kernel component  by the microscopic dissipation. In a way, this replaces the use of preserved orthogonality in the homogeneous case. We estimate the two eigenmodes
separately, using their eigenvalue asymptotics and the smallness of their spectral projections on initial data orthogonal to $\ker L$.    
Regarding \emph{spatial frequencies bounded away from zero}, we obtain control of the kernel component by adding a bounded quadratic
correction to the energy. This construction follows the general idea of modifying the energy as usual in hypocoercivity theory  see
\cite{GUO2010, DMS15, BDMMS20, villani09hypocoercivity, VillaniICM,  CDHMMS26, EvansMenegaki25, BDL22fractional} and the references therein.   

In both frequency regimes, we split the momentum domain into its bulk $\mathcal{B}$ and its edges $\mathcal{E}$. The edges are thick and the thickness decreases with time. On the bulk we exploit  the lower bound on $a$, while 
on the edges we use a proper weighted norm. This eventually gives $L^2_p$ decay on the whole momentum domain.  
Together with a smoothing estimate for $K$ we obtain $L^c_p (\mathcal{B})$-decay with $c>2$, on a fixed bulk $\mathcal{B}$. A further domain splitting and estimates on $K$ allow to control the contributions from the edges, and eventually to get the weighted $L^\infty_p$ decay on the whole momentum domain. To pass back to the physical variable $x$, we apply Fourier inversion once sufficient integrability in $\xi$ has been established.  

For the nonlinear stability, we combine the linear estimates with weighted bounds on the quadratic and cubic nonlinearities. The cancellation identities of the nonlinearities allow us to apply the linear decay estimates to the nonlinear source in
Duhamel's formula. We then close the argument with a small-data
bootstrap argument. 

\subsection{Relation to previous results} 
The phonon Boltzmann equation is a kinetic description of weakly anharmonic lattice dynamics. Its solution describes the phonon distribution and predicts the average distribution of energy among the normal modes. Extensive reviews of the derivation and of the mathematical kinetic theory of such oscillator chains can be found in   \cite{Spohnreview06}, \cite{Lukkarinen2016}, and a more recent one in \cite{GLM26review}. 
Rigorous short-time kinetic expansions have recently been obtained in \cite{VassilevWu2026}. These results concern a spatially homogeneous kinetic description and times much shorter than the
kinetic time scale.   

Regarding the spatially homogeneous kinetic FPUT equation, global well-posedness for initial data close to non-singular Rayleigh-Jeans equilibria was provided in  \cite{GLM26} in weighted $L^\infty$ spaces, implying the stability of such equilibria. For the inhomogeneous kinetic FPUT equation, long-time existence for small data near vacuum was proved in 
\cite{Xiang}.

Regarding the classification of the collisional invariants of phonon Boltzmann equations, Spohn in  \cite{Spohn06}  characterised them, in $d\geq 2$ and under general assumptions
on the dispersion relation, as linear combinations of $1$ and $\omega$.  
The corresponding classification for the one-dimensional FPUT
equation was established in \cite{LukkSpohn}. 
Entropy maximisers under fixed mass and energy constraints, for a general class of kinetic wave equations/phonon Boltzmann equations on tori arising from microscopic lattice dynamics, have been classified in \cite{EGLM25}.  

Substantial progress has been made in the recent years for related wave turbulence equations arising from continuous wave systems. In such cases, the momentum domain is the whole $\mathbb{R}^d$, as opposed to the torus here (which arises as the dual of the lattice). 
In particular, when starting from cubic NLS, in dimensions $d\geq3$, the spatially homogeneous wave kinetic equation has been rigorously derived for arbitrarily long kinetic times, for as long as  sufficiently regular solutions of the kinetic equation exist, see \cite{DengHani1/2, DengHani3}
and references therein. For quadratic dispersive models, a   derivation in the spatially
inhomogeneous setting was obtained in \cite{ACG2025}, up to an arbitrarily small polynomial loss of the kinetic time scale.   
Global-in-time results near vacuum are available for the inhomogeneous NLS wave kinetic equation in \cite{AmpMillPavlovTaskov,Ioakeim22,AmpatzoglouLeger25}, while 
for the homogeneous equation in the 3D  isotropic setting, global well-posedness for data near a nonsingular RJ 
equilibrium, together with its exponential stability, was proved recently in \cite{escobedo2026global}. The linearised collision operator in this nonsingular setting in fact possesses a spectral gap, leading thus to exponential relaxation, 
even though the respective $K$ integral operator is not compact, as opposed to the setting studied here.  For singular RJ equilibria, finite-time
condensation for suitable initial data is proved 
in \cite{EscMenegaki26}.

\noindent
\textbf{Thermal conductivity and fractional diffusion.} 
A standard microscopic setting consists of a chain of oscillators of length $N$ coupled to thermal reservoirs at different temperatures at the boundaries.
Numerically, it is observed that in such models, in the nonequilibrium stationary state with the FPUT-$\beta$ interactions, the thermal conductivity diverges as
$N^{2/5}$, see \cite{AokiKuzn01,LLP05}. The dependence on the model and the different predictions
for transport exponents are discussed in the recent review
\cite{LLP2026}. 
For broader accounts of anomalous heat transport in such 
low-dimensional systems, including applications to 
condensed matter and nanosystems, we refer to  
\cite{BLL2020,BDLL2023}.

An alternative equilibrium approach to conductivity via linear response theory and the Green-Kubo formula relates conductivity to the time integral of the energy current correlation function.  
In the kinetic description this correlation is conjectured to be expressed through the linearised collision semigroup, see \cite[Section 12]{GLM26review} for an extensive formal discussion. In  \cite{LukkSpohn}, Lukkarinen and Spohn proved that the kinetic current correlation decays like $t^{-3/5}$. Integrating this decay up to a transit time proportional to $N$ gives the conductivity prediction $N^{2/5}$. 

Finally we refer to the rigorous fractional diffusion limit for the linearised
inhomogeneous phonon Boltzmann equation for the FPUT model, which was established in  \cite{MelletMerino}. It concerns linearisation  around the singular thermal equilibrium $ \bar{T}/\omega$, with $\bar{T}>0$, and the
derivation of a fractional diffusion equation for the macroscopic temperature perturbation, of the form 
$\partial_t T+c(-\Delta)^{4/5} T=0$. The corresponding limit around nonsingular equilibria is
also discussed formally in
\cite[Section 10.2]{GLM26review}.  Here we focus on nonsingular equilibria (with $\gamma>0$) and we identify both fractional and diffusive eigenvalue scales. 
The fractional eigenvalue scale $|\xi|^{8/5}$ has the same order as the Fourier symbol of $(-\Delta_x)^{4/5}$ appearing
in the hydrodynamic limit of \cite{MelletMerino} and we also
identify a diffusive eigenvalue scale $\xi^2$, corresponding to a longer relaxation time at small spatial frequencies.

\subsection{Notation and organisation of the article}
Throughout the paper, $A\lesssim B$ means that $A\leq CB$
for a positive constant independent of the parameters under consideration, and $A\sim B$ means that $cB \leq A \leq CB$. We write $R=\smallO(B)$ when
$|R|/B\to0$ in the limit under consideration, with $B>0$.  
A sharp asymptotic is always written explicitly, for instance
$A=cB +\smallO(B)$ or $A/B\to c$.
Constants may depend on the fixed equilibrium: $\beta, \gamma$.   

We write $L^2_a:=L^2(a dp)$. We use $\Pi$ and $\Pi_a$ to denote the 
orthogonal projections onto $\ker L$ in the flat  $L^2_p$ and the weighted 
$L^2(a dp)$, respectively. We also write
$h^\parallel = \Pi_a h$ and $h^{\perp_a}=(I-\Pi_a)h$. In the low-frequency spectral analysis in Section \ref{sec: lowfrequency_spectrum}, we call $\mathsf{P}$ the $L^2_p$-orthogonal projection onto $\mathscr{N}=a^{1/2}\ker L$, and its complement $\mathsf{Q}=I-\mathsf{P}$. These projections are related by
$\mathsf{P} (a^{1/2}h) =a^{1/2} \Pi_a h$.

Finally, we denote by $\mathbb{P}_{\rm fr}(\xi)$ and $\mathbb{P}_{\rm df}(\xi)$ the Riesz projections on the e-space of the two low-frequency eigenvalues we find of $B(\xi)$.

We first establish the low-frequency spectral description in Section \ref{sec: lowfrequency_spectrum}. The following section \ref{sec:linear} proves the linear semigroup decay estimates, and the nonlinear stability section applies them
to the full equation in Section \ref{sec:nonlinear}. 
In the Appendix we collect the collisional kernel estimates needed in these arguments, including the smoothing effect of $K$, as well as the sharp asymptotics for the collisional  frequency $a$.

\section{The low-frequency spectrum of $B$}
\label{sec: lowfrequency_spectrum}
In this section we study the spectrum of $B(\xi)=L-i\xi v$ near the origin. 
Since $K$ is compact, the essential spectrum of $B(\xi)$ is the same as that of the multiplication operator $-a-i\xi v$. 
At $\xi=0$, it touches the origin because $a(p)$ vanishes at the edges. For $\xi\neq 0$, the edge-limits of $-a(p)-i\xi v(p)$ are instead $\mp i\xi/2$.  
Lemma \ref{lem:m_small_Disk} below shows that, for sufficiently small $|\xi|>0$, the central disk
$|z|\leq\vartheta_*|\xi|$ avoids the essential spectrum, it lies in the resolvent set of $(-a(p)-i\xi v(p))$ : throughout this disk, $|a(p)+z+i\xi v(p)| \gtrsim a(p)+|\xi|$. We prove that this disk contains exactly two (algebraically simple)  eigenvalues of $B(\xi)=-a-i\xi v+K$. 

We gather the main results of this section in the following Theorem:

\begin{theorem}[The two low-frequency eigenvalues]
\label{thm:low_frequency_spectrum}
There exist $\vartheta_*>0$, $\xi_*>0$ and constants 
$c_{\rm fr},c_{\rm df}>0$ such that,
for $0<|\xi|<\xi_*$, 
the spectrum of $B(\xi)$ in $|z|\leq\vartheta_*|\xi|$
consists of exactly two simple eigenvalues.
They are real, strictly negative, and so that 
\begin{equation}
z_{\rm fr}(\xi) =-c_{\rm fr}|\xi|^{8/5}+ \smallO(|\xi|^{8/5}), \qquad
z_{\rm df}(\xi) = - c_{\rm df}\xi^2 + \smallO(\xi^2).
\end{equation}
\end{theorem}

The proof of this Theorem has three main parts:  
(i) We first show that the equation can be
solved when projected on the orthogonal complement of the Kernel $\mathscr{N} =a^{1/2}\ker L$. (ii) We then reduce the remaining eigen-problem to a spectral problem for a two-by-two matrix and locate the zeros of its determinant (two because this is the dimension of $\mathscr{N}$). (iii) We finally justify the passage back to eigenvectors of $B(\xi)$ and provide algebraic simplicity.

With this detailed description of the low-frequency spectrum, in the next section, we will be able to estimate their spectral projections which will be used to prove the linear decay. 
Related reductions appear in the fluid-mode analysis of Ellis-Pinsky and its fractional extensions: see for example 
\cite{ellispinsky1975fluid,bouin2024fractional}.

\subsection{Invertibility on the complement of the collision kernel} 
We remind that $L=-a+K$ is self-adjoint and non-positive on the flat 
$L^2_p$, with
$\ker L=\operatorname{span}\{\mathfrak{f},\omega\mathfrak{f}\}$.

We consider the operator $\mathscr{K} = a^{-1/2}Ka^{-1/2}$ which extends to a compact, self-adjoint
operator on $L^2_p$, 
see \cite[Lemmas 3 and 4]{GLM26}. 
We then define 
$$ \mathscr{N}:=a^{1/2}\ker L=\ker(I-\mathscr{K}).
$$ 
We give a short proof of the identification with $\ker(I-\mathscr{K})$ in 
Appendix \ref{app:appendix_eigenspace}. 
We denote by $\mathsf{P}$ the $L^2_p$- orthogonal projection
onto $\mathscr{N}$, and set $\mathsf{Q}=I-\mathsf{P}$.

Observe that these projections are related to the weighted projection
$\Pi_a$ by
\begin{equation} \label{eq:relation projections}
 \mathsf{P}(a^{1/2}h)=a^{1/2}\Pi_a h,
    \quad \text{ for }\ h \in L^2_p.
\end{equation}
Indeed, $a^{1/2}\Pi_a h$ belongs to $\mathscr{N}$, while
$a^{1/2}(I-\Pi_a)h$ is orthogonal to $\mathscr{N}$ in $L^2_p$.

If we omit $K$, the operator $B(\xi)$ reduces to multiplication by $-a(p)-i\xi v(p)$. Its resolvent is multiplication by 
$(a(p)+z+i\xi v(p))^{-1},$ whenever this exists.  
This is the resolvent of the transport and loss terms.  
The eigen-problem for $B(\xi)$ is
$$ (a+z+i\xi v)u=Ku.
$$ 
Using $K=a^{1/2}\mathscr{K} a^{1/2}$ and setting $w=a^{1/2}u$
it rewrites as $\mathscr{M}_\xi(z)w=0$, where
\begin{equation}
\mathscr{M}_\xi(z)
:=I-\frac{a}{a+z+i\xi v}\mathscr{K}. 
\label{eq:def_M_xi matrix}
\end{equation} 
We also write
\begin{equation}
\mathscr{M}_\xi(z)= [I - \mathscr{K}]  +
\mathfrak{m}_\xi(z)\mathscr{K}, \qquad \text{where }\
\mathfrak{m}_\xi(z,p):=\frac{z+i\xi v(p)}{a(p)+z+i\xi v(p)},  
\label{eq: M = (I-K)_+_pertub}
\end{equation}
so that $\mathscr{M}_\xi(z)$ is written as $I-\mathscr{K}$ plus a perturbation.

The next lemma gives us a disk on which these expressions are well defined and bounded. We will then use compactness of $\mathscr{K}$ to show that the microscopic $\mathsf{Q}$ block remains invertible in this disk.

\begin{lemma} 
\label{lem:m_small_Disk}
There exist $\vartheta_*>0$ and $\xi_*>0$ such that, whenever $0<|\xi|<\xi_*,$ $ |z|\leq\vartheta_*|\xi|$, 
one has
\begin{equation} \label{eq:lowerbound_freedisk}
    |a(p)+z+i\xi v(p)|
    \geq c\bigl(a(p)+|\xi|\bigr),
    \qquad \text{for a.e. }p\in\mathbb T.
\end{equation}
Therefore, 
\begin{equation}
    \sup_{0<|\xi|<\xi_*}
    \sup_{|z|\leq\vartheta_*|\xi|} 
    \left\|\frac{a}{a+z+i\xi v}\right\|_{L^\infty_p}<\infty,
    \label{eq:uniformbound_lemma1}
\end{equation}
and
\begin{equation}
    \sup_{|z|\leq\vartheta_*|\xi|}
    |\mathfrak{m}_\xi(z,p)|
    \lesssim
    \frac{|\xi|}{a(p)+|\xi|}.
    \label{eq: pointwise bound for m}
\end{equation}
In particular, for every fixed $g\in L^2_p$,
\begin{equation}
    \lim_{|\xi|\to0}
    \sup_{|z|\leq\vartheta_*|\xi|}
    \|\mathfrak{m}_\xi(z)g\|_{L^2_p} =0.
    \label{eq:z_uniform_strong conv of m}
\end{equation}
\end{lemma}

\begin{proof}
We choose disjoint endpoint neighborhoods on which $|v(p)|\geq v_*>0$, for the FPUT dispersion relation one may take any
$v_*< 1/2$ after shrinking them.  On their complement there is
$a_*>0$ such that $a(p)\geq a_*$.  Fix
$ 0<\vartheta_*<\frac{v_*}{4}$ 
and write $z=s+i\tau$. Since $|a+ z+ i\xi v| = \sqrt{(a+s)^2+ (\tau + \xi v)^2 }$, either the real or the imaginary part can provide the required lower bound \eqref{eq:lowerbound_freedisk}. 

On an endpoint neighborhood, first suppose that
$a(p)\geq2\vartheta_*|\xi|$.  Then the real part gives 
\[|a(p)+s| \geq a(p)-|s|
    \geq\frac{1}{2}a(p), 
\]
and thus $|a+ z+ i\xi v| \geq \left( \frac{\vartheta_*}{2\vartheta_* + 1}\right) (a+|\xi| )$. 
If instead $a(p)<2\vartheta_*|\xi|$, then the imaginary part gives 
\[ |\tau+\xi v(p)|
    \geq |\xi||v(p)|-|\tau|
    \geq(v_*-\vartheta_*)|\xi|, 
\]
and thus $|a+ z+ i\xi v| \geq \left( \frac{v_* - \vartheta_*}{2\vartheta_* + 1}\right) (a+|\xi|)$.  
On the complement of the endpoint neighborhoods, $a \geq a_*>0$ and 
$|a(p)+s|\geq a(p)/2$ for all sufficiently small $|\xi|$, i.e. $\vartheta_* \xi_* < \alpha_*/2, \xi_*< a_*$. Thus $|a+ z+ i\xi v| \geq \frac{1}{4} (a+|\xi| )$.  
Take the minimum over the three constants to prove \eqref{eq:lowerbound_freedisk} with a constant independent of $\xi_*$.  In particular, the
closed disk is disjoint from the singular curve of the free
multiplier.  The bounds \eqref{eq:uniformbound_lemma1} and
\eqref{eq: pointwise bound for m} follow immediately.  Since
$a(p)>0$ almost everywhere, the right-hand side of
\eqref{eq: pointwise bound for m} converges pointwise to zero and is
uniformly bounded.  Dominated convergence proves
\eqref{eq:z_uniform_strong conv of m}.
This is strong convergence, not operator-norm convergence: near the edges $a(p) \ll |\xi|$, the multiplier $\mathfrak{m}_\xi(z)$ remains of order one.
\end{proof}

The preceding lemma shows that  $\mathfrak{m}_\xi(z)$ converges strongly to zero, uniformly in the central disk.  
Although $\mathfrak{m}_\xi(z)$ is not necessarily small in operator norm, 
its composition with the compact operator $\mathscr{K}$ is.
With this, we may conclude microscopic invertibility, as shown in the lemma below.   

\begin{lemma}[Uniform invertibility of the microscopic block]
\label{lem:W_micro_invertible}
Let $\vartheta_*$ and $\xi_*$ be as in
Lemma~\ref{lem:m_small_Disk}.  On $ 0<|\xi|<\xi_*$ and $|z|\leq\vartheta_*|\xi|$ and for $\xi_*$ sufficiently small, the microscopic block 
\begin{equation}
 \mathscr{W}_\xi(z) :=  \mathsf{Q}\mathscr{M}_\xi(z)\mathsf{Q}
\end{equation}
is invertible on $\mathsf{Q}L^2_p$.  
Moreover,
\begin{equation} \label{eq:W^-1_uniform bound} 
    \sup_{0<|\xi|<\xi_*}
    \sup_{|z|\leq\vartheta_*|\xi|}
    \|\mathscr{W}_\xi(z)^{-1}\|_{2\to2}<\infty.
\end{equation}
\end{lemma}

\begin{proof}
\underline{Invertibility for $\xi=z=0$.} 
Since $\mathscr{K}$ is compact and self-adjoint and its eigenspace at
one is exactly $\mathscr{N}$,  
$$ d_*:=\operatorname{dist}
\big( 1,\sigma(\mathscr{K})\setminus\{1\} \big)
$$ 
is strictly positive: every nonzero spectral point of a compact
self-adjoint operator is isolated, and the only possible accumulation
point of its spectrum is zero, by the spectral theorem for compact self-adjoint operators. The spectral projection $\mathsf{P}$
commutes with $\mathscr{K}$, so $\mathsf{Q}L^2_p$ reduces $\mathscr{K}$
and
\begin{equation} 
\label{eq:gap_for_micro: Q(I-K)Q}
\|\mathsf{Q}(I-\mathscr{K})\mathsf{Q}g\|_{L^2_p}
    \geq d_*\|\mathsf{Q}g\|_{L^2_p}.
\end{equation} 
Moreover, $I-\mathscr{K}$ is nonnegative: Indeed, for every
$h\in L^2_p$, the identity $K=a^{1/2}\mathscr{K} a^{1/2}$ gives
\begin{equation} \label{eq:positive_micro_block}
\big \langle (I-\mathscr{K}) a^{1/2}h, a^{1/2}h \big \rangle_{L^2_p}
= - \langle L h, h \rangle_{L^2_p} \geq 0.
\end{equation} 
Since $a>0$ almost everywhere, multiplication by $a^{1/2}$
has dense range in $L^2_p$. Boundedness of $[I-\mathscr{K}]$
then extends this positivity to all of $L^2_p$.
 Thus, the self-adjoint operator
$$  \mathscr{A}:=\mathsf{Q}(I-\mathscr{K})\mathsf{Q} \quad\text{ on }\ \mathsf{Q}L^2_p
$$
satisfies $\mathscr{A} \geq d_* I_{\mathsf{Q}}$, where $I_{\mathsf{Q}}$ is the identity on 
$\mathsf{Q} L^2_p$.  
 In particular, it is invertible on $\mathsf{Q}L^2_p$, with 
$ \|\mathscr{A}^{-1}\|_{\mathsf{Q}L^2_p\to\mathsf{Q}L^2_p}
    \leq d_*^{-1}.$

\underline{Smallness of the perturbation.} We next use compactness of $\mathscr{K}$ to get convergence in operator-norm after composition with $\mathscr{K}$. For a fixed $\delta>0$, since $\mathscr{K}$ is compact, there exists a
finite-rank operator $\mathscr{K}_\delta$ such that
$$ \|\mathscr{K}-\mathscr{K}_\delta\|_{2\to2}<\delta.
$$ 
By \eqref{eq:uniformbound_lemma1} in Lemma \ref{lem:m_small_Disk}, the family $\mathfrak{m}_\xi(z)$ is
uniformly bounded for $  0<|\xi|<\xi_*,
$ and $|z|\leq\vartheta_*|\xi|$.   
Moreover, since $\operatorname{Ran}\mathscr{K}_\delta$ is
finite-dimensional, from the uniform strong convergence
\eqref{eq:z_uniform_strong conv of m} we get 
$$ \lim_{|\xi|\to 0} \sup_{|z|\leq\vartheta_*|\xi|}
 \|\mathfrak{m}_\xi(z) \mathscr{K}_\delta\|_{2 \to2} = 0.
$$ 
Therefore altogether, 
\begin{align} 
\sup_{|z|\leq\vartheta_*|\xi|}
\|\mathfrak{m}_\xi(z)\mathscr{K}\|_{2 \to 2} \leq 
C \|\mathscr{K} -\mathscr{K}_\delta\|_{2 \to 2} +
\sup_{|z|\leq\vartheta_*|\xi|}
\|\mathfrak{m}_\xi(z)\mathscr{K}_\delta\|_{2\to2}.
\end{align} 
Taking first $|\xi|\to 0$ and then $\delta\to 0$, we obtain
\begin{equation} 
\label{eq:strong converg_compactK}
\lim_{|\xi|\to 0} \sup_{|z|\leq \vartheta_*|\xi|}
 \|\mathfrak{m}_\xi(z)\mathscr{K}\|_{2 \to 2}= 0.
\end{equation}
 
\underline{Inversion by Neumann series.} Using the decomposition  \eqref{eq: M = (I-K)_+_pertub}, we  write
$ \mathscr{W}_\xi(z) = \mathscr{A} +
\mathsf{Q}\mathfrak{m}_\xi(z)\mathscr{K}\mathsf{Q}.
$ 
By the preceding estimates, for $\xi_*$ sufficiently small we have 
$$ 
\sup_{|z|\leq\vartheta_*|\xi|} \left\|
\mathscr{A}^{-1} \mathsf{Q}\mathfrak{m}_\xi(z)\mathscr{K}\mathsf{Q} \right\|
\leq \frac{1}{2}. 
$$ 
Therefore by Neumann series we have 
\begin{align*}
\mathscr{W}_\xi(z)^{-1} =  \sum_{n=0}^{\infty}
\left[ -\mathscr{A}^{-1}
\mathsf{Q}\mathfrak{m}_\xi(z)\mathscr{K}\mathsf{Q}\right]^n
 \mathscr{A}^{-1},
\end{align*} 
and so $\|\mathscr{W}_\xi(z)^{-1}\|_{2 \to2} \leq 2 d_*^{-1}.$ which proves the claim \eqref{eq:W^-1_uniform bound}.
\end{proof}

\subsection{Reduction of spectral problem  to a two-by-two matrix} 
We now eliminate the microscopic component using the inverse constructed above. This reduces the original spectral problem to an exact equation only on the two-dimensional space $\mathscr{N}$.  
Recall that
$$ \mathscr{M}_\xi(z) = [I-\mathscr{K}] + \mathfrak{m}_\xi(z)\mathscr{K}.
$$

If we  decompose $w = \mathsf{P} w + \mathsf{Q} w= :\zeta + q$, the equation $\mathscr{M}_\xi(z)w=f$, on $L^2_p = \mathsf{P}L^2_p \oplus \mathsf{Q}L^2_p$,  is equivalent to
\begin{align} \label{eq:eigenproblem_decomposed}
 q  = \mathscr{W}_\xi(z)^{-1}
  \big(\mathsf{Q}f-\mathsf{Q}\mathscr{M}_\xi(z)\zeta \big),\quad \text{ and } \quad \mathscr{F}_\xi(z)\zeta = \mathsf{P}f-
   \mathsf{P}
   \mathscr{M}_\xi(z)\mathsf{Q}
 \mathscr{W}_\xi(z)^{-1}\mathsf{Q}f, 
\end{align} 
where the operator on $\mathscr{N}$ is
\begin{equation}
\mathscr F_\xi(z):=
\mathsf P\mathscr M_\xi(z)\mathsf P
-\mathsf P\mathscr M_\xi(z)\mathsf Q\mathscr W_\xi(z)^{-1}
 \mathsf Q\mathscr M_\xi(z)\mathsf P.
\label{eq:Feshbach}
\end{equation}
This is the Schur complement, also called the Feshbach operator, 
see for example, \cite[Theorem 1.2]{dusson2021feshbach}. For the related formulation in terms of Grushin problems, see \cite[Appendix~D]{zworski2012semiclassical}.

In total the following equivalencies hold: 

\begin{equation}\label{eq:equiv_invertibility}
\begin{split}
(z-B(\xi))\text{ is invertible on }L^2_p 
\quad &\Longleftrightarrow\quad \mathscr{M}_\xi(z) \text{ is invertible on }L^2_p
\quad \\
& \hspace{0cm} \Longleftrightarrow\quad
\mathscr F_\xi(z)\text{ is invertible on }\mathscr{N}.
\end{split}
\end{equation}
We can therefore determine the spectrum of $B(\xi)$ in the central disk by
studying $\mathscr{F}_\xi(z)$ on the two-dimensional space $\mathscr{N} = a^{1/2}\ker L$.  

To see \eqref{eq:equiv_invertibility}:  
The two equations in \eqref{eq:eigenproblem_decomposed} show the second equivalence.  
Indeed, if  $\mathscr{F}_\xi(z)$ is invertible, 
$\zeta = \mathsf{P}\zeta$ is uniquely determined for each $f \in L^2_p$. The first formula determines $q=\mathsf{Q} q$, thus $w$ is uniquely determined. This means that $\mathscr{M}_\xi(z)$ is invertible.  
If $\mathscr{M}_\xi(z)$ is invertible, then $w$ is uniquely determined for all $f$. If $f \in \mathsf{P}L^2_p$, then $\mathsf{Q} f =0$, so the reduced equation is $\mathscr{F}_\xi(z) \zeta =f$. From the equation for $q$ we also get that $\zeta$ is uniquely determined and thus altogether imply that   $\mathscr{F}_\xi(z)$ is invertible.

In order to justify the first equivalence in \eqref{eq:equiv_invertibility}:  we use the identity $(I-AB)^{-1}=I+A(I-BA)^{-1}B$, applied to
the (bounded) factors $R(z)a^{1/2} \mathscr{K}$ and $a^{1/2}$, for 
$R(z)=(a+z+i\xi v)^{-1}$. This gives 
\begin{equation}
 (z-B(\xi))^{-1} = R(z)+R(z)a^{1/2} \mathscr{K} \mathscr{M}_\xi(z)^{-1} a^{1/2} R(z).
    \label{eq:bounded_physical_resolvent}
\end{equation} 
This proves that indeed  $z-B(\xi)$ is invertible if and only if
$\mathscr{M}_\xi(z)$ is invertible.

\subsection{Choice of basis and the matrix equation.} 
First note that since $\mathsf{P}$ is the spectral projection of $\mathscr{K}$ associated with the eigenvalue one, 
$ \mathscr{K}\mathsf{P} =\mathsf{P}\mathscr{K}
    =  \mathsf{P}, $ and so 
$ (I-\mathscr{K})\mathsf{P}=\mathsf{P}(I-\mathscr{K})=0$.  
It follows that
 $\mathsf{P}\mathscr{M}_\xi(z)\mathsf{P}
    = \mathsf{P} \mathfrak{m}_\xi(z)\mathsf{P}$,  
 $\mathsf{P}\mathscr{M}_\xi(z)\mathsf{Q} =
 \mathsf{P} \mathfrak{m}_\xi(z)\mathscr{K}\mathsf{Q}$, and 
 $\mathsf{Q}\mathscr{M}_\xi(z)\mathsf{P}
 = \mathsf{Q} \mathfrak{m}_\xi(z)\mathsf{P}$.  Thus the Schur complement \eqref{eq:Feshbach} simplifies into 
\begin{equation}
\mathscr{F}_\xi(z) =\mathsf{P} \mathfrak{m}_\xi(z)\mathsf{P}- \mathsf{P} \mathfrak{m}_\xi(z)\mathscr{K}\mathsf{Q}
\mathscr{W}_\xi(z)^{-1}\mathsf{Q}  \mathfrak{m}_\xi(z)\mathsf{P}.
\end{equation}

We will now choose a basis of $\mathscr{N}$ and write the reduced equation $\mathscr{F}_\xi (z) \zeta= 0$ as a system for the two coefficients of $\zeta$ in this basis.

Recall that $\ker L=\operatorname{span}\{\mathfrak{f},\omega\mathfrak{f}\}$. 
One of these functions, $\omega\mathfrak{f}$, vanishes at the momentum edges, while the other does not. To preserve this distinction in an
orthonormal basis, we set 
\begin{align*}
\psi_{\rm df}:= \frac{\omega\mathfrak{f}}{\|\omega\mathfrak{f}\|_{L^2_p}}, \qquad 
 \psi_{\rm fr}:= \frac{\mathfrak{f}-\langle \mathfrak{f}, \psi_{\rm df}\rangle \psi_{\rm df}}{
\|\mathfrak{f} - \langle\mathfrak{f},\psi_{\rm df}\rangle \psi_{\rm df}\|_{L^2_p}}.
\end{align*}

The subscripts ${\rm fr}$ and ${\rm df}$ here express the \emph{fractional} and \emph{diffusive} eigenvalues that we will find and locate precisely later in this section. They have leading orders $|\xi|^{8/5}$ and $\xi^2$, respectively, thus the naming.

The resulting basis $\{\psi_{\rm fr},\psi_{\rm df}\}$ is real and invariant under reflection $p\mapsto 2\pi- p$. Moreover, $\psi_{\rm fr}(0)>0$, and 
\begin{equation} 
\psi_{\rm fr}(p)\sim 1, \qquad \psi_{\rm df}(p)\sim\omega(p) 
\end{equation} at both edges $p\to 0,2\pi$. To pass from $\ker L$ to the weighted space $\mathscr{N}:= a^{1/2}\ker L $, 
define  $J:\mathbb{C}^2 \longrightarrow \mathscr{N}$ by 
$$  J(c_1,c_2) :=c_1 e_{\rm fr} + c_2 e_{\rm df},\quad \text{ where }\quad  e_{\rm fr}:=a^{1/2}\psi_{\rm fr},\ e_{\rm df}:=a^{1/2}\psi_{\rm df}. $$  

The vectors $e_{\rm fr}, e_{\rm df}$ form now a basis of 
$\mathscr{N}$, although they need not be orthonormal. Now define the two-by-two pairing matrix of the Feshbach operator 
\begin{align} \label{eq:Feshbachmatrix_def}
    \mathfrak{F}_\xi(z)
    &:=J^*\mathscr{F}_\xi(z)J, \qquad
    \bigl[\mathfrak{F}_\xi(z)\big]_{ij}
    := \left\langle
        \mathscr{F}_\xi(z)e_j,e_i
    \right\rangle_{L^2_p},
    \quad i,j\in\{\mathrm{fr},\mathrm{df}\}.
\end{align}

More explicitly, the Feshbach formula gives 
\begin{equation} \label{eq:Feshback_pairing matrix_entries}
\big[\mathfrak{F}_\xi(z)\big]_{ij}
= \langle\mathfrak{m}_\xi(z)e_j,e_i \rangle -
\left\langle \mathfrak{m}_\xi(z)\mathscr{K}\mathsf{Q}
 \mathscr{W}_\xi(z)^{-1}
\mathsf{Q}\mathfrak{m}_\xi(z)e_j, e_i
    \right\rangle, 
\end{equation}
where the second term accounts for the microscopic component. 
Since $J$ and $J^*$ are invertible between the two-dimensional
spaces $\mathbb C^2$
and $\mathscr{N}$, writing $h=J c$ gives
$$ \mathscr{F}_\xi(z)h=0 \quad\Longleftrightarrow\quad
  \mathfrak{F}_\xi(z) c = 0.
$$
Thus, within the central disk, we have that $z \in \sigma(B(\xi)) $ if and only if $\det\mathfrak{F}_\xi(z)=0$.  
Whenever the inverses exist, we will also use the identity 
$ 
\mathscr{F}_\xi(z)^{-1} = J \mathfrak{F}_\xi(z)^{-1} J^*$.

\subsection{Estimate of the entries in \eqref{eq:Feshback_pairing matrix_entries} as $|\xi|\to 0$}    
To estimate the matrix entries, we need the precise behavior
of $a(p)$ near the endpoints. The asymptotics $a(p)\sim\omega(p)^{5/3}$ is established in \cite{LukkSpohn} and 
\cite[Lemma 18]{GLM26}.  
In order to determine the leading coefficient of the fractional branch,
we use the sharper limit from
Lemma \ref{lem:sharp_a_limit_appendix} which proof is given in the Appendix \ref{app:sharp_collision_frequency}, stating that 
\begin{equation}
\begin{split}
&a(p)=a_0\ \omega(p)^{5/3}(1+ \smallO(1)), \quad a_0>0, \\
 &v(p)\to \frac{1}{2} \quad \text{as }p\to 0^+, \quad
v(p)\to - \frac{1}{2} 
\quad\text{as }p\to 2\pi^-.
\end{split}
\label{eq:sharp_edgebehaviour_a(p)}
\end{equation}
The same coefficient $a_0$ occurs at both endpoints by reflection symmetry.

Heuristically, at the diagonal contribution at $z=0$ and looking at the first term in \eqref{eq:Feshback_pairing matrix_entries}, we see that 
$$\langle \mathfrak{m}_\xi (0) e_j,e_j \rangle = |\xi|^2 \int_{\mathbb{T}} \frac{|\psi_j(p)|^2 |v(p)|^2 a(p)}{a(p)^2 + \xi^2v(p)^2} d p,
$$ 
due to reflection symmetry.  
Using the asymptotics in \eqref{eq:sharp_edgebehaviour_a(p)}: First, the fractional component, $j={\rm fr}$, gives a term of order $\mathcal{O}(|\xi|^{8/5})$, which we see by splitting the integral in the regions: $0<p<|\xi|^{3/5}$, $p\in (|\xi|^{3/5}, p_0)$ for fixed $p_0$, and a fixed bulk and using that $\psi_{\rm fr}\sim 1$.   This integral splitting is determined by the two terms in the denominator. 

While, for the diffusive component, $j={\rm df}$, we rather get a contribution of order $\mathcal{O}(|\xi|^{2})$, since $\psi_{\rm df}\sim \omega$. To see this, we split again the integral in the same three regions, and we check that the coefficient of $|\xi|^2$ is integrable. 

Thus we expect \textbf{two different orders/scales to appear} in the matrix expansion!     
The next lemma establishes the corresponding expansions for the full matrix, including the microscopic correction and the mixed (off-diagonal) coordinates.  

\begin{lemma}
\label{lem:matrix expansion_xi small}
There exist constants
$c_{\rm fr}, c_{\rm df} >0$ 
such that the following asymptotics hold:  
\begin{itemize}
 \item[(i)] Uniformly  for $|z|\lesssim|\xi|^{8/5}$,
\begin{equation}  \label{eq:matrix_expansion}
 \mathfrak{F}_\xi(z) =
\begin{bmatrix}
 z+c_{\rm fr}|\xi|^{8/5}
 + \smallO(|\xi|^{8/5}) &
     \mathcal{O}(|\xi|^{9/5})
\\
 \mathcal{O}(|\xi|^{9/5}) &
z+\mathcal{O}(\xi^2) 
\end{bmatrix},
\end{equation}
and 
\begin{equation} \label{eq:matrix_cross-product}
\big[\mathfrak{F}_\xi(z)\big]_{\rm fr,df} 
\big[\mathfrak{F}_\xi(z)\big]_{\rm df,fr} =\smallO(|\xi|^{18/5}).
\end{equation} 
\item[(ii)] Uniformly, on the smaller region $|z|\lesssim\xi^2$, the diffusive
diagonal has the sharper expansion
\begin{equation} \label{eq: df,df entry_sharper expansion}
[\mathfrak F_\xi(z)]_{\rm df,df} = z + c_{\rm df}\xi^2+ \smallO(\xi^2).
\end{equation}
\item[(iii)] Finally, on $|z|=\vartheta_*|\xi|$,
\begin{equation} \label{eq:Outer_expansion}
\mathfrak{F}_\xi(z) = z I_2 + \smallO(|\xi|).
\end{equation}
\end{itemize}
\end{lemma}

\begin{proof}[Proof of Lemma \ref{lem:matrix expansion_xi small}]

We start with \underline{some first upper bound estimates} for the matrix entries. 
We recall that the entries are given by 
\begin{equation}\label{eq:entries for Feshbach_formula}
[\mathfrak{F}_\xi(z)]_{ij}
=\langle\mathfrak{m}_\xi(z)e_j,e_i\rangle  -\left\langle\mathfrak{m}_\xi(z)\mathscr{K}\mathsf{Q} \mathscr{W}_\xi(z)^{-1}\mathsf{Q}\mathfrak{m}_\xi(z)e_j,e_i\right\rangle.
\end{equation}
We first compute 
\begin{align}
& \|\mathfrak{m}_\xi(z)e_{\rm fr}\|_{L^2_p}^2
\lesssim\int_{a\leq |\xi| }a\ dp
     +|\xi|^2\int_{a >|\xi|}\frac{1}{a}\  dp
 \lesssim |\xi|^{8/5},
\quad \text{ and} \\  &\|\mathfrak{m}_\xi(z)e_{\rm df}\|_{L^2_p}^2 
\lesssim\int_{a\leq |\xi|}a\omega^2 dp  +|\xi|^2\int_{a>|\xi|}\frac{\omega^2}{a} dp
 \lesssim |\xi|^2.
\end{align} 
The same bounds hold for
$\|\mathfrak m_\xi(z)^*e_j\|_{L^2_p}$. For example, the two integrals in the first line are estimated as
$\int_0^{|\xi|^{3/5}}p^{5/3}\ dp= \mathcal{O}(|\xi|^{8/5})$ and
$|\xi|^2\int_{|\xi|^{3/5}}^1 p^{-5/3}\ dp= \mathcal{O}(|\xi|^{8/5})$.
In the second line, $\omega^2/a\sim\omega^{1/3}$ is integrable and thus the integral contributes of order $|\xi|^2$, while the first term is even faster, thus the leading order is $\mathcal{O}(|\xi|^2)$. 

It is now handy to write 
$$ 
a\mathfrak{m}_\xi(z)
=z+i\xi v-\frac{(z+i\xi v)^2}{a+z+i\xi v}. 
$$
Then due to orthonormality and that $v(p) = -v(2\pi - p)$ and $\psi_j(p) = \psi_j(2\pi - p)$, it holds $\langle \psi_j,\psi_i \rangle= \delta_{ij}$,  and  
$\int v \psi_j \psi_i dp=0$. So we have 
\begin{equation} 
\langle\mathfrak{m}_\xi(z)e_j,e_i
\rangle 
=z\delta_{ij}
-\int_{\mathbb{T}}\frac{(z+i\xi v)^2}{a+z+i\xi v}
    \psi_j\psi_i dp.
\label{eq:first_term of (F_xi)ij}
\end{equation}

Throughout the disk $|z|\leq\vartheta_* |\xi| $, Lemma \ref{lem:m_small_Disk} gives the lower bound $|a+ z +i\xi v|\gtrsim a+|\xi|$ and that 
$|\mathfrak{m}_\xi(z)|\lesssim\frac{|\xi|}{a+|\xi|}$. 
Looking at the denominator and since $a\sim \omega^{5/3}$, the split $a \leq |\xi|$ corresponds to $\omega \lesssim |\xi|^{3/5}$.  
Then, for the integral in \eqref{eq:first_term of (F_xi)ij}, we obtain
$$
|\xi|^2
\int_{\mathbb{T}}\frac{|\psi_i\psi_j|}{a+|\xi|} dp 
\lesssim 
\begin{cases}
|\xi|^{8/5},&i=j={\rm fr},\\
|\xi|^2,&\text{otherwise}.
\end{cases} 
$$
The first bound follows by the same split as above and since $\psi_{\rm fr} \sim 1$. For the mixed
entries, $|\psi_{\rm fr}\psi_{\rm df}|/a\lesssim \omega^{-2/3}$
which is integrable. Finally for the diffusive diagonal we also use that $\omega^2/a$ is integrable, and thus the leading term is $\mathcal{O}(|\xi|^2)$.  
Regarding the second term in \eqref{eq:entries for Feshbach_formula} we estimate 
\begin{equation}\label{eq: micro K_part of m_estimate1}
\begin{split}
&\left|\left\langle\mathfrak{m}_\xi(z)\mathscr{K}\mathsf{Q}
\mathscr{W}_\xi(z)^{-1}\mathsf{Q}\mathfrak{m}_\xi(z)e_j,e_i
 \right\rangle
 \right|\\
&\qquad
\lesssim
 \|\mathfrak{m}_\xi(z)e_j\|_{L^2_p}  \|\mathscr{K}\mathfrak{m}_\xi(z)^*e_i\|_{L^2_p}
 \lesssim
\|\mathfrak{m}_\xi(z)e_j\|_{L^2_p} \|\mathfrak{m}_\xi(z)^* e_i\|_{L^2_p},
 \end{split} 
\end{equation}
where we used the self-adjointness of $\mathscr{K}$ and the uniform bound for 
$\mathscr{W}_\xi(z)^{-1}$.
Altogether these estimates give the  $\mathcal{O}(|\xi|^{9/5})$ upper bound for the off-diagonal terms and the diffusive diagonal $z + \mathcal{O}(|\xi|^2)$. 
They also give $\mathfrak{F}_\xi(z)=zI_2+\mathcal{O}(|\xi|^{8/5})$ in the central disk. In particular, $|\xi|^{8/5}=\smallO(|\xi|)$ proves
\eqref{eq:Outer_expansion}, item (iii),  since $\frac{\|\mathfrak{F}_\xi(z) -z I_2 \|}{|\xi|}\lesssim |\xi|^{3/5} \to 0$. 

So far we have that the diagonal fractional and diffusive contribution are upper bounded by $z+\mathcal{O}(|\xi|^{8/5})$ and $z+\mathcal{O}(|\xi|^{2})$, respectively. We now identify their leading terms. This will allow us to locate later the eigenvalues at the precise  scales. 

We start by providing a little-$\smallO$ improvement for the fractional contribution to the second term in \eqref{eq:entries for Feshbach_formula}, exploiting the compactness of $\mathscr{K}$: 
The sequence 
$\nu_{\xi,z} :=|\xi|^{-4/5}\mathfrak{m}_\xi(z)^* e_{\rm fr}$ is bounded in $L^2_p$  
and converges weakly to zero, uniformly for $z$ in the central disk. This is since on every fixed bulk it tends to zero strongly, but its norm stays bounded: test against a fixed $\phi  \in L^2_p$ and split at some 
$\omega = \delta$. On the fixed bulk $\{\omega \geq \delta\}$, its $L^2$ norm 
is of order $|\xi|^{1/5} \to 0$. On the edge $\{\omega < \delta\}$ the pairing 
is bounded by 
 $ C 
 \| \mathbf{1}_{\{\omega < \delta\}}\phi \|_{L^2_p}$. First let $|\xi| \to 0$ and then $\delta \to 0$, to see the weak convergence of $\nu_{\xi,z}$.  
 Then the compactness of $\mathscr{K}$ implies 
\begin{equation}
\sup_{|z| \leq \vartheta_* |\xi| }
\|\mathscr{K}\mathfrak{m}_\xi(z)^* e_{\rm fr}\|_{L^2_p} = \smallO(|\xi|^{4/5}).
\end{equation}
The uniform conclusion follows by applying compactness along
any sequence $|\xi| \to 0$ and $|z|\leq\vartheta_* |\xi|$. 
Inserting this into \eqref{eq: micro K_part of m_estimate1} shows 
that the second term is $\smallO(|\xi|^{8/5})$ for $(i,j)=({\rm fr},{\rm fr})$ 
and $\smallO(|\xi|^{9/5})$ for $(i,j)=({\rm fr},{\rm df})$.  
The direct mixed entry is $\mathcal{O} (|\xi|^2)= \smallO(|\xi|^{9/5})$, and so
$$
[\mathfrak{F}_\xi(z)]_{\rm fr,df} = \smallO(|\xi|^{9/5}), \qquad
[\mathfrak{F}_\xi(z)]_{\rm df,fr}=\mathcal{O}(|\xi|^{9/5}).
$$ Their product is $\smallO(|\xi|^{18/5})$, proving
\eqref{eq:matrix_cross-product}.

We continue by \underline{identifying the fractional leading coefficient} for item (i).   
At $z=0$, reflection symmetry gives
\begin{equation}
\langle\mathfrak{m}_\xi(0)e_{\rm fr},e_{\rm fr}\rangle
=\int_{\mathbb{T}}
 \frac{a|\psi_{\rm fr}|^2\xi^2v^2}{a^2+\xi^2v^2} dp.
\end{equation} 
On any fixed bulk $\{\omega \geq \delta\}$ this integral is
$\mathcal{O}_\delta(|\xi|^2)=\smallO(|\xi|^{8/5})$.  Near the edges of the frequency domain,   
we have from the asymptotics of each term:  
$\frac{a}{\omega^{5/3}}\to a_{0}$, $v^2 \to \frac{1}{4}$ and $|dp/d\omega|\to2$. Also, $\psi_{\rm fr}$ has the same positive limit at both edges. 
Changing the variable to $\omega=|\xi|^{3/5}y$ and applying then the dominated convergence theorem, with the dominating function $C\min\{y^{5/3},y^{-5/3}\}$, we have:  
\begin{equation}
\langle \mathfrak{m}_\xi(0) e_{\rm fr},e_{\rm fr}\rangle = c_{\rm fr} |\xi|^{8/5} + \smallO(|\xi|^{8/5}),
\end{equation}
where, with the measure $dp$ used above,
\begin{align}
c_{\rm fr} := 4|\psi_{\rm fr}(0)|^2
 \int_0^\infty 
 \frac{a_{0} y^{5/3}}{4 a_0^2 y^{10/3}+1} dy >0.
\end{align}
We also used reflection symmetry. The factor $4$ comes from the two edges and their limiting
Jacobians, each equal to $2$. 
We pass perturbatively to $z \neq 0$. We write 
\begin{align*}
\langle \mathfrak{m}_\xi(z)e_{\rm fr},e_{\rm fr}\rangle - \langle \mathfrak{m}_\xi(0)e_{\rm fr},e_{\rm fr}\rangle & = z\int_{\mathbb{T}}|\psi_{\rm fr}|^2
 \frac{a^2}{(a + i \xi v)(a+ z+ i\xi v)} dp
 \\ & = z (1+\smallO(1)).
\end{align*} 
For the last equality: the integrand is bounded by $C|\psi_{\rm fr}|^2$ and
converges to $|\psi_{\rm fr}|^2$ uniformly in $|z|\leq\vartheta_* |\xi|$ 
for almost every fixed $p$. Its integral therefore converges
uniformly to $\|\psi_{\rm fr}\|_{L^2_p}^2=1$. 

So for $|z|\lesssim |\xi|^{8/5}$, this together with the $\smallO(|\xi|^{8/5})$ bound for
the second term, give 
$$
[\mathfrak{F}_\xi(z)]_{\rm fr,fr}
=z+c_{\rm fr} |\xi|^{8/5} + \smallO(|\xi|^{8/5}).
$$
Altogether with the previous bounds,  conclude 
\eqref{eq:matrix_expansion}. 

Finally we continue by \underline{identifying the diffusive leading coefficient}, in item (ii).   
We work uniformly for $|z|\lesssim\xi^2$ and we expand the two terms in
\eqref{eq:entries for Feshbach_formula} with $i=j={\rm df}$. 

For the first term gives
$$
\langle\mathfrak{m}_\xi(z)e_{\rm df},e_{\rm df}\rangle = z - \int_{\mathbb{T}}
\frac{(z +i\xi v)^2}{a+z+i\xi v} |\psi_{\rm df}|^2 dp.
$$ Now since $z/\xi \to 0$ uniformly in this region,
$$
\frac{1}{|\xi|^2} \frac{(z+i\xi v)^2}{(a+z+i\xi v)} |\psi_{\rm df}|^2 \to
-\frac{v^2|\psi_{\rm df}|^2}{a}
$$
for almost every $p$. By the bound \eqref{eq:lowerbound_freedisk} in Lemma \ref{lem:m_small_Disk}, the absolute value of the integrand is bounded by $C|\psi_{\rm df}|^2/a$, which is integrable as it behaves like $\omega^{1/3}$ near the edges.  Dominated convergence then gives
\begin{align} \label{eq:contrib_1st term}
\langle\mathfrak{m}_\xi(z)e_{\rm df}, e_{\rm df} \rangle
= z + \xi^2 \int_{\mathbb{T}} \frac{v^2|\psi_{\rm df}|^2}{a} dp
 +\smallO(\xi^2).
\end{align}
Next, we expand the second term, which writes as 
$$ \left\langle \mathscr{K}
\mathsf{Q}\mathscr{W}_\xi(z)^{-1}\mathsf{Q} \mathfrak{m}_\xi(z)e_{\rm df}, 
 \mathfrak{m}_\xi(z)^*e_{\rm df}
\right\rangle.
$$  We notice that the first order limit of the multiplier-factor is 
$$
\frac{\mathfrak{m}_\xi(z)e_{\rm df}}{i\xi}
=
\frac{ \left(\tfrac{z}{i\xi} + v\right) a^{1/2}\psi_{\rm df}}{ a +z+ i\xi v} 
\to \frac{v\psi_{\rm df}}{a^{1/2}} =: \varphi_{\rm df}.
$$
The limit belongs in $L^2_p$, since
$|\varphi_{\rm df}|^2\lesssim\omega^{1/3}$ near the edges.  
Since also the quotient above is bounded by
$C a^{-1/2}|\psi_{\rm df}|\in L^2_p$, dominated convergence yields (uniformly in  $|z|\lesssim\xi^2$):   
\begin{align} \label{eq:leading error_in_xi_for_ m_{df}}
\mathfrak{m}_\xi(z)e_{\rm df}
=i\xi\  \varphi_{\rm df}+ \smallO(|\xi|),
\quad \text{ and }\quad 
\mathfrak{m}_\xi(z)^*e_{\rm df}
=-i\xi \ \varphi_{\rm df} + \smallO(|\xi|) 
\ \in \ L^2_p.  
\end{align} 
Also $\mathsf{Q} \varphi_{\rm df}=\varphi_{\rm df} \in \mathsf{Q}L^2_p$, thanks to the reflection symmetry: 
$ \langle \varphi_{\rm df},e_j\rangle
=\int_{\mathbb{T}}v\psi_{\rm df}\psi_j dp=0$, 
for both $j\in\{{\rm fr},{\rm df}\}$. 
 
We next replace $\mathscr{W}_\xi(z)^{-1}$ 
by its limit $\mathscr{A}^{-1}$, in operator norm: Recall that
$$ \mathscr{W}_\xi(z)
=\mathscr{A} + \mathsf{Q}\mathfrak{m}_\xi(z)\mathscr{K}\mathsf{Q}, 
\qquad \mathscr{A} = \mathsf{Q}(I-\mathscr{K})\mathsf{Q}. 
$$ 
The perturbation tends to zero in operator norm by 
\eqref{eq:strong converg_compactK}.
Since both inverses are uniformly  
bounded on $\mathsf{Q}L^2_p$, from the identity
$$
\mathscr{W}_\xi(z)^{-1}-\mathscr{A}^{-1}
=-\mathscr{W}_\xi(z)^{-1} 
\mathsf{Q}\mathfrak{m}_\xi(z)\mathscr{K}\mathsf{Q} \mathscr{A}^{-1}
$$
we get $
\|\mathscr{W}_\xi(z)^{-1}-\mathscr{A}^{-1}\|
 _{\mathsf{Q}L^2_p\to\mathsf{Q}L^2_p}
\to 0$. 
uniformly in the disk.
From this and \eqref{eq:leading error_in_xi_for_ m_{df}}, the second term is
\begin{align}\label{eq:contrib_2nd term}
-\xi^2 \langle \mathscr{K}\mathscr{A}^{-1} \varphi_{\rm df},
 \varphi_{\rm df}\rangle
+\smallO(\xi^2). \footnote{Here the minus sign follows from
$\langle i\xi h_1,-i\xi h_2\rangle
=-\xi^2\langle h_1, h_2\rangle $.}
\end{align}
Now the contributions from both the first and second term, \eqref{eq:contrib_1st term} and \eqref{eq:contrib_2nd term}, give 
$$
[\mathfrak F_\xi(z)]_{\rm df,df}
=z+\xi^2\Big(
 \|\varphi_{\rm df}\|_2^2+
 \langle\mathscr{K}\mathscr{A}^{-1}\varphi_{\rm df},\varphi_{\rm df}\rangle \Big) + \smallO(\xi^2).
$$
On $\mathsf{Q}L^2_p$, the identity
$(I-\mathscr{K})\mathscr{A}^{-1}=I_{\mathsf{Q}}$
gives
$ I_{\mathsf{Q}}+\mathscr{K}\mathscr{A}^{-1}
=\mathscr{A}^{-1}$.
Thus indeed we conclude the claim $[\mathfrak F_\xi(z)]_{\rm df,df}
=z+c_{\rm df}\xi^2+\smallO(\xi^2),
$ with 
$c_{\rm df}:= \langle\mathscr{A}^{-1} \varphi_{\rm df}, \varphi_{\rm df} \rangle > 0$. 
The sign follows from positivity and invertibility 
of $\mathscr{A}$ on $\mathsf{Q}L^2_p$, (cf. the proof of Lemma \ref{lem:W_micro_invertible}), and
$\varphi_{\rm df}\neq 0$. 
\end{proof}

\subsection{Location of the two eigenvalues} 
We continue by applying Rouch\'e's theorem on two small circles centered at 
$-c_{\rm fr}|\xi|^{8/5}$ and $-c_{\rm df}\xi^2$, i.e. the scales identified above, to locate the zeros of $\det\mathfrak F_\xi(z)$. We then apply it on the outer circle to show that 
there are no other zeros in the central disk. Figure \ref{fig:lowfrequency_spectrum} below  illustrates their locations.

\begin{proposition}
\label{lem:2poles_disk}
\label{prop:determinant_contours_expansion small xi} 
Let $0<\vartheta_{\rm fr}<c_{\rm fr}$ and
$0<\vartheta_{\rm df}<c_{\rm df}$, and define the disks  
\begin{align*} 
\mathbb{D}_{\rm fr}(\xi):=
 \{z:|z+c_{\rm fr}|\xi|^{8/5}|<\vartheta_{\rm fr}|\xi|^{8/5}\},& \quad  
\mathbb{D}_{\rm df}(\xi):=
 \{z:|z+c_{\rm df}\xi^2|<\vartheta_{\rm df}\xi^2\}. 
\end{align*} 
Write $\Gamma_j(\xi)=\partial\mathbb{D}_j(\xi)$ for 
$j\in\{{\rm fr},{\rm df}\}$. 
For sufficiently small $0<|\xi|<\xi_*$, these two closed disks are disjoint and contained in the larger central disk $\mathbb{D}_\xi$: 
$$\mathbb{D}_\xi:=\{z\in\mathbb{C}: |z|<\vartheta_*|\xi|\},
\qquad 
\Gamma_\xi:=\partial\mathbb{D}_\xi.$$
The determinant $\det\mathfrak{F}_\xi$ has exactly one zero, 
counted with multiplicity, in each local disk, and no other zeros 
in $\mathbb{D}_\xi$. In particular, both zeros are simple and satisfy 
$$ z_{\rm fr}(\xi) = - c_{\rm fr} |\xi|^{8/5}+ \smallO(|\xi|^{8/5}),
\quad 
z_{\rm df}(\xi) = -c_{\rm df}\xi^2 + \smallO(\xi^2).
$$
\end{proposition}

\begin{proof}
For each sufficiently small $\xi\neq 0$,
Lemmas \ref{lem:m_small_Disk} and \ref{lem:W_micro_invertible}
imply that $\mathfrak{F}_\xi$ is holomorphic in a neighborhood of $\overline{\mathbb{D}_\xi}$: 
Indeed, the denominator $(a+z+i\xi v)$ is uniformly bounded 
away from zero in a neighborhood of the closed disk. So $\mathfrak{m}_\xi(z)$,
$\mathscr{M}_\xi(z)$ and $\mathscr{W}_\xi(z)$ are holomorphic 
as bounded operator-valued functions. 
Also $\mathscr{W}_\xi(z)$ is invertible on the closed disk, and its inverse is holomorphic in a possibly smaller neighborhood. 

We first see that since $\xi^2= \smallO(|\xi|^{8/5})$ and
$|\xi|^{8/5}=\smallO(|\xi|)$, both local disks are lying inside the central one. 
Their disjointness follows, for sufficiently small $|\xi|$, from the fact that the sum of their radii is smaller than the distance of their centers:  
$$ (c_{\rm fr}-\vartheta_{\rm fr}) |\xi|^{8/5}
>(c_{\rm df} + \vartheta_{\rm df}) \xi^2.
$$ 
\underline{On $\Gamma_{\rm fr}(\xi)$}, $|z|\sim |\xi|^{8/5}$.
Taking the determinant in \eqref{eq:matrix_expansion}, and using
\eqref{eq:matrix_cross-product}, gives
\begin{equation}
\det \mathfrak{F}_\xi(z)
=z (z+c_{\rm fr}|\xi|^{8/5}) + \smallO(| \xi|^{16/5}).
\end{equation} 
On this circle,
$$
g_{{\rm fr},\xi}(z) := z(z+c_{\rm fr}|\xi|^{8/5}), \quad
|g_{{\rm fr},\xi}|
\geq (c_{\rm fr}-\vartheta_{\rm fr})\vartheta_{\rm fr}|\xi|^{16/5}.
$$
Thus, for sufficiently small $|\xi|$,
$$
|\det\mathfrak{F}_\xi(z) - g_{{\rm fr},\xi}(z)| < |g_{{\rm fr},\xi}(z)| 
\quad\text{on }\Gamma_{\rm fr}(\xi). 
$$
An application of Rouch\'e's theorem then gives us that $\det\mathfrak{F}_\xi$
and $g_{{\rm fr},\xi}$ have the same number of zeros
in $\mathbb{D}_{\rm fr}(\xi)$, counted with multiplicity. 
The polynomial $g_{{\rm fr},\xi}$ has zeros $0$ and $-c_{\rm fr}|\xi|^{8/5}$.
Only the second lies in this disk, since
$\vartheta_{\rm fr}<c_{\rm fr}$. 
Therefore $\det\mathfrak{F}_\xi$ has exactly one zero there. 

\noindent
\underline{On $\Gamma_{\rm df}(\xi)$}, similarly we have  
$(c_{\rm df}- \vartheta_{\rm df}) \xi^2 \leq |z| \leq (c_{\rm df} + \vartheta_{\rm df}) \xi^2$, and so $|z|\lesssim \xi^2$, which also lies in the region $|z|\lesssim |\xi|^{8/5}$.  
We estimate the determinant of   \eqref{eq:matrix_expansion}, which with  \eqref{eq:matrix_cross-product}
and the sharper asymptotic \eqref{eq: df,df entry_sharper expansion},  give  
\begin{equation}
\det\mathfrak{F}_\xi(z)
=(z+c_{\rm fr}|\xi|^{8/5})(z+c_{\rm df}\xi^2)+ \smallO (|\xi|^{18/5}). 
\end{equation}
The polynomial  
$ g_{{\rm df}, \xi}:=(z+c_{\rm fr}|\xi|^{8/5})(z+c_{\rm df}\xi^2) $ satisfies the lower bound $ | g_{{\rm df}, \xi}| \geq \frac12 c_{\rm fr}\vartheta_{\rm df}|\xi|^{18/5}$ on this circle, 
after decreasing $\xi_*$. Thus on $\Gamma_{\rm df}(\xi)$: 
$$\det\mathfrak{F}_\xi(z) - g_{{\rm df},\xi}(z) = \smallO(|\xi|^{18/5}) \ \Longrightarrow \ |\det\mathfrak{F}_\xi(z) - g_{{\rm df},\xi}(z)| < g_{{\rm df},\xi}(z).$$ 

The only zero of $g_{{\rm df},\xi}(z)$ lying in $\mathbb{D}_{\rm df}(\xi)$ is  $-c_{\rm df}\xi^2$. Rouch\'e's theorem then  gives exactly one zero of the determinant in this disk as well. 

\noindent 
\underline{On the central circle  $\Gamma_{\xi}$},  
\eqref{eq:Outer_expansion} gives $ \det\mathfrak{F}_\xi(z)= z^2+ \smallO(\xi^2)$. 
Since
$|z|^2= \vartheta_*^2\xi^2$ 
on $\Gamma_{\xi}$,  Rouché's theorem applies again with the function $z^2$, and shows that the $\det\mathfrak{F}_\xi=0$ has exactly two zeros counted  
with multiplicity, in $\mathbb{D}_{\xi}$. These are the two zeros already found in the disjoint local disks, and so they are both simple. 

To conclude the precise stated little-$\smallO$ asymptotics, we may repeat the local arguments with arbitrarily
small fixed positive $\vartheta_{\rm fr},\vartheta_{\rm df}$. 
\end{proof}

\begin{remark}[Reality of the zeros] 
\label{rem:reality_evalues} The two 
zeros that we found are real and strictly negative. 
Indeed, first notice that both disks $\mathbb{D}_{\rm fr}(\xi), \mathbb{D}_{\rm df}(\xi)$ are invariant under complex conjugation (since their centers are real).  Also, if $z_j(\xi)$ is the unique zero in one of these disks, then it holds that 
$\overline{z_j(\xi)}$ is also a zero in the same disk: 
For $h\in L^2_p$, write $h^\#(p):=\overline{h(2\pi-p)}$. We have 
$(\mathscr{M}_\xi(z)h)^\# = \mathscr{M}_\xi(\overline{z})h^\#$ due to the  reflection symmetries of $a(2\pi-p)=a(p)$, $v(2\pi-p) = - v(p)$,  
and the reality and reflection invariance of $\mathscr{K}$.  
Since $a$ and the basis $\psi_j$ are real and
reflection invariant, $e_j=a^{1/2}\psi_j$ satisfies $e_j^\#=e_j$, 
and $(\mathsf{P}h)^\#=\mathsf{P}h^\#$, $(\mathsf{Q}h)^\#=\mathsf{Q}h^\#$.  
Thus both the microscopic inverse and the (macroscopic) Feshbach map have the same properties.  
In the basis  
$\{e_{\rm fr},e_{\rm df}\}$, this means 
$$
 \mathfrak{F}_\xi(\overline z) = \overline{\mathfrak{F}_\xi(z)},
\qquad
\det\mathfrak{F}_\xi(\overline z) =
\overline{\det\mathfrak{F}_\xi(z)}. 
$$
Due to uniqueness counting multiplicity, we have  
$z_{\rm fr}(\xi),\ z_{\rm df}(\xi)\in \mathbb{R}$.  
Finally, the positivity of
$c_{\rm fr},c_{\rm df}$ and the asymptotics of the zeros $z_j$ give, for sufficiently small $\xi$, 
$$
z_{\rm fr}(\xi)\leq
    -\frac{c_{\rm fr}}2|\xi|^{8/5}<0, \quad 
    z_{\rm df}(\xi)
    \leq
    -\frac{c_{\rm df}}2\xi^2<0. 
$$
\end{remark}

\subsection{Eigenvectors in the original variables} 
The invertibility equivalence proved above in \eqref{eq:equiv_invertibility}, shows that the two simple, real zeros of $\det\mathfrak{F}_\xi$ that we found in Proposition \ref{lem:2poles_disk}, are precisely the spectral points of $B(\xi)$ in the central disk $\mathbb{D}_\xi$. 
We now return from the variable $w=a^{1/2}u$ to the original variable $u$, construct their corresponding eigenvectors and show that each associated Riesz projection has rank one, i.e. they are in fact algebraically simple eigenvalues.  

\begin{lemma}[Eigenvectors and simplicity]
\label{lem:u_reconstruction}
For each of the two zeros $z_j(\xi)$, $j  \in \{ {\rm df}, {\rm fr}\}$, 
choose a nonzero 
$\zeta_j \in \ker \mathscr{F}_\xi(z_j)\subset\mathscr{N}:=a^{1/2}\ker L$  and set
$$ w_j:= \zeta_j-\mathscr{W}_\xi(z_j)^{-1}
\mathsf{Q} \mathscr{M}_\xi(z_j)\zeta_j,
\qquad
u_j:= \frac{a^{1/2}}{a+z_j+i\xi v}\mathscr{K} w_j.
$$
Then $u_j \in L^2_p$ is nonzero,
$a^{1/2}u_j=w_j$, and $(z_j-B(\xi))u_j=0$. 
Moreover, $z_j$ is an algebraically simple eigenvalue of $B(\xi)$.
\end{lemma}

\begin{proof}  By construction, $\mathsf{P}w_j=\zeta_j$ and
$\mathsf{Q}\mathscr{M}_\xi(z_j)w_j=0$. Its remaining component is
$\mathsf{P}\mathscr{M}_\xi(z_j)w_j
=\mathscr{F}_\xi(z_j)\zeta_j=0$, since $\mathsf{P}\zeta_j = \zeta_j$. Now since $w_j \in L^2_p$ and $\mathscr{K}$ is bounded, $\mathscr{K} w_j \in L^2_p$. Moreover for the denominator, Lemma \ref{lem:m_small_Disk} gives us  for each fixed $\xi\neq 0$,
$$  \left|\frac{a^{1/2}}{a+z_j+i\xi v}\right| \lesssim\frac{a^{1/2}}{a+|\xi|}
 \leq\frac{1}{2\sqrt{|\xi|}}. 
$$ Thus $u_j \in L^2_p$ indeed and its definition gives
$a^{1/2}u_j=w_j\neq 0$. Since
$Ku_j=a^{1/2}\mathscr K w_j$, it also gives
$(a+z_j+i\xi v)u_j=Ku_j$, which is equivalent to $(z_j - B(\xi))u_j=0$, as required.

To prove algebraic simplicity, recall that
$\det \mathfrak{F}_\xi(z)$ has a simple zero at $z_j$\footnote{meaning that $\det \mathfrak{F}_\xi(z_j)=0,$ and $  \partial_z\det\mathfrak{F}_\xi(z_j)\neq 0$.},  
and so for $z\neq z_j$ sufficiently close to $z_j$, the 
matrix inverse formula gives 
$ \mathfrak F_\xi(z)^{-1}  =
\frac{\operatorname{adj}\mathfrak F_\xi(z)}{\det\mathfrak{F}_\xi(z)}, 
$ where $\operatorname{adj}$ denotes the adjugate matrix. 
Taylor expanding at $z_j$ then yields
$$ \mathfrak{F}_\xi(z)^{-1}
= \frac{1}{z-z_j}
 \frac{\operatorname{adj}\mathfrak F_\xi(z_j)}{\partial_z\det
 \mathfrak{F}_\xi(z_j)}
    + \text{a matrix holomorphic near }z_j.
$$
The residue $\operatorname{Res}_{z=z_j}\mathfrak F_\xi(z)^{-1} =
 \frac{\operatorname{adj}\mathfrak F_\xi(z_j)}{\partial_z\det
 \mathfrak{F}_\xi(z_j)}$ has rank $1$: since $\mathfrak F_\xi(z_j)$ 
is singular but it is not the zero matrix, the same holds for its adjugate (since these are two-by-two matrices) and therefore this has rank one. Also the denominator $\partial_z\det\mathfrak F_\xi(z_j)\neq 0$, thus the residue
 has indeed rank one. In particular, it is nonzero, so 
$\mathfrak{F}_\xi(z)^{-1}$ has a simple pole at $z_j$. 

Now \eqref{eq:bounded_physical_resolvent} together with the block inverse formula 
\begin{align*}
\mathscr{M}_\xi(z)^{-1} &= \mathsf{Q}\mathscr{W}_\xi(z)^{-1}\mathsf{Q} + \big[\mathsf{P}-\mathsf{Q}\mathscr{W}_\xi(z)^{-1} \mathsf{Q}\mathscr{M}_\xi(z)\mathsf{P}\big] \mathscr{F}_\xi(z)^{-1} \big[\mathsf{P}-\mathsf{P}\mathscr{M}_\xi(z)\mathsf{Q}
\mathscr{W}_\xi(z)^{-1}\mathsf{Q}\big]
\end{align*} show that
the resolvent of $B$ has at most a simple pole with residue of
rank at most one. This is because all factors except $\mathscr{F}_\xi(z)^{-1}$  are holomorphic near $z_j$, and the inverse 
$\mathscr{F}_\xi(z)^{-1}
=J\mathfrak{F}_\xi(z)^{-1}J^*$ has a simple pole with rank-one residue. Thus, $\mathscr{M}_\xi(z)^{-1}$ has at most a simple pole, whose residue has rank at most one. 

The same conclusion holds for $(z-B(\xi))^{-1}$ by
\eqref{eq:bounded_physical_resolvent}, since its other factors are holomorphic near $z_j$. Its residue is nonzero because $
(z-B(\xi))^{-1}u_j=u_j (z-z_j)^{-1}$. 

Thus, its residue, which is the Riesz projection since $\mathbb{P}_j = \frac{1}{2\pi i}\int_{\Gamma_j} (z-B(\xi))^{-1} dz
=\operatorname{Res}_{z=z_j}(z-B(\xi))^{-1}$, for $\Gamma_j$ being a contour containing (only) $z_j$ and by Cauchy's theorem, has rank one. Thus its range, the generalised e-space associated to $z_j$ has dimension one which  proves the claimed algebraic simplicity.
\end{proof}

\section{Linear Semigroup Decay} \label{sec:linear}

Having now a detailed understanding of the low-frequency spectrum, we move to establish the decay of the linear semigroup. We use $\widehat{h}(\xi,p)=\int_{\mathbb{R}}e^{-ix\xi}h(x,p) dx$
and the inverse factor $(2\pi)^{-1}$. For $t\geq 0$, write
$\langle t \rangle=10+t$. 

We recall here the Fourier norms that we are using:
\begin{align*}
 X_{\mu,r}[h]
 :=\left\|\|\omega^\mu\widehat{h}(\xi)\|_{L^\infty_p}\right\|_{L^r_\xi},
 \quad \mu\in\{-1/2,1/6,1/2\},\ \text{ for }\ r=1,2.
\end{align*}

We also work in the $L^2_p$-weighted space
$$
 \mathscr{Y}=L^2(a^{-1}dp),\qquad
 \|f\|_{\mathscr{Y}}^2 = \int_{\mathbb{T}}\frac{|f(p)|^2}{a(p)} dp,
$$
and denote its spatial Fourier supremum as
$$ \|\widehat{h}\|_{L^\infty_\xi\mathscr{Y}}
 :=\operatorname*{ess\,sup}_{\xi\in\mathbb{R}} \|\widehat{h}(\xi, \cdot) \|_{\mathscr{Y}}. 
$$  
The semigroup acts in Fourier variables by
$\widehat{S(t)h}(\xi)=e^{tB(\xi)}\widehat{h}(\xi)$.

The main result of this section is

\begin{theorem}[Decay of linear semigroup]\label{thm:linear}
Let $0<\delta<1/10$, and suppose that $\Pi g_0(x,\cdot)=0$ for almost every $x$. For $r=1,2$, whenever the right-hand sides are finite, 
\begin{align}
 X_{1/6,r}[S(t)g_0] &\lesssim_\delta  \langle t\rangle^{-2/5}  \big(X_{-1/2,r}[g_0]+
 \|\widehat{g}_0\|_{L^\infty_\xi\mathscr{Y}} 
 \big),
 \label{eq:decay_weight6}\\
 X_{1/2,r}[S(t)g_0]
 &\lesssim_\delta  \langle t\rangle^{-3/5+\delta} \big(X_{-1/2,r}[g_0]+
 \|\widehat{g}_0\|_{L^\infty_\xi \mathscr{Y}}
 \big).
 \label{eq:decay_weight-2}
\end{align}
The same argument gives the unweighted estimate
\begin{equation} \label{eq:decay_L2}
 \| S(t)g_0 \|_{L^2_{x,p}} \lesssim_\delta\langle t\rangle^{-3/5+\delta} \big(X_{-1/2,2}[g_0] + \|\widehat{g}_0\|_{L^\infty_\xi\mathscr{Y}} 
 \big).
\end{equation}
In particular,
\begin{align}
\| \omega^{1/6}S(t)g_0 \|_{L^\infty_{x,p}} &\lesssim_\delta\langle t \rangle^{-2/5}
\big(X_{-1/2,1}[g_0]+
\|\widehat{g}_0\|_{L^\infty_\xi \mathscr{Y}}\big),
 \label{eq:decay_physical1}\\
 \| \omega^{1/2}S(t)g_0 \|_{L^\infty_{x,p}}
 &\lesssim_\delta\langle t\rangle^{-3/5+\delta} \big(X_{-1/2,1}[g_0]+
\|\widehat{g}_0\|_{L^\infty_\xi \mathscr{Y}}\big).
 \label{eq:decay_physical2}
\end{align}
\end{theorem}

\medskip

The proof proceeds in the following steps: (i) We first obtain $L^2_p$ decay for each spatial frequency.
At small frequencies, we separate the two eigenmodes and use an ordinary $L^2_p$ energy estimate on the remaining invariant subspace. At frequencies bounded away from zero, we use a modified energy functional different but equivalent to $L^2_p$ and perform an energy estimate to this, as in classical hypocoercivity theory. 
In both cases, we decompose the momentum domain into its bulk, where the collision frequency $a$ is bounded below, and the edges, where we use a weighted norm.  
(ii) We pass from an $L^2_p$-decay, to an $L^c_p$-decay for $c>2$, on a fixed bulk, by using the smoothing effect of the operator $K$. Which we then turn into a weighted pointwise decay on the whole momentum domain, via a kernel splitting estimate and Duhamel's formula. 
(iii) Finally, we pass back to the physical variables, by integrating the estimates in spatial frequency and Fourier inverting to obtain the stated decay in physical space. 

Let us remark that the strategy is inspired by \cite{GLM26}, but of course, the additional nontrivial feature  needed to be taken care of here, is transport. The argument  in \cite{GLM26} uses in an essential way the preservation of flat orthogonality to 
$\ker L$ along the linear flow. But for $B(\xi)$ here, that orthogonality is generally not preserved. We replace it, at small frequencies, by orthogonality to the two adjoint eigenvectors on the invariant spectral complement, while at larger frequencies, we use the modified energy. We must also keep track of the dependence on $\xi$, so that the resulting bounds can be integrated in frequency, in order to Fourier invert. 

We recall that for fixed $\xi_* \leq 1$ sufficiently small, the low-frequency spectral analysis of the previous section gives exactly two algebraically simple, real eigenvalues in the central disk, satisfying 
\begin{equation}
 -z_{\rm fr}(\xi)\sim|\xi|^{8/5},\qquad  -z_{\rm df}(\xi)\sim|\xi|^2,  \qquad 0<|\xi|<\xi_*.
\end{equation} 
Their Riesz projections are denoted by $\mathbb{P}_{\rm fr}(\xi),\mathbb{P}_{\rm df}(\xi)$ and define the projection onto the remaining spectral subspace:  
\begin{equation}
 \mathbb{R}_\xi := I-\mathbb{P}_{\rm fr}(\xi)-\mathbb{P}_{\rm df}(\xi).
 \label{eq:remainder_projection}
\end{equation}
All three projections commute with $B(\xi)$.

For $0 < r\leq 1/2$, we write
$$ \mathcal{E}_r:=\{\omega<r\},\quad
 \mathcal{B}_r:=\{\omega\geq r\}.
$$

\medskip 

We will be using the weighted $L^2_p$-norm $\|\omega^{-\theta} h\|_{L^2_p}$ where 
$\theta =\frac{5 \rho}{3} \in (5/6,1)$   for 
$\frac{1}{2}< \rho <\frac{3}{5}$.  The reason for introducing this weight is the elementary edge-estimate 
$\| 
\mathbf{1}_{\mathcal{E}_r}h\|_{L^2_p}^2 \le r^{ 2 \theta }\| \omega^{-\theta}h\|_{L^2_p}^2$.  It will give the decay exponent $\rho = 3\theta/5$.
Also note that the restriction $\theta<1$ is compatible with the initial $\omega^{-1/2}L^\infty_p$ norm since 
\begin{equation}
\| \omega^{-\theta}f\|_{L^2_p}^2
 \leq \| \omega^{-1/2}f \|_{L^\infty_p}^2
       \int_\mathbb{T} \omega^{1-2\theta} dp
 \lesssim_\theta \| \omega^{-1/2}f \|_{L^\infty_p}^2.
 \end{equation}

 We decompose  
$ h=h^\parallel + h^{\perp_a}$,  where 
  $ h^\parallel = \Pi_a h$ is the Kernel component in the weighted geometry $L_a^2$, and $h^{\perp_a}=(I-\Pi_a)h$ is its weighted orthogonal component. Here $\Pi_a$ is the orthogonal projection onto $\ker L$ in $L^2_a$, whereas $\Pi$ denotes the orthogonal projection in
flat $L^2_p$.

We start with the following lemma, which collects the eigenvector and projection estimates needed below.

\begin{lemma}\label{lem:estimates_linear}
After decreasing $\xi_*$, the eigenvectors can be scaled so that
the following estimates hold for $0<|\xi|<\xi_*$:  
 \begin{itemize} 
 \item[(i)] For each index $j={\rm df}, {\rm fr}$:    
 $$ u_j(\xi) \to \psi_j \quad \text{in } \ L^2_p,\qquad \|u_j\|_{\mathscr{Y}} \lesssim  |\xi |^{-1/3}\ \ \text{ and } \quad \| \omega^{-\theta} u_j\|_{L^2_p}
 \lesssim_\rho |\xi|^{-2/5}.$$
\item[(ii)] For each index $j={\rm df}, {\rm fr}$,  the projections $\mathbb{P}_j$ are uniformly bounded on the flat $L^2_p$ space. Also, if $f \in \mathscr{Y}$ and $\Pi f=0$, then 
\begin{align} 
\| \mathbb{P}_{\rm fr} f\|_{L^2_p} \lesssim |\xi|^{4/5} \| f \|_{\mathscr{Y}}, \qquad 
 \| \mathbb{P}_{\rm df}f\|_{L^2_p}
 \lesssim |\xi| \|f \|_{\mathscr{Y}},   \qquad  \| \omega^{-\theta} \mathbb{R}_\xi f \|_{L^2_p}
&\lesssim \| \omega^{-\theta} f \|_{L^2_p} 
 + \|f\|_{\mathscr{Y}}.
\end{align}
\item[(iii)] For every $h\in\operatorname{Ran}\mathbb{R}_\xi$,
\begin{equation}
 \| h^\parallel\|_{L^2_p}  \lesssim |\xi|^{-1/3} \| h^{\perp_a}\|_{L^2_a}.
 \label{eq:moment_control}
\end{equation}
\end{itemize} 
\end{lemma}

Before we proceed with the proof of this lemma, let us remark that the last estimate \eqref{eq:moment_control} is the inhomogeneous substitute for using preservation of the collision moments in the homogeneous problem. The subspace $\operatorname{Ran}\mathbb{R}_\xi$ is invariant under the evolution 
$e^{tB(\xi)}$, so \eqref{eq:moment_control} remains valid
throughout its evolution.

\begin{proof}
\underline{Item (i):} Set $ w_j:=a^{1/2}u_j$. We first show that  $a^{1/2}u_j(\xi) \to a^{1/2}\psi_j$ in  $L^2_p$.  The equation for $w_j$ is $\mathscr{M}_\xi(z_j)w_j=0$. Taking its microscopic 
$\mathsf{Q}$-component gives
$$ \mathsf{Q} w_j= -\mathscr{W}_\xi(z_j)^{-1}
      \mathsf{Q} \mathfrak{m}_\xi(z_j)\mathsf{P} w_j.
$$
In particular, $\mathsf{P} w_j\ne0$, since otherwise
$w_j=0$ and hence $u_j=0$, contradiction. Taking its $\mathsf{P}$-component gives
$$ \mathscr{F}_\xi(z_j)\mathsf{P} w_j=0.
$$
Write $\mathsf{P} w_j= c_{{\rm fr},j} e_{\rm fr}
 + c_{{\rm df},j}e_{\rm df}$, for  $e_j=a^{1/2} \psi_j$ which implies that the coefficient vector belongs in the 
$\ker\mathfrak{F}_\xi(z_j)$ \footnote{After applying $J^*$ to
$\mathscr{F}_\xi(z_j)\mathsf{P} w_j=0$, we have this matrix equation.} :  
 $$ \mathfrak{F}_\xi(z_j) (c_{{\rm fr},j}, c_{{\rm df},j})^T = 0. $$  

Now, at $j= {\rm fr}$, the matrix estimates \eqref{eq:matrix_expansion} give
$  \left|   \big(\mathfrak{F}_\xi(z_{\rm fr})\big)_{\rm df,df} \right|
    \gtrsim |\xi|^{8/5}$, and $
  \big(\mathfrak{F}_\xi(z_{\rm fr})\big)_{\rm df,fr} = \mathcal{O}(|\xi|^{9/5}).$ Also, looking at the second row of the above matrix equation we have $$[\mathfrak{F}_\xi(z_{\rm fr})]_{\rm fr,{\rm fr}}\ c_{{\rm fr}, {\rm fr} } + [\mathfrak{F}_\xi(z_{\rm fr})]_{\rm fr,df}\ c_{{\rm df},{\rm fr}} =0. $$
So, the first diagonal entry is lower bounded and thus we may scale the e-vector $u_{\rm fr}$ so that 
$c_{{\rm fr},{\rm fr}}=1$. The second row of the 
matrix equation then gives
$c_{{\rm df},{\rm fr}}= \mathcal{O} (|\xi|^{1/5})$. 

Similarly, at $z_{\rm df}$: the first  diagonal entry has order bounded below by a positive 
constant times $|\xi|^{8/5}$, and thus $c_{{\rm df},{\rm df}}\neq 0$. We may therefore scale $u_{\rm df}$ 
so that $c_{{\rm df},{\rm df}}=1$, and then the first row gives $c_{{\rm fr},{\rm df}}=\mathcal{O} (|\xi|^{1/5})$. 
Thus $\mathsf{P} w_{{\rm fr}} = e_{{\rm fr}} + \mathcal{O}(|\xi|^{1/5}) e_{{\rm df}}$ and $\mathsf{P} w_{{\rm df}} = \mathcal{O}(|\xi|^{1/5}) e_{{\rm fr}} +  e_{{\rm df}}$. In other words: 
$$ \mathsf{P} w_j  \to e_j
    \quad\text{in }L^2_p.
$$
We already know that the microscopic inverse is uniformly bounded, and the multiplier 
$\mathfrak{m}_\xi(z_j)$ is uniformly bounded and converges strongly to zero, so 
\begin{align*}
 \|\mathsf{Q} w_j\|_{L^2_p}
 \lesssim \|\mathfrak{m}_\xi(z_j)\mathsf{P}w_j\|_{L^2_p} \lesssim
\|\mathsf{P} w_j-e_j\|_{L^2_p}
+\|\mathfrak{m}_\xi(z_j)e_j\|_{L^2_p}
\to  0.
\end{align*}

It follows indeed that
\begin{equation}
a^{1/2}u_j=w_j \to a^{1/2}\psi_j
    \quad\text{in }L^2_p.
\label{eq:linear_vectors_weighted_convergence}
\end{equation}
In particular, 
$\|a^{1/2} u_j\|_{L^2_p} $ is uniformly bounded. We want now to pass to the unweighted convergence. From the previous section we already got $u_j \in L^2_p$ for each $\xi\neq 0$. This allows us to apply the smoothing effect of $K$ from Appendix \ref{app:Kernel_smoothing_edges} for $K$, using that $u_j =\frac{Ku_j}{a+z_j+i\xi v} = \frac{a^{1/2}}{a+z_j+i\xi v} \mathscr{K} w_j$, and aiming to get the boundedness uniformly in $\xi$.

Choose $\varepsilon,R$ so that $\varepsilon R>2$ and 
$$ \varepsilon+\frac{1}{R} < \min\left\{\frac{1}{20},1-\theta \right\}.
$$  
On the edge $\mathcal{E}_r=\{\omega < r\}$, H\"older's inequality with $\frac{1}{q} = \frac{1}{2}-\frac{1}{R}$,
 and Lemma \ref{lem:app_Kernelbound_edge}, give
\begin{align*}
\|\mathbf{1}_{\mathcal{E}_r} u_j\|_{L^2_p}
&\lesssim \left\|
\mathbf{1}_{\mathcal{E}_r}     \frac{\omega^{3/2-\varepsilon}}{ a + |\xi| }
    \right \|_{L^q_p}
\|\omega^{-3/2+\varepsilon} Ku_j\|_{L^R_p}\lesssim  \|\mathbf{1}_{\mathcal{E}_r}
        \omega^{-1/6 - \varepsilon}\|_{L^q_p}
    \|u_j\|_{L^2_p}, 
\end{align*}
where we used the lower bound \eqref{eq:lowerbound_freedisk} and $a + |\xi| \geq a \gtrsim \omega^{5/3}$. 
Also, 
\begin{align*}
 \|\mathbf{1}_{\mathcal{E}_r}
        \omega^{-1/6-\varepsilon}\|_{L^q_p}
    &\lesssim
    \left(\int_0^{r}
        x^{-q (1/6 + \varepsilon)} dx \right)^{1/q}\lesssim r^{1/q-1/6-\varepsilon},
\end{align*}
which is finite  because $\varepsilon+1/R< 1/3$, due to the choices of the parameters. 
Thus on the edge: 
\begin{equation}\label{eq:edgeestimate_2_linear}
\|\mathbf{1}_{\mathcal{E}_r}u_j\|_{L^2_p}
    \leq C r^{1/3-(\varepsilon+\tfrac{1}{R})}\|u_j\|_{L^2_p},
\end{equation}
while on the bulk $\mathcal{B}_r$, since $a \gtrsim r^{5/3}$ there, we have 
$$\|\mathbf{1}_{\mathcal{B}_r}u_j \|_{L^2_p}
\lesssim r^{-5/6}\|a^{1/2}u_j\|_{L^2_p}
 =r^{-5/6}\|w_j\|_{L^2_p}.
$$
Adding these together: 
$$ \|u_j\|_{L^2_p} \leq C r^{1/3-(\varepsilon+ \tfrac{1}{R})} \|u_j\|_{L^2_p} + C r^{-5/6} \|w_j\|_{L^2_p}.
$$
From this we see that $u_j$ is $\xi$-uniformly bounded, i.e. $\sup_{0<|\xi|<\xi_*}\| u_j(\xi)\|_{L^2_p} < \infty,$ since we can choose $r$ small enough and since $w_j$ is uniformly bounded.  
Now notice that on each fixed bulk, the weighted convergence 
\eqref{eq:linear_vectors_weighted_convergence}  implies
$$ \|\mathbf{1}_{\mathcal{B}_r}(u_j-\psi_j)\|_{L^2_p} \lesssim r^{-5/6}
 \| w_j - a^{1/2} \psi_j\|_{L^2_p} \to 0 \quad \text{ as } |\xi|\to 0. 
$$
On the edges from \eqref{eq:edgeestimate_2_linear}: 
$$\|\mathbf{1}_{\mathcal{E}_r}(u_j-\psi_j)\|_{L^2_p} \leq C r^{1/3-(\varepsilon+ \tfrac{1}{R})}
+\|\mathbf{1}_{\mathcal{E}_r}\psi_j\|_{L^2_p}.$$

Thus altogether 
$$ \limsup_{|\xi|\to 0} \|u_j -  \psi_j\|_{L^2_p} \leq C r^{1/3-(\tfrac{1}{R} + \varepsilon)}
+ \|\mathbf{1}_{\mathcal{E}_r}\psi_j\|_{L^2_p}, $$
which by letting $r\to 0$ proves the first claim of item (i). Since $\|u_j\|_{\mathscr{Y}} = \|a^{-1/2}u_j\|_{L^2_p}$, from 
Lemma \ref{lem:app_Kernelbound_edge}, we have
\begin{align*}
 \|u_j\|_{\mathscr{Y}} &\lesssim
  \left\| \frac{a^{-1/2}\omega^{3/2-\varepsilon}}{a+|\xi|}
 \right\|_{L^q_p} \|\omega^{-3/2+\varepsilon} K u_j\|_{L^R_p} \lesssim \left\|
 \frac{\omega^{2/3-\varepsilon}}{\omega^{5/3}+ |\xi| }\right\|_{L^q_p}, 
\end{align*}
where in the last step we used
$a^{-1/2} \sim \omega^{-5/6}$ and the uniform
$L^2_p$ bound on $u_j$ established above. Also,
$$ \|\omega^{-\theta} u_j\|_{L^2_p}
\lesssim
 \left\| \frac{\omega^{3/2-\varepsilon-\theta}}{\omega^{5/3}+ |\xi|}
    \right\|_{L^q_p}.
$$
On the fixed bulk $\mathcal{B}_r$, both norms
are uniformly bounded. On the edges, to estimate each one of these integrals, we separate when $\omega<|\xi|^{3/5}$ or when $\omega >|\xi|^{3/5}$. We also decrease $\xi_*$ if necessary in order to have $|\xi|^{3/5}<r$. 
For the first integral:    
$$ \int_{0}^{|\xi|^{3/5}} \frac{\omega^{(2/3-\varepsilon)q }}{|\xi|^q} d \omega +
\int_{|\xi|^{3/5}}^r
 \omega^{-q(1+\varepsilon)} d\omega \leq 
\mathcal{O} \left( |\xi|^{-q+\tfrac{3}{5}\left(\tfrac{2}{3}-\varepsilon \right)q +\tfrac{3}{5}} \right),   
$$
where we bounded the denominator $\omega^{5/3}+|\xi| \geq |\xi|$, and since both integrals have the same order in $\xi$. After taking the $1/q$-root and using $1/q = 1/2 - 1/R$, gives $|\xi|^{-\tfrac{3}{10} - \tfrac{3}{5}(\varepsilon + R^{-1})} \leq |\xi|^{-1/3}$, when for example  $\varepsilon+1/R< 1/20 $.    
For the second integral, the weighted $L^2$-norm of item (i), the relevant integral is 
$$
|\xi|^{-q}\int_0^{|\xi|^{3/5}} 
\omega^{(3/2-\varepsilon-\theta)q}\ d\omega
 + \int_{|\xi|^{3/5}}^r
    \omega^{-q(1/6+\varepsilon+\theta)} d\omega
 \lesssim |\xi|^{-q+\frac{3}{5}[(3/2-\varepsilon-\theta)q+1]}, 
$$  
which again after taking the $1/q$-root, gives $ |\xi|^{\tfrac{1}{5}-\rho-\tfrac{3}{5} (\varepsilon + R^{-1})} \leq |\xi|^{-2/5}$, since $\theta=5\rho/3$ and $\varepsilon + R^{-1}<1-\theta$.
These conclude point (i). 

\underline{Item (ii):} Since $L$ is real and self-adjoint,
$ B(\xi)^*=L+i\xi v$ and the eigenvalues $z_j$ are real, so
$\overline{u_j}$ satisfies
$$ B(\xi)^*\overline{u_j} =z_j \overline{u_j}.
$$
Because $z_j$ is algebraically simple, its Riesz projection is of rank $1$ and in particular  
\begin{equation}\label{eq:linear_projection_formula}
\mathbb{P}_j f =
\frac{\langle f,\overline{u_j}\rangle}
 {\langle u_j,\overline{u_j} \rangle} u_j.
\end{equation}
Note that the denominator satisfies 
$ \langle u_j,\overline{u_j} \rangle
=\int_{\mathbb{T} } u_j^2 dp
\to \int_{\mathbb{T} }\psi_j^2 dp=1, 
$
since from item (i): $\left|\int(u_j^2-\psi_j^2) dp\right| \leq
 \|u_j-\psi_j\|_{L^2_p}(\|u_j\|_{L^2_p} +\|\psi_j\|_{L^2_p})
    \to 0$. 
So for $\xi$ sufficiently small, the denominators in
\eqref{eq:linear_projection_formula} are bounded away from zero. 
Together with the uniform bound of the $  \|u_j\|_{L^2_p}$ from item (i), 
this proves uniform boundedness of $\mathbb{P}_j$ on the flat $L^2_p$. It also implies uniform boundedness
of $\mathbb{R}_\xi$ on that space. 

Next, from the adjoint eigenvalue equation we have (taking real parts)  
$$ -\langle L \overline{u_j},\overline{u_j}\rangle
= -z_j\|\overline{u_j}\|_{L^2_p}^2.
$$
The left-hand side is equal to $-\langle L \overline{u_j}^{\perp_a},\overline{u_j}^{\perp_a}\rangle$, since $\overline{u_j}^{\parallel}$ is in $\operatorname{ker} L$. Now the weighted coercivity (see \cite[Proposition 4.1]{GLM26}) gives: $$-\langle L \overline{u_j}^{\perp_a},\overline{u_j}^{\perp_a}\rangle \gtrsim \int a(p) |\overline{u_j}^{\perp_a}|^2 dp  = \|\overline{u_j}^{\perp_a}\|_{L^2_a}^2, $$
and the uniform $L^2$-bound of $u_j$, imply 
$$
    \|\overline{u_j}^{\perp_a}\|_{L^2_a}
    \lesssim \sqrt{-z_j}.
$$
If $\Pi f=0$, then $f$ is $L^2$ flat-orthogonal to
$\overline{u_j}^\parallel\in\ker L$. Thus, by the Cauchy-Schwarz inequality, we have 
\begin{align*} 
|\langle f,\overline{u_j}\rangle| = 
|\langle f,\overline{u_j}^{\perp_a} \rangle| \leq 
\|a^{-1/2}f\|_{L^2_p} 
\|a^{1/2}\overline{u_j}^{\perp_a}\|_{L^2_p}\lesssim
 \sqrt{-z_j} \|f\|_{\mathscr{Y}}.
\end{align*}

Inserting this to the Riesz projection 
\eqref{eq:linear_projection_formula} gives indeed 
$$ \|\mathbb{P}_{\rm fr}f\|_{L^2_p}
\lesssim |\xi|^{4/5}\|f\|_{\mathscr{Y}},
 \qquad
 \|\mathbb{P}_{\rm df}f\|_{L^2_p}
 \lesssim |\xi| \|f\|_{\mathscr{Y}}.
$$
Combining now these with the bounds in item (i), we obtain
$$ 
\|\omega^{-\theta} \mathbb{P}_{\rm fr} f \|_{L^2_p}
\lesssim_\rho 
|\xi|^{2/5}\|f\|_{\mathscr{Y}},
\qquad
 \|\omega^{-\theta} \mathbb{P}_{\rm df} f \|_{L^2_p}
 \lesssim_\rho |\xi|^{3/5}
 \|f\|_{\mathscr{Y}}.
$$
Finally, since $\mathbb{R}_\xi f = f -\mathbb{P}_{\rm fr}f-\mathbb{P}_{\rm df} f$, triangle inequality gives 
$$
\|\omega^{-\theta} \mathbb{R}_\xi f \|_{L^2_p} \lesssim_\rho \|\omega^{-\theta} f \|_{L^2_p} +
\|f\|_{\mathscr{Y}}.
$$
\underline{Item (iii):} Let $h \in \operatorname{Ran} \mathbb{R}_\xi$.
The spectral projections associated with distinct
eigenvalues satisfy
$\mathbb{P}_j\mathbb{R}_\xi=0$.
Thus from  \eqref{eq:linear_projection_formula} it must hold that 
$\langle h, \overline{u_j} \rangle = 0$,
for both $j \in \{\mathrm{fr},\mathrm{df}\}$. 
Since $h^\parallel \in \operatorname{Ker} L$, we also write
$$ h^\parallel= c_1\phi_1 + c_2\phi_2,
\qquad \phi_1 = \mathfrak{f}, \quad
\phi_2 = \omega\mathfrak{f}.
$$
Since $h=h^\parallel+h^{\perp_a}$, these mean that 
$$\sum_{i=1}^2
c_i\langle\phi_i,\overline{u_j} \rangle
= -\langle h^{\perp_a},\overline{u_j} \rangle,
\quad j \in\{\mathrm{fr},\mathrm{df}\}.
$$
The coefficient matrix converges, as $|\xi|\to0$, to $\big[\langle\phi_i,\psi_j\rangle\big]_{j,i}$, and 
this limiting matrix is invertible because
$\{\phi_1,\phi_2\}$ and $\{\psi_{\rm fr},\psi_{\rm df}\}$
are bases of the same space $\ker L$.
For $|\xi|$ sufficiently small, the nearby matrices have uniformly bounded inverses. We see this from a Neumann series argument for example. Thus,
$$ \|h^\parallel\|_{L^2_p} \lesssim |c_1|+|c_2| \lesssim \sum_j |\langle h^{\perp_a}, \overline{u_j} \rangle|.
$$
Finally, inserting the weight $a$ and applying Cauchy-Schwarz together with 
the estimates in item (i), give
\begin{align*}
\|h^\parallel\|_{L^2_p}
\lesssim \|h^{\perp_a}\|_{L^2_a}
    \sum_j \|\overline{u_j} \|_{\mathscr{Y}} \lesssim
    |\xi|^{-1/3} \|h^{\perp_a}\|_{L^2_a}.
\end{align*}
\end{proof}

We continue by observing that with the same exponents $R>2$, $\varepsilon R >2$, $\frac{1}{q} = \frac{1}{2} - \frac{1}{R}$ and $\varepsilon+1/R < \min\{1/20, 1-\theta\}$,  we also have 
\begin{align}\label{eq:weighted_gain}
\| \omega^{-\theta} e^{-t(a+i\xi v)} Kh \|_{L^2_p}
 &\leq \| \omega^{3/2-\varepsilon-\theta}e^{-ta}\|_{L^q_p}
\| \omega^{-3/2+\varepsilon}Kh \|_{L^R_p} 
 \lesssim  \langle t \rangle^{-3/5} \| h \|_{L^2_p}.
\end{align}

Indeed this follows by an application of the appendix smoothing estimate \eqref{lem:app_Kernelbound_edge} and H\"{o}lder's Inequality.   
For the last inequality we estimate the norm $\| \omega^{3/2-\varepsilon-\theta}e^{-ta}\|_{L^q_p}$: at the bulk the decay is exponential since $a$ is lower bounded while at the edge, we may substitute $y := t\omega^{5/3}$ and get 
$$
\|\omega^{3/2-\varepsilon-\theta}e^{-ta}\|_{L^q_p}
 \lesssim \langle t \rangle^{-\frac{3}{5}(3/2-\varepsilon-\theta + 1/q)}.  
$$ The decay exponent is  
$6/5-\rho-(3/5)(\varepsilon+1/R)>3/5$, thus the \eqref{eq:weighted_gain} final exponent.  
The estimate is uniform in $\xi$, since the transport term has modulus $1$. 

In the following abstract lemma we get $L^2_p$-decay estimates under the general assumption of having an
energy functional $H$ which is
(i) equivalent to the squared $L^2_p$ norm,
(ii) dissipates along the flow, and
(iii) is controlled by its dissipation and the weighted norm  
$\|\omega^{-\theta}g(t)\|_{L^2_p}$. 
When applying the lemma later, the energy $H$ and the
parameter $\sigma$ will need to be adapted to the frequency region, 
but the lemma applies in both cases.

\begin{lemma}\label{lem:abstract_energy}
Fix $1/2<\rho<3/5$ and $\theta=5\rho/3$. Let $g(t)=e^{tB(\xi)} g_0$ with $\| \omega^{-\theta} g_0 \|_{L^2_p}<\infty$. 
Suppose there are $0 < \sigma \leq 1$, a locally absolutely continuous energy $ H(t)\sim\|g(t)\|_{L^2_p}^2$, and  $ D(t)\geq 0$ such that for fixed $r_1>0$,  
\begin{align}
 &H'(t) \leq -c D(t), 
 \quad \text{ and } \label{eq:abstract_dissipation}
 \\
& H(t)\leq C(r^{-5/3}+\sigma^{-1})D(t) +Cr^{2\theta} \| \omega^{-\theta} g(t) \|_{L^2_p}^2, \qquad 0<r\leq r_1.  
\label{eq:abstract_bound}
\end{align}
Then
\begin{align}
 &\sup_{t\geq 0 } \| \omega^{-\theta} g(t) \|_{L^2_p}
 \lesssim \| \omega^{-\theta} g_0 \|_{L^2_p} + \sigma^{-2/5} \| g_0 \|_{L^2_p},
 \label{eq:propagated_weight}\\
 &\sup_{t \geq 0}\ \langle t \rangle^\rho \| g(t) \|_{L^2_p}
 \lesssim  
 \| \omega^{-\theta} g_0 \|_{L^2_p}
 + \sigma^{-\rho} \| g_0 \|_{L^2_p}. 
 \label{eq:energy_estimate}
\end{align}
\end{lemma}

\begin{proof} 
We start by multiplying \eqref{eq:abstract_bound}  
by $\lambda>0$ and adding it with 
\eqref{eq:abstract_dissipation}. This gives  
$$ H'(t) + \lambda H(t) \leq 
\left[-c+C\lambda(r^{-5/3}+\sigma^{-1})\right]D(t)
 + C \lambda r^{2\theta} \| \omega^{-\theta} g\|_{L^2_p}^2.
$$
We can make the coefficient of $D$ negative: For example if we wite $r^{5/3}=C_0\lambda$,  $0 < \lambda \leq c_0 \sigma $, for $C_0$ chosen sufficiently large and a $c_0>0$ chosen sufficiently small so that 
$$ C \lambda r^{-5/3} = \frac{C}{C_0}\leq \frac{c}{2},
    \qquad
    C\lambda \sigma^{-1}\leq Cc_0\leq \frac{c}{2}.
$$

Then 
\begin{equation}
 H'+\lambda H \leq C \lambda r^{2\theta} \| \omega^{-\theta} g\|_{L^2_p}^2 =  C_\rho \lambda^{1+2\rho} \|\omega^{-\theta} g\|_{L^2_p}^2.
\label{eq:abstract_differential_decay}
\end{equation}
since $r=(C_0\lambda)^{3/5}$ and $\theta= 5\rho/3$. 

If $\lambda = \lambda(t)$ and $\Lambda (t) := \int_0^t \lambda(\tau) d \tau$ so that $\Lambda(0)=0 $ and $\Lambda'=\lambda$, then the equation \eqref{eq:abstract_differential_decay} becomes 
$$ \frac{d}{dt}\left( e^{\Lambda(t)} H(t) \right) \leq C_\rho e^{\Lambda(t)} \lambda(t)^{1+2\rho} \|\omega^{-\theta} g(t) \|_{L^2_p}^2.$$ Integrating we get
$$H(t) \leq e^{-\Lambda(t)} H(0) + C_\rho \int_0^t e^{\Lambda(s)-\Lambda(t)} \lambda(s)^{1+2\rho}\|\omega^{-\theta} g(s) \|_{L^2_p}^2 ds. $$

We choose $\lambda(t) = \frac{\lambda_0}{t+t_0}$ so that $t_0 = \overline{C} \sigma^{-1}$, and $\lambda_0 \geq 4$. We choose $\overline{C}$ sufficiently large so that $\overline{C} \geq \text{max} \left\{ 10, \tfrac{\lambda_0}{c_0},
\tfrac{C_0\lambda_0}{r_1^{5/3}} \right\}$.  
Then $t_0 \geq 10$ and for every $t\geq 0$,  $\lambda(t)\leq c_0\sigma$ and $ r(t):=(C_0\lambda(t))^{3/5}\leq r_1$ (as $ 0<\sigma \leq 1 $). Thus we may apply the above inequalities with this choice of $\lambda(t)$ and $r(t)$.

Then $\Lambda(t) = \lambda_0 \log ((t+t_0)/t_0)$ and so we compute $e^{\Lambda(s)-\Lambda(t)} = \left( \frac{s+t_0}{t+t_0} \right)^{\lambda_0}$. 
Inserting all these in our main inequality gives 
\begin{align*} 
H(t) &\lesssim \left( \frac{t_0}{t+t_0} \right)^{\lambda_0} H(0) + C_\rho \left( \sup_{0\leq s \leq t} \|\omega^{-\theta} g(s) \|_{L^2_p}^2 \right)(t+t_0)^{-\lambda_0} \int_0^t (s+t_0)^{\lambda_0-1-2\rho} ds  \\
&\lesssim 
\left( \frac{t_0}{t+t_0} \right)^{\lambda_0} H(0) 
+ 
C_\rho \left( \sup_{0\leq s \leq t} \|\omega^{-\theta} g(s) \|_{L^2_p}^2\right)(t+t_0)^{-2\rho}.
\end{align*}

Thus taking the square root, using that $t_0\geq 10$ and the equivalence between $H(t)$ and $\|g(t)\|_{L^2_p}^2$ we get 
\begin{equation} \label{eq:abstract_L^2 estimate}
\begin{split}
\|g(t)\|_{L^2_p} &\lesssim \left( \frac{t_0}{t+t_0} \right)^{\lambda_0/2} \|g_0\|_{L^2_p}  + C_\rho \langle t \rangle^{-\rho} \sup_{0\leq s \leq t} \|\omega^{-\theta} g(s) \|_{L^2_p}\\
& \leq C_\rho (1+ \sigma t)^{-\lambda_0/2}\|g_0\|_{L^2_p} + C_\rho \langle t \rangle^{-\rho} \sup_{0\leq s \leq t} \|\omega^{-\theta} g(s) \|_{L^2_p}.  
\end{split}
\end{equation}
Here we also used that $t_0\sim\sigma^{-1}$. 
We now want to obtain a uniform in time estimate: Fix $T\geq 0$ and let $t\geq T$. 
Duhamel's formula gives
$$
g(t)= e^{-(a(p) + i \xi v(p))t}g_0 + \int_0^t e^{-(a(p) + i \xi v(p))(t-s) } Kg(s) ds, \quad g(s) = e^{sB(\xi)}g_0.
$$
and since $|e^{-t(a+i\xi v)}| \leq 1$,
together with \eqref{eq:weighted_gain}, we get 
\begin{equation}
 \|\omega^{-\theta} g(t) \|_{L^2_p} \leq 
 \|\omega^{-\theta} g_0 \|_{L^2_p}
 + C_\rho\int_0^t
 \langle t-s \rangle^{-3/5} \|g(s)\|_{L^2_p} ds.
\label{eq:abstract_weighted_duhamel}
\end{equation}
Inserting 
\eqref{eq:abstract_L^2 estimate} into
\eqref{eq:abstract_weighted_duhamel}, for $t\leq T$, gives
\begin{equation}\label{eq:abstract_weighted_close}
\begin{split}
\|\omega^{-\theta} g(t) \|_{L^2_p} & \leq 
    \|\omega^{-\theta} g_0 \|_{L^2_p}
    +C_\rho\|g_0\|_{L^2_p}
 \int_0^t
  \langle t-s\rangle^{-3/5}(1+\sigma s)^{-\lambda_0/2} ds \\  & \hspace{4cm}+
    C_\rho
       \sup_{0 \leq s\leq T}\|\omega^{-\theta} g(s) \|_{L^2_p}
       \int_0^t
       \langle t-s\rangle^{-3/5}
       \langle s \rangle^{-\rho} ds.
    \end{split}
\end{equation}
Now the integral in the second term is bounded by
$\langle t \rangle^{2/5-\rho},
$ (since $\rho < 1$), 
while the integral in the first term is
$$ 
\int_0^t 
  \langle t-s\rangle^{-3/5}(1+\sigma s)^{-\lambda_0/2} ds \lesssim \sigma^{-2/5},
$$
since, separating into the cases if 
$t \leq \sigma^{-1}$ and  $t\geq \sigma^{-1}$, in the first case discard the second
factor and use
$$  \int_0^t\langle t-s\rangle^{-3/5} ds
    \lesssim \langle t\rangle^{2/5}
    \lesssim \sigma^{-2/5},
$$
and in the second case, split the integral at $t/2$: For $0\leq s \leq t/2$ we have 
$$  \int_0^{t/2}
 \langle t-s\rangle^{-3/5}(1+\sigma s)^{-\lambda_0/2}ds
\lesssim
       t^{-3/5}\int_0^\infty(1+\sigma s)^{-\lambda_0/2} ds
    =t^{-3/5}\sigma^{-1}
    \leq \sigma^{-2/5}.
$$ 
since $\lambda_0\geq 4$, and for $t/2\leq s\leq t$ we have 
$$ \int_{t/2}^t
       \langle t-s\rangle^{-3/5}(1+\sigma s)^{-\lambda_0/2} ds
    \lesssim
   (\sigma t)^{-2}t^{2/5}=\sigma^{-2/5}(\sigma t)^{-8/5}
    \leq \sigma^{-2/5}.
$$
Altogether, inserting these in \eqref{eq:abstract_weighted_close} gives 
\begin{align}\label{eq:abstract_weighted_close2}
\|\omega^{-\theta} g(t) \|_{L^2_p} & \leq  \|\omega^{-\theta} g_0 \|_{L^2_p}
    +C_\rho \sigma^{-2/5} \|g_0\|_{L^2_p} +
    C_\rho
       \sup_{0 \leq s\leq T}\|\omega^{-\theta} g(s) \|_{L^2_p}\ 
       \langle t \rangle^{2/5-\rho}.
\end{align}
Since we choose $\rho>1/2>2/5$, we may take $t$ large enough, say $t>T_0$, so that $C_\rho\langle t \rangle^{2/5-\rho} <\frac{1}{2}$. 
We take the supremum over time in both sides and rearrange. Of course for this to make sense, we need to make sure that $\sup_{0 \leq s\leq T}\|\omega^{-\theta} g(s) \|_{L^2_p}<\infty$, which we prove below. But given this for now, we have: 
\begin{align}\label{eq:abstract_weighted_close3}
\sup_{T_0 \leq t\leq T} \| \omega^{-\theta} g(t) \|_{L^2_p} & \leq  C_\rho \Big[ \|\omega^{-\theta} g_0 \|_{L^2_p}
    + \sigma^{-2/5} \|g_0\|_{L^2_p} \Big].
\end{align}
Take $T \uparrow \infty$ to conclude for all large times. Regarding the times up to $T_0$: take the supremum over all times up to $T_0$ in  \eqref{eq:abstract_weighted_duhamel} and use that $\|g(t)\|_{L^2_p}\leq \|g_0\|_{L^2_p}$, which we get by the basic $L^2$ energy estimate. Then: 
$$
\sup_{0 \leq t\leq T_0} \|\omega^{-\theta} g(t) \|_{L^2_p} \lesssim  
 \|\omega^{-\theta} g_0 \|_{L^2_p}
 + 
 \langle T_0 \rangle^{2/5} \|g_0 \|_{L^2_p} \lesssim \|\omega^{-\theta} g_0 \|_{L^2_p}
 + 
 \sigma^{-2/5} \|g_0 \|_{L^2_p}.
$$
since $0<\sigma \leq 1$. 
The right-hand side is finite, implying the finiteness of $\sup_{0 \leq t\leq T} \|\omega^{-\theta} g(t) \|_{L^2_p}$. This proves the first claim \eqref{eq:propagated_weight}. For the last one: taking again the sup over $t$ in \eqref{eq:abstract_L^2 estimate} and insert \eqref{eq:propagated_weight}, we immediately get 
$$
\sup_{t\geq 0}\ \langle t \rangle^\rho \|g (t)\|_{L^2_p} \lesssim 
\sup_{t\geq 0} (1+ \sigma t)^{-\lambda_0/2} \langle t \rangle^\rho \|g_0\|_{L^2_p} + 
C_\rho \Big[ \|\omega^{-\theta} g_0 \|_{L^2_p}
 + 
 \sigma^{-2/5} \|g_0 \|_{L^2_p}\Big].
$$
This implies the claim \eqref{eq:energy_estimate}
 after noting that $\sigma^{-2/5}<\sigma^{-\rho}$ and that $(1+ \sigma t)^{-\lambda_0/2} \langle t \rangle^\rho \lesssim \sigma^{-\rho}$. This last thing we see it if we write for example $(t+10)^\rho(1+ \sigma t)^{-\lambda_0/2} \leq \left( \tfrac{10}{\sigma}\right)^{\rho} (1+ \sigma t)^{\rho-2} \lesssim \sigma^{-\rho}$, since $\lambda_0>4$ and $\rho<2$. 
\end{proof}

\begin{remark}
We stress here that, as might be expected, the desired $r$ in \eqref{eq:abstract_bound}, will be the radius separating the edge area $\mathcal{E}_r := \{\omega<r\}$ from the bulk $\mathcal{B}_r := \{\omega\geq r\}$.  
From the proof of Lemma \ref{lem:abstract_energy}  we see that $r$ may depend on time, as $r=r(t)\sim (t+t_0)^{-3/5}$, implying eventually the algebraic decay rate. 

In order to explain the form of \eqref{eq:abstract_bound}:   
The term $ r^{-5/3} D$ will control the microscopic energy on the bulk $\mathcal{B}_r$, exploiting the lower bound $a\gtrsim r^{5/3}$, while the term $\sigma^{-1}D$ will account for the (macroscopic) kernel component. 
On the edge region $\mathcal{E}_r$, where the collision frequency degenerates, we instead use the estimate 
$$ \|\mathbf{1}_{\mathcal{E}_r}g\|_{L^2_p}^2  \leq r^{2\theta}\|\omega^{-\theta}g\|_{L^2_p}^2. 
$$ 
Thus the remainder in \eqref{eq:abstract_bound} measures the edge-effect, where the energy is not directly controlled by dissipation. 
\end{remark}

Now we look at the two cases of small versus larger spatial frequencies, in order to verify the assumptions required in the abstract Lemma \ref{lem:abstract_energy}: 
\\

\noindent
\underline{\textbf{Small spatial frequencies:}} 
Assume that $\xi_*\leq 1$ and let $0<|\xi|<\xi_*$ 
and initial data satisfy $ \Pi f=0$ and 
 $ \| \omega^{-\theta} f\|_{L^2_p} +\|f\|_{\mathscr{Y}} < \infty.$ 
We first estimate the part of the solution remaining after
the two eigenmodes have been removed. 

Because the Riesz projections commute with $B(\xi)$,
the subspace $\operatorname{Ran}\mathbb{R}_\xi$ is $e^{tB(\xi)}$-invariant. 
Thus $g(t) := e^{tB(\xi)}\mathbb{R}_\xi f \in\operatorname{Ran}\mathbb{R}_\xi$ for every $t\ge0$.
Decomposing as 
$$ g=g^\parallel + g^{\perp_a}, \qquad
    g^\parallel = \Pi_a g,
$$ 
item (iii) of Lemma \ref{lem:estimates_linear} yields 
\begin{equation} \label{eq:small_frequency_moment_bound}
    \|g^\parallel(t)\|_{L^2_p}^2 \leq C |\xi|^{-2/3}\|g^{\perp_a}(t)\|_{L^2_a}^2.
\end{equation} 
We now split the momentum domain into edges $\mathcal{E}_r$ and bulk $\mathcal{B}_r$ for $0<r \leq 1/2$.  
The radius $r$ remains free, as required in
Lemma \ref{lem:abstract_energy}.

\noindent
\underline{On the bulk $\mathcal B_r$} we have  $a(p) \geq c_a\omega(p)^{5/3} \geq c_a r^{5/3}$. So,
\begin{align*}
\|\mathbf{1}_{\mathcal B_r}g\|_{L^2_p}^2
 \leq  2\|\mathbf{1}_{\mathcal B_r}g^{\perp_a}\|_{L^2_p}^2 + 2\|\mathbf{1}_{\mathcal{B}_r}g^\parallel\|_{L^2_p}^2 &\leq C r^{-5/3}\|g^{\perp_a}\|_{L^2_a}^2
 +2\|g^\parallel\|_{L^2_p}^2\\
 &\leq
 C (r^{-5/3}+ |\xi|^{-2/3} )
  \|g^{\perp_a}\|_{L^2_a}^2,
\end{align*}
where in the last inequality we used 
\eqref{eq:small_frequency_moment_bound}.

\noindent 
\underline{On the edges $\mathcal{E}_r$}, we instead estimate the full function $g$:
$$ \|\mathbf{1}_{\mathcal{E}_r} g\|_{L^2_p}^2
    = \int_{\mathcal{E}_r}
        \omega^{2\theta} |\omega^{-\theta}g|^2 dp \leq r^{2\theta} \| \omega^{-\theta} g\|_{L^2_p}^2.
$$

Adding now the estimates on the two disjoint sets of the momentum domain, gives
\begin{equation}
    \|g\|_{L^2_p}^2
    \leq
    C (r^{-5/3}+ |\xi|^{-2/3} )
       \|g^{\perp_a}\|_{L^2_a}^2
    +r^{2\theta} \| \omega^{-\theta} g\|_{L^2_p}^2.
    \label{eq:small_frequency_split}
\end{equation}
Now we are going to apply the abstract energy Lemma \ref{lem:abstract_energy} for $H$ being exactly the squared $L^2$ norm:  
Since $v$ is real a standard energy estimate gives 
$$\frac{d}{dt}\|g(t)\|_{L^2_p}^2  =
    2\operatorname{Re}\langle Lg-i\xi vg,g\rangle
    = 2\langle Lg,g\rangle = 2 \langle Lg^{\perp_a},g^{\perp_a}\rangle.
$$
where in the last equality we used that  $Lg^\parallel = 0$ and $L$ is self-adjoint.  
Weighted coercivity then yields that for some constant $c>0$,
\begin{equation} \label{eq:small_frequency_dissipation}
\frac{d}{dt}\|g(t)\|_{L^2_p}^2
 \leq -c \|g^{\perp_a}(t)\|_{L^2_a}^2.
\end{equation} 

\noindent
\underline{Application of the Lemma \ref{lem:abstract_energy}.} Equations \eqref{eq:small_frequency_split} and \eqref{eq:small_frequency_dissipation} now verify the hypotheses
of Lemma \ref{lem:abstract_energy} indeed with
$ H(t)=\|g(t)\|_{L^2_p}^2$, its energy production 
  $ D(t)=\|g^{\perp_a}(t)\|_{L^2_a}^2$ and 
$0< \sigma = | \xi |^{2/3}\leq 1$. This yields 
\begin{equation} \label{eq:small_frequency_remainder_initial}
\sup_{t \geq 0}\ \langle t\rangle^\rho
\|e^{tB(\xi)} \mathbb{R}_\xi f \|_{L^2_p}
\lesssim_\rho
\| \omega^{-\theta} \mathbb{R}_\xi f\|_{L^2_p} 
 +|\xi|^{-2\rho/3}\|\mathbb{R}_\xi f \|_{L^2_p}.
\end{equation}
 
We now express the right-hand side in terms of the original datum $f$. Since $\Pi f=0$, item (ii) of Lemma \ref{lem:estimates_linear} gives
$$  \| \omega^{-\theta} \mathbb{R}_\xi f \|_{L^2_p} 
 \lesssim \| \omega^{-\theta} f\|_{L^2_p}
 +\|f\|_{\mathscr{Y}}.
$$
Also again from item (ii) of Lemma \ref{lem:estimates_linear}, the uniform $L^2_p$ bounds for the two Riesz projections
imply
$$ \|\mathbb{R}_\xi f\|_{L^2_p}
    \leq
    \big(1+\|\mathbb{P}_{\rm fr}\|_{2\to 2 }
    +\|\mathbb{P}_{\rm df}\|_{2\to 2}\big)\|f\|_{L^2_p}
    \lesssim\|f\|_{L^2_p} \leq \|a\|_{L^\infty_p}^{1/2} \|f\|_{\mathscr{Y} }  \lesssim \|f\|_{\mathscr{Y}}.
$$
since $a$ is bounded above.  
Inserting these into
\eqref{eq:small_frequency_remainder_initial}, and since 
$|\xi|^{-2\rho/3}\geq 1$, we have 
\begin{equation} \label{eq:small_frequency_remainder_decay} 
\sup_{t\geq 0}\ \langle t \rangle^\rho
    \|e^{tB(\xi)}\mathbb{R}_\xi f\|_{L^2_p}
    \lesssim_\rho
    \|\omega^{-\theta} f \|_{L^2_p}
    + |\xi|^{-2\rho/3}\|f\|_{\mathscr{Y}}.
\end{equation}

Now regarding the two eigenmodes: 
First the two eigenvalues are algebraically simple, and so 
$$
e^{tB(\xi)}\mathbb{P}_jf =e^{tz_j}\mathbb{P}_jf,
    \qquad j\in\{\mathrm{fr},\mathrm{df}\}.
$$
Thanks to their asymptotics we have $e^{tz_{\rm fr}}\leq e^{-ct |\xi|^{8/5}}$ and 
$ e^{tz_{\rm df}}\leq e^{-ct|\xi|^2}$, for some $c>0$. 
 
Now for any $\nu>0$, $\rho>0$ and $|\xi|$ small enough, we have the bound 
   $$ \sup_{t\geq 0}\ \langle t\rangle^\rho e^{-ct|\xi|^\nu }
    \lesssim_{\rho,c} |\xi|^{-\nu \rho}, 
$$
since for $s = t|\xi|^\nu$, it holds $\langle t\rangle \leq |\xi|^{-\nu} \langle s\rangle$, and the left-hand side is bounded by $\sup_s \langle s\rangle^\rho |\xi|^{-\nu \rho} e^{-cs} \lesssim |\xi|^{-\nu \rho}$.  

Using this with the projection estimates in item (ii) of Lemma  \ref{lem:estimates_linear} gives
\begin{align}
    \sup_{t\geq 0}\ \langle t\rangle^\rho
    \|e^{t z_{\rm fr}}\mathbb{P}_{\rm fr}f\|_{L^2_p}
    \lesssim_\rho
    |\xi|^{4/5 - 8\rho/5}\| f \|_{\mathscr{Y}}\quad \text{ and }\quad 
    \sup_{t\geq 0}\ \langle t\rangle^\rho
    \|e^{t z_{\rm df}}\mathbb{P}_{\rm df}f \|_{L^2_p}
    &\lesssim_\rho
    |\xi|^{1-2\rho}\|f\|_{\mathscr{Y}}.
\end{align}

In both cases we have 
$ |\xi|^{4/5-8\rho/5}\leq |\xi|^{-2\rho/3}$ and $|\xi|^{1-2\rho}\leq |\xi|^{-2\rho/3}$ (same as the remainder decay).  
Finally, altogether with the exact decomposition 
$$  e^{t B(\xi)}f  = e^{t z_{\rm fr}}\mathbb{P}_{\rm fr}f  +
    e^{t z_{\rm df}}\mathbb{P}_{\rm df}f
    + e^{t B(\xi)}\mathbb{R}_\xi f
$$
gives for all $0<|\xi|<\xi_*$: 
\begin{equation} \label{eq:small_frequency_L2_decay}
    \sup_{t\geq 0}\ \langle t \rangle^\rho
    \|e^{t B(\xi)} f\|_{L^2_p}
    \lesssim_\rho \|\omega^{-\theta} f\|_{L^2_p} + |\xi|^{-2\rho/3} \| f \|_{\mathscr{Y} }. 
\end{equation}

Notice that in particular the condition $\Pi f=0$ is imposed only on the initial data. We do not assume that it is preserved by $e^{ t B(\xi)}$. The invariant condition used in the energy argument is $g(t)\in\operatorname{Ran}\mathbb{R}_\xi$, which is preserved by the flow.

\noindent
\underline{\textbf{Spatial frequencies bounded away from zero:}}
We now estimate $g(t)=e^{tB(\xi)}f$ for $|\xi|\ge\xi_*$. In this region, we only assume that  
$ \|\omega^{-\theta} f \|_{L^2_p} <\infty$.  
Recall the collision estimates
\begin{equation} \label{eq:large_frequency_collision_bounds}
    -\langle Lh,h\rangle
    \geq c \|h^{\perp_a}\|_{L^2_a}^2,
    \qquad
    |\langle Lh,z\rangle|
    \leq C\|h\|_{L^2_a}\|z\|_{L^2_a}.
\end{equation}

The ordinary energy identity controls $g^{\perp_a}$. At small frequencies, in order to control the kernel component $g^\parallel$ for $g \in \operatorname{Ran} \mathbb{R}_\xi$, 
we used the estimate \eqref{eq:small_frequency_moment_bound}. This was proved when $|\xi|<\xi_*$, and so it is not available in this regime. In order to control $g^\parallel$ as well, we add a bounded quadratic
correction in the energy that uses transport.  This is the usual idea of modifying the energy in hypocoercivity theory, see for example  \cite{DMS15, villani09hypocoercivity}.   

\smallskip
\noindent
\underline{Construction of the modified energy.}
Write $\phi_1 = \mathfrak{f}$, $
    \phi_2  = \omega\mathfrak{f}$ so that 
  $  \ker L= \operatorname{span}\{\phi_1,\phi_2\}$.  Choose a fixed smooth function $\chi\geq 0$, supported away
from the momentum-edges and invariant under reflection $p \mapsto 2\pi-p$.
We choose it positive on an interval where $v\neq 0$ and $\omega$ is nonconstant, and on the reflected interval.

Define
$$ T_\xi:= i\xi v,  \qquad  \kappa_\xi:= \frac{\xi^2}{1+\xi^2},  \qquad \varphi_\xi:= \frac{\chi}{1+\xi^2}, \qquad  A_\xi:= \Pi_a^* T_\xi^* \varphi_\xi.
$$
Here $\Pi_a^*$ and $T_\xi^*$ are adjoints in flat $L^2_p$ and in particular, $T_\xi^* = -T_\xi$.
The weighted projection $\Pi_a$ is bounded on flat $L^2_p$,
because it is a finite-rank operator whose defining
functionals are pairings with the functions $a\phi_j$. So, $\|A_\xi\|_{2\to2} \leq C\frac{|\xi|}{1+\xi^2} \leq C$. 

For a sufficiently small $\varepsilon_0>0$, we define for each $\xi$ in that region
\begin{equation} \label{eq:large_frequency_energy}
 H_\xi(g) := \|g\|_{L^2_p}^2+
    \varepsilon_0 \operatorname{Re} \langle A_\xi g, g \rangle.
\end{equation}

Since
$ | \langle A_\xi g, g\rangle | \leq C\| g\|_{L^2_p}^2$
we may choose $\varepsilon_0$ so that
\begin{equation} \label{eq:large_frequency_energy_equivalence}
    \frac{1}{2}\|g\|_{L^2_p}^2
    \le H_\xi(g)
    \le\frac{3}{2}\|g\|_{L^2_p}^2\quad \text{for all }\ \xi. 
\end{equation}

For the following calculation, for notational simplicity we neglect the index $\xi$ on $A,T,\varphi,\kappa$. 
Now the functions $a,\chi,\phi_1,\phi_2$ are invariant under
reflection, whereas $v$ is odd and thus $T g^\parallel=i\xi v g^\parallel $ is odd. So 
$ A g^\parallel =0$, 
and also $ \Pi_a T g^\parallel = 0 $. 
Also, since $\Pi_a g^{\perp_a}=0$, for every $z \in L^2_p$,
\begin{equation} \label{eq:large_frequency_projection_pairing}
    \langle Az, g^{\perp_a} \rangle  =
    \langle T^*\varphi z,\Pi_a g^{\perp_a} \rangle =0, 
\end{equation}
and also similarly,
$ \langle A g^{\perp_a},T g^\parallel \rangle
 = \langle T^*\varphi g^{\perp_a},\Pi_a T g^\parallel \rangle = 0.$ 
Finally, because $\Pi_a g^\parallel =g^\parallel$, it holds 
\begin{equation} \label{eq:large_frequency_pairing_identity}
\langle Az, g^\parallel \rangle =
 \langle T^*\varphi z, g^\parallel \rangle =
\langle z,\varphi T g^\parallel \rangle.
\end{equation}

\noindent
\underline{Modified-energy estimate.} We now calculate the dissipation of the modified energy-part, by using the above identities on the projections that contribute zero. Since $L g^\parallel = 0$, the evolution equation is
$ g'= L g^{\perp_a} - T g^\parallel - T g^{\perp_a}$ and so by time-differentiating, we obtain:
\begin{align} \label{eq:large_frequency_correction_derivative}
\frac{d}{dt} \Big( \operatorname{Re}\langle Ag(t),g(t)\rangle\Big) = & -\operatorname{Re}\langle AT g^\parallel, g^\parallel \rangle -\operatorname{Re}\langle AT g^{\perp_a}, g^\parallel \rangle -\operatorname{Re}\langle A g^{\perp_a},T g^{\perp_a}\rangle \notag\\
 &+
 \operatorname{Re}\langle AL g^{\perp_a},g^\parallel \rangle
+\operatorname{Re}\langle A g^{\perp_a},L g^{\perp_a}\rangle.
\end{align}
 
The first term has the good sign. By
\eqref{eq:large_frequency_pairing_identity}: 
$$  I:=\langle ATg^\parallel, g^\parallel\rangle
    = \langle T g^\parallel,\varphi T g^\parallel \rangle
    =  \kappa\int_{\mathbb T}\chi v^2| g^\parallel|^2 dp, 
$$
and so $I$ is real and nonnegative, and
there is a constant $c = c_\chi>0$ such that
\begin{equation} \label{eq:bulk_kernel_bound} 
    \int_{\mathbb{T}}\chi v^2|g^\parallel|^2 dp
    \geq c_\chi \|g^\parallel\|_{L^2_p}^2.
\end{equation}
Indeed, if the integral vanishes, then
$g^\parallel=c_1\phi_1+c_2\phi_2$ should vanish on that interval.
Since $\mathfrak{f} > 0$, this means
$c_1+c_2\omega = 0$ there. But $\omega$ is nonconstant and so it has to be that $c_1=c_2=0$.  
Thus the integral has to be strictly positive whenever $ g^\parallel \in \ker L$ has $L^2_p$ norm one. To get the same lower bound $c_\chi$ for all $g^\parallel$, notice that the $L^2$-unit sphere of the two-dimensional space $\ker L$ is compact and the integral depends continuously on $g^\parallel$. Thus its minimum is attained on that sphere. This minimum cannot be zero, by the above argument. Apply this bound to 
$g^\parallel/\|g^\parallel\|_{L^2_p}$ to see the \eqref{eq:bulk_kernel_bound}.

We next estimate the other four terms. 

For the second term: 
\begin{align*}
 | \langle A T g^{\perp_a}, g^\parallel\rangle|
= |\langle T g^{\perp_a},\varphi T g^\parallel\rangle| &\leq
    \|\varphi^{1/2}T g^{\perp_a}\|_{L^2_p}
    \|\varphi^{1/2}T g^\parallel\|_{L^2_p} \leq C \| g^{\perp_a}\|_{L^2_a} I^{1/2},
\end{align*}
where we used that because $\chi$ is supported where $a$ is lower bounded, 
$$ \| \varphi^{1/2} T g^{\perp_a} \|_{L^2_p}^2
= \kappa \int \chi v^2 | g^{\perp_a}|^2 dp
    \leq C \kappa \| g^{\perp_a}\|_{L^2_a}^2 \leq C \| g^{\perp_a}\|_{L^2_a}^2. 
$$ 

For the fourth term, involving $AL g^{\perp_a}$, we will exploit the following bound for the collision form: Since
$L=-a+K= a^{1/2} [\mathscr{K} - I] a^{1/2}$ and $[\mathscr{K} - I]$ is bounded on $L^2_p$, for every $h,z\in L^2_p$ we have
\begin{align}  \label{eq:collision_form_bound}
|\langle L h,z\rangle| =
 \left|\left\langle [\mathscr{K} - I]a^{1/2}h , a^{1/2}z \right\rangle  \right|
&\leq
\|\mathscr{K}-I\|_{2\to2}
\|a^{1/2}h\|_{L^2_p}\|a^{1/2}z\|_{L^2_p}
 \lesssim \|h\|_{L^2_a}  \|z\|_{L^2_a}.
\end{align}
Therefore for the fourth term we write (using also \eqref{eq:large_frequency_pairing_identity}): 
$$  |\langle AL g^{\perp_a}, g^\parallel \rangle| =  |\langle L g^{\perp_a},\varphi T g^\parallel \rangle|
    \leq C\| g^{\perp_a}\|_{L^2_a}
    \|\varphi T g^\parallel\|_{L^2_a}. 
$$
Now since $a\chi$ is bounded,
$ \|\varphi T g^\parallel\|_{L^2_a}^2 =
    \frac{\xi^2}{(1+\xi^2)^2}
    \int a\chi^2v^2|g^\parallel|^2 dp
    \leq C I$. Therefore the fourth and second term together give:  
\begin{equation} \label{eq:large_frequency_mixed_terms}
    |\langle A T g^{\perp_a}, g^\parallel \rangle| +|\langle AL g^{\perp_a}, g^\parallel \rangle|
    \leq C\| g^{\perp_a} \|_{L^2_a}I^{1/2}.
\end{equation}

To estimate the third and the fifth terms, note that
$ \operatorname{Ran} \Pi_a^*
    =\operatorname{span}\{a\phi_1,a\phi_2\}$. 
In particular, if
$G_{ij}:=\langle\phi_j,\phi_i\rangle_{L^2_a}$,
then
$$\Pi_a^*z = \sum_{k,\ell=1}^2 (G^{-1})_{k\ell}\langle z,\phi_\ell\rangle a\phi_k.
$$
Applying this to $z=T^*\varphi  g^{\perp_a}$ shows that
$$
 A g^{\perp_a}=\sum_{k=1}^2b_k a\phi_k,
    \qquad  |b_1|+|b_2|  \leq
    C\frac{|\xi|}{1+\xi^2}\|g^{\perp_a}\|_{L^2_a}, 
$$
where we used Cauchy-Schwarz, since $\varphi$ is supported on a fixed bulk. 
Now for each $k$, Cauchy-Schwarz and the bound \eqref{eq:collision_form_bound} again give respectively 
$$  |\langle a\phi_k , T g^{\perp_a} \rangle|
    \leq C|\xi|\| g^{\perp_a} \|_{L^2_a}, \qquad |\langle a\phi_k, L g^{\perp_a} \rangle| 
 \lesssim  \| g^{\perp_a} \|_{L^2_a}.
$$

Thus altogether the third and the fifth  terms contribute as 
\begin{equation} \label{eq:large_frequency_microscopic_terms}
 |\langle A g^{\perp_a}  ,T g^{\perp_a}  \rangle| + |\langle A g^{\perp_a} ,L g^{\perp_a}  \rangle| \lesssim \| g^{\perp_a}  \|_{L^2_a}^2,
\end{equation}
uniformly in $\xi$, because
$\xi^2/(1+\xi^2),\ |\xi|/(1+\xi^2) \leq 1$.

Inserting now 
\eqref{eq:large_frequency_mixed_terms} and
\eqref{eq:large_frequency_microscopic_terms}
into \eqref{eq:large_frequency_correction_derivative},
and applying Young's inequality, gives
$$ \frac{d}{dt} \Big( \operatorname{Re}\langle Ag(t),g(t)\rangle\Big) \leq -I + C\|g^{\perp_a}\|_{L^2_a} I^{1/2}
  + C\|g^{\perp_a} \|_{L^2_a}^2
    \leq -\frac{1}{2} I + C_1\| g^{\perp_a} \|_{L^2_a}^2.
$$
The ordinary energy identity yields for some $C_L>0$
$$
    \frac{d}{dt}\|g\|_{L^2_p}^2  =
    2\langle L g^{\perp_a},g^{\perp_a}\rangle
    \leq -2 C_L \|g^{\perp_a}\|_{L^2_a}^2.
$$
Thus the full modified-energy estimate is $$
\frac{d}{dt}H_\xi(g) \leq
 -(2C_L-\varepsilon_0 C_1)\|g^{\perp_a}\|_{L^2_a}^2
    -\frac{\varepsilon_0}{2}I.
$$
Choose $\varepsilon_0$ smaller if necessary so that
$\varepsilon_0 C_1\leq C_L$. Combining with the lower bound in \eqref{eq:bulk_kernel_bound} we obtain 
\begin{equation}  \label{eq:large_frequency_dissipation}
    \frac{d}{dt}H_\xi(g(t)) \leq
    -c\left(
        \|g^{\perp_a}(t)\|_{L^2_a}^2
        +\kappa_\xi\|g^\parallel(t)\|_{L^2_p}^2
    \right),
\end{equation}
where $c>0$ is independent of $\xi$ and $t$.

\noindent
\underline{Application of the Lemma \ref{lem:abstract_energy}.} 
Now we are ready to apply the abstract Lemma \ref{lem:abstract_energy}, this time for 
$$ D_\xi(g):=
\|g^{\perp_a}\|_{L^2_a}^2 + \kappa_\xi\|g^\parallel\|_{L^2_p}^2.$$

Indeed, define $\mathcal{E}_r$ and
$\mathcal{B}_r$ as we did before for the small frequencies.  
As before on the edges we have 
$  \|\mathbf{1}_{\mathcal{E}_r}g\|_{L^2_p}^2
    \leq r^{2\theta}  \|\omega^{-\theta} g \|_{L^2_p}^2$ 
    while on the bulk we have, 
\begin{align*}
\|\mathbf{1}_{\mathcal{B}_r}g\|_{L^2_p}^2
    \leq
2\|\mathbf{1}_{\mathcal{B}_r}g^{\perp_a}\|_{L^2_p}^2
    +2 \|g^\parallel\|_{L^2_p}^2 &\leq
    Cr^{-5/3}\|g^{\perp_a}\|_{L^2_a}^2
    +2\|g^\parallel\|_{L^2_p}^2\\
    &\leq
    C (r^{-5/3}+\kappa_\xi^{-1} )D_\xi(g).
\end{align*}

Together with the equivalency of the energies in 
\eqref{eq:large_frequency_energy_equivalence}, this gives
\begin{equation} \label{eq:large_frequency_split}
 H_\xi(g) \leq
    C ( r^{-5/3}+\kappa_\xi^{-1} )D_\xi(g)
 + Cr^{2\theta} \| \omega^{-\theta} g\|_{L^2_p}^2.
\end{equation}

For every $\xi\neq 0$, we have $0<\kappa_\xi \leq 1$ and so indeed we are good to go: 
 \eqref{eq:large_frequency_dissipation} and
\eqref{eq:large_frequency_split} verify the hypotheses of
Lemma \ref{lem:abstract_energy}, with
$ H(t)=H_\xi(g(t))$, $ D(t)=D_\xi(g(t))$, and $\sigma=\kappa_\xi$.  We therefore get that
\begin{equation}  \label{eq:modified_energy_decay_all_frequencies}
 \sup_{t\geq 0}\ \langle t\rangle^\rho
    \|e^{tB(\xi)}f\|_{L^2_p}
\lesssim_\rho \| \omega^{-\theta} f\|_{L^2_p} +\kappa_\xi^{-\rho}\|f\|_{L^2_p}.
\end{equation}

Finally, for $|\xi|\ge\xi_*$ we have 
$ \kappa_\xi \geq \frac{\xi_*^2}{1+\xi_*^2} =:\kappa_*>0$.  
Also $\|f\|_{L^2_p}\leq \| \omega^{-\theta} f \|_{L^2_p}$, 
because $0 < \omega \leq 1$ a.e., and so 
\begin{equation}  \label{eq:large_frequency_L2_decay}
    \sup_{t \geq 0}\ \langle t \rangle^\rho
    \|e^{tB(\xi)} f\|_{L^2_p}
    \lesssim_{\rho,\xi_*} 
    \| \omega^{-\theta} f\|_{L^2_p},
    \qquad |\xi|\geq \xi_*.
\end{equation}

\begin{remark}
Let us note here that we could in principle apply the modified energy $H$ for the whole spatial (nonzero) frequency range, not just for the large ones. 
But if we were to do that, in the very last step, the term $\kappa_\xi^{-\rho}$ would not be uniformly bounded for small $\xi$: in fact it would be 
$\kappa_\xi^{-\rho} \sim |\xi|^{-2\rho}$ near zero, which is too singular (not integrable) there when $\rho>1/2$. At small frequencies, the preceding argument using the two eigenmodes gives the better frequency dependence needed for Fourier integration, i.e. the weaker singularity $|\xi|^{-2\rho/3}$, which is locally integrable in both $L^1_\xi$ and $L^2_\xi$ for our choice of $\rho$.
Therefore, indeed we need to separate into the two cases of small and larger spatial frequencies and apply the modified energy argument only when $|\xi|\geq \xi_*$. 

The ordinary $L^2_p$ energy would not work for all spatial frequencies, since, as we  explained also when we applied Lemma \ref{lem:abstract_energy}, at larger spatial frequencies we lack control of the kernel component $g^\parallel$. The hypocoercive modified functional provides such control exploiting the transport through its construction.  
\end{remark}

Next, to pass from $L^2_p$ to a weighted supremum as stated in the main linear theorem, we first need to estimate the solution in $L^c_p$ on a fixed bulk, with $c>2$.

We start with a lemma that bounds $\omega^\mu|Kh|$ by a bulk-$L^c_p$ norm and a small multiple
of the global weighted supremum.
The estimate holds for a.e. $p$, including near
the edges. Inserting it into Duhamel's formula will
allow us later to pass from decay on the bulk to weighted
pointwise decay on the whole momentum domain.

\begin{lemma}\label{lem:kernel_splitting_linear}
Fix $2<c<\infty$ and $\mu\in\{1/6,1/2\}$.  For every $\eta>0$, there exists a single $0 < r_0 < 1$ such that,  
with $\mathcal{B}_{r_0} := \{\omega \geq r_0\}$, 
\begin{equation}  \label{eq:kernel_splitting_linear}
    \omega(p)^\mu |K h(p) |
    \leq a(p)\left( \eta\|\omega^\mu h\|_{L^\infty_p} + C_{\eta,c,r_0}\|\mathbf{1}_{\mathcal B}h\|_{L^c_p}  \right), 
\end{equation}
for both $\mu=1/6$ and $\mu=1/2$, every
$h\in L^2_p$ with $\omega^\mu h\in L^\infty_p$,
and almost every $p\in\mathbb{T}$.
\end{lemma}

\begin{proof}
Write for the edge $\mathcal{E}_{r_0}:=\{\omega<r_0\}$. 
 Recall that $K=-K_1+2K_2$ and $a\sim\omega^{5/3}$.  
 For the two roots $r_-(p)<r_+(p)$ of $F_-(p,\cdot)$, the
Appendix \ref{app:Kernel_smoothing_edges}, \eqref{eq:appendix_K1_bound}, gives
\begin{equation}  \label{eq:kernel_splitting_K1_Abel}
    |K_1 h(p)|
    \lesssim \omega(p)^{3/2}
    \bigg[
        \int_0^{r_-(p)}
        \frac{|h(q)|}{\sqrt{r_-(p)-q}} dq
        +
        \int_{r_+(p)}^{2\pi}
        \frac{|h(q)|}{\sqrt{q-r_+(p)}} dq
    \bigg].
\end{equation}
The same appendix also gives the asymptotic, as $p \to 0$, for the roots: 
$$
 2\pi - r_-(p)\sim p^{1/3},\quad
 2\pi - r_+(p)\sim p^3, 
$$
and their reflected versions at $2\pi$. 
We first apply this to $\mathbf{1}_{\mathcal{E}_{r_0}}h$: Since $|h(q)|\le\omega(q)^{-\mu} \| \omega^\mu h \|_{L^\infty_p}$, for $\mu=1/6$, the elementary integrals
$$
 \int_0^{\min\{r_0,y\}}
       \frac{s^{-1/6}}{\sqrt{y-s}} ds\lesssim r_0^{1/3},
 \qquad
 \int_z^{r_0} \frac{s^{-1/6}}{\sqrt{s-z}} ds\lesssim r_0^{1/3},\ \  \text{ for } 0 <z < r_0
$$
give from \eqref{eq:kernel_splitting_K1_Abel}, 
$$
 \omega^{1/6}| K_1(\mathbf{1}_{\mathcal{E}_{r_0}}h) | \lesssim r_0^{1/3}a \| \omega^{1/6} h \|_{L^\infty_p}.
$$
For $\mu = 1/2$, the corresponding integrals satisfy for $ 0 < z < r_0$, 
\begin{align*}
 \int_0^{\min\{r_0,y\}} \frac{s^{-1/2}}{\sqrt{y-s}} ds  \lesssim \min\left\{1,\sqrt{\frac{r_0}{y}} \right\},\qquad
 \int_z^r\frac{s^{-1/2}}{\sqrt{s-z}} ds  \lesssim 1+ \log(r_0/z). 
\end{align*} 
So we multiply by $\omega^{1/2}$ and divide by $a$ (to recreate the right-hand side as in the $\mu=1/6$ case) and we end up with the factor 
$\frac{\omega(p)^{1/2}\omega(p)^{3/2}}{a(p)}
    \lesssim\omega(p)^{1/3}$.   
The $r_{\pm}$-estimates from the Appendix imply
$$\min\{r_\pm(p),2\pi-r_\pm(p)\}  \gtrsim \omega(p)^3. $$
So, using this in the second line below and treating the contributions near $0$\footnote{The contributions near $2\pi$ can be treated identically after the change of variables $q \mapsto 2\pi-q$, with $r_-$ and $r_+$ only replaced by $2\pi-r_+$ and $2\pi-r_-$, respectively.} 
\begin{align*} 
\omega^{1/2} &|K_1 (\mathbf{1}_{\mathcal{E}_{r_0} \cap [0,\pi)}h)|  
\lesssim \omega^{1/3}a(p) \|\omega^{1/2}h\|_{L^\infty_p} 
\left[ \min\left( 1,\sqrt{\frac{r_0}{r_{-}(p)}} \right) + \mathbf{1}_{\{r_+(p)<r_0\}} \left(1+ \log\left(\frac{r_0}{r_{+}(p)} \right) \right)  \right]\\
& \lesssim a(p) \|\omega^{1/2}h\|_{L^\infty_p} \left[ r_-(p)^{1/9}\min\left( 1,\sqrt{\frac{r_0}{r_{-}(p)}} \right) + r_+(p)^{1/9} \mathbf{1}_{\{r_+(p)<r_0\}} \left(1+ \log\left(\frac{r_0}{r_{+}(p)} \right) \right)  \right] \\& 
\lesssim a(p) r_0^{1/9} \|\omega^{1/2}h\|_{L^\infty_p}. 
\end{align*} 
To see the final scaling, we start with the first term: if $r_- \leq r_0, $ then $r_{-}^{1/9}\leq r_0^{1/9}$. If $r_- > r_0$, then we bound by $ r_-^{1/9} (r_0/r_-)^{1/2} \leq r_0^{1/2} r_{-}^{-7/18} \leq r_0^{1/2}r_0^{-7/18} \leq r_0^{1/9}$. For the second term, we present only when $0<r_+<r_0$, set $x=r_+/r_0$ and use
$\sup_{0<x<1}x^{1/9}(1+\log(1/x))<\infty$ for $0<x<1$.
 
So eventually we have 
$$\omega^{1/2}
    |K_1(\mathbf1_{\mathcal{E}_{r_0}}h)|
    \lesssim  r_0^{1/9} a \|\omega^{1/2}h\|_{L^\infty_p}.
$$
Together with the estimate for $\mu=1/6$, we obtain
\begin{equation} \label{eq:kernel_splitting_K1_edge}
\omega^\mu |K_1(\mathbf1_{\mathcal{E}_{r_0}}h)|
 \lesssim a\|\omega^\mu h\|_{L^\infty_p}
\begin{cases}
        r_0^{1/3}, &\mu=1/6,\\
        r_0^{1/9}, &\mu=1/2.
\end{cases}
\end{equation}

We turn now to the $K_2$-contribution: 
The only singularities of its kernel occur near
$(p,q)=(0,2\pi)$ and $(2\pi,0)$. 
Fix small neighborhoods of these two points,
independently of $r_0$. Near either point, write
$\epsilon = \omega(p)$ and $s =\omega(q)$.
The calculation in the proof of
\cite[Lemma 8(i)]{GLM26}
bounds the absolute value of the kernel by
$$ C \frac{\epsilon^2s^2}{\epsilon^2+s^2}
       \frac1{\sqrt{(\epsilon^2-s^2)^2+\epsilon s}}
    \lesssim
    \begin{cases}
        \epsilon^{-1/2}s^{3/2}, &s\leq \epsilon,\\
        \epsilon^{3/2}s^{-1/2}, &s\geq \epsilon.
    \end{cases}
$$

If $\epsilon \geq r_0$: we multiply by
$\omega(p)^\mu/a(p)$ and use the bound
$|h(q)|\leq s^{-\mu}\|\omega^\mu h\|_{L^\infty_p}$, as $0<s=\omega(q)<r_0\leq\epsilon$ for
$q\in\mathcal{E}_{r_0}$, which 
give the factor
$$
\epsilon^{\mu - 5/3}
  \int_0^{r_0}
    \epsilon^{-1/2}s^{3/2-\mu} ds
    \lesssim
    \epsilon^{\mu-13/6} r_0^{5/2-\mu}
    \leq r_0^{1/3}.
$$
The last inequality holds because $\mu<13/6$.  

If $\epsilon< r_0$, split the integral at $s=\epsilon = \omega(p)$. 
The resulting factor is bounded by
\begin{align*}
    \epsilon^{\mu-5/3} & 
    \left[
        \epsilon^{-1/2}\int_0^\epsilon s^{3/2-\mu} ds + \epsilon^{3/2}\int_\epsilon^{r_0} s^{-1/2-\mu} ds
    \right]\\
    &\hspace{4cm}\lesssim
    \epsilon^{1/3} +
    \epsilon^{\mu-1/6}\int_\epsilon^{r_0}s^{-1/2-\mu} ds.
\end{align*}
For $\mu=1/6$, the last integral is at most
$3r_0^{1/3}$, while for $\mu=1/2$, the last term is
$ \epsilon^{1/3}\log(r_0/\epsilon)
    \lesssim r_0^{1/3}$. 
Thus both cases give the same bound $r_0^{1/3}$. 
Outside fixed neighborhoods of $(0,2\pi)$ and $(2\pi,0)$, $F_+$ is bounded below. The kernel formula and
\cite[Lemma 19]{GLM26} therefore give $|k_2(p,q)|\lesssim \omega(p)^2$.  Its contribution with input in $\mathcal{E}_{r_0}$
then satisfies
\begin{align*}
\omega(p)^\mu \int_{\mathcal{E}_{r_0}}
    \omega(p)^2|h(q)| dq
    &\lesssim \omega(p)^{\mu+2}r_0^{1-\mu}
        \|\omega^\mu h\|_{L^\infty_p} \lesssim
    a(p) r_0^{1/3} \| \omega^\mu h\|_{L^\infty_p}.
\end{align*}
since $\mu+2 \geq 5/3$ and $1-\mu\ge1/3$.
Altogether these estimates give for both values of $\mu$: 
$$  \omega^\mu| K_2(\mathbf{1}_{\mathcal{E}_{r_0}}h)| \lesssim
    a r_0^{1/3}\|\omega^\mu h\|_{L^\infty_p}.
$$
It remains to estimate the input in $\mathcal{}B_{r_0}$.
Let $c'=c/(c-1)<2$.
H\"older's inequality gives us, uniformly in $0<y<2\pi$:
\begin{align*}
\int_0^y \frac{|\mathbf{1}_{\mathcal{B}_{r_0}}h(q)|}
     {\sqrt{y-q}} dq
    &\leq
    \left(\int_0^y(y-q)^{-c'/2}\,dq\right)^{1/c'}
    \|\mathbf{1} _{\mathcal B_{r_0}}h\|_{L^c_p} \lesssim_c\|\mathbf1_{\mathcal B_{r_0}}h\|_{L^c_p}.
\end{align*}
The integral from $y$ to $2\pi$ has the same bound. Thus
$$\omega^\mu|K_1(\mathbf1_{\mathcal B_{r_0}}h)|
    \lesssim_c 
\omega^{\mu+3/2}\|\mathbf{1}_{\mathcal{B}_{r_0}}h\|_{L^c_p} \lesssim_c
    a \|\mathbf{1}_{\mathcal{B}_{r_0}} h \|_{L^c_p}.
$$

The preceding bounds for $K_2$ also show that its kernel is bounded globally by
$ C \omega(p)^{3/2} \omega(q)^{-1/2}$.
Indeed, when $s \leq \epsilon$,
$ \epsilon^{-1/2}s^{3/2}  \leq \epsilon^{3/2}s^{-1/2}
$. 
The other case already has this form.
Outside the two fixed neighborhoods,
$\omega(p)^2\leq \omega(p)^{3/2} \omega(q)^{-1/2}$.

Now since $c'< 2$, we have $\omega^{-1/2}\in L^{c'}_p$ and applying again H\"older's gives
\begin{align*}
    \omega(p)^\mu
    |K_2(\mathbf{1}_{\mathcal{B}_{r_0}}h)(p)|
    &\lesssim
    \omega(p)^{\mu+3/2}
    \int_{\mathcal{B}_{r_0}}
        \omega(q)^{-1/2} |h(q)| dq \lesssim_c a(p)\|\mathbf{1}_{\mathcal{B}_{r_0}}h\|_{L^c_p}.
\end{align*}
where we used that  $\mu+3/2\geq 5/3$ and $a \sim \omega^{5/3}$.  
Combining everything and using $K=-K_1+2K_2$,
we conclude that, for $0<r_0<1$: 
$$
 \omega(p)^\mu |Kh(p)|
    \leq a(p)\left[
C r_0^{1/9}\|\omega^\mu h\|_{L^\infty_p}
+C_c\|\mathbf{1}_{\mathcal{B}_{r_0}}h\|_{L^c_p}
\right].
$$ 
Choosing $r_0$ sufficiently small that
$Cr_0^{1/9} \leq \eta$ proves the lemma.
\end{proof}

We next combine the $L^2_p$ decay estimates with the smoothing estimate \ref{lem:app_Kernelbound_edge} in the Appendix to obtain $L^c_p$ decay
on any fixed bulk, for $c>2$.

\begin{lemma}\label{lem:bulkdecay_linear}
Fix $2<c<\infty$ and a radius $0<r_0<1$.  Let $g(t)=e^{tB(\xi)}f$ and $\Pi f=0$ with  $ \|\omega^{-1/2}f\|_{L^\infty_p} <\infty.$ Then for each fixed $\xi\neq 0$: 
\begin{equation} \label{eq:bulk_decay_linear}
\sup_{t\geq 0}\ \langle t \rangle^\rho
\|\mathbf{1}_{\mathcal B} e^{tB(\xi)}f\|_{L^c_p}
 \lesssim_{\rho,c,r_0,\xi_*} 
    \|\omega^{-1/2}f\|_{L^\infty_p} +
    \mathbf{1}_{\{0<|\xi|<\xi_*\}}
|\xi|^{-2\rho/3}\|f\|_{\mathscr{Y}}.
\end{equation} 
\end{lemma} 

\begin{proof} We have $  a_{\mathcal B}:= \operatorname*{ess\,inf}_{p\in\mathcal B}a(p)\gtrsim r_0^{5/3}>0 $,  
since $a\gtrsim\omega^{5/3}$.   
Choose $\varepsilon>0$ be fixed in terms of $c$ so that $\varepsilon c > 2$.
The appendix smoothing estimate in Lemma \ref{lem:app_Kernelbound_edge}, applied with $R=c$, gives
$\|\omega^{-3/2+\varepsilon}Kh\|_{L^c_p}    \lesssim_{\varepsilon,c}\|h\|_{L^2_p}$.  
On $\mathcal B$, the multiplier
$\omega^{3/2-\varepsilon}$ is bounded. Thus,
\begin{align} \label{eq:bulk_gain_L2_Lc}
    \|\mathbf{1}_{\mathcal{B}} K h \|_{L^c_p}
    \leq
    \|\mathbf{1}_{\mathcal B}
        \omega^{3/2-\varepsilon} \|_{L^\infty_p}
    \|\omega^{-3/2+\varepsilon} Kh\|_{L^c_p}
    \lesssim_{r_0,c}\|h\|_{L^2_p}.
\end{align}

The Duhamel formula for 
$g(t) = e^{tB(\xi)}f$ for $B(\xi)=-a-i\xi v+K$ is
\begin{equation}  \label{eq:linear_free_Duhamel}
  g(t) =   e^{-t(a+i\xi v)}f +
    \int_0^t
        e^{-(t-s)(a+i\xi v)}Kg(s) ds.
\end{equation}
The transport factor has modulus one: $ |e^{-i\xi v(p)(t-s)}|=1$. So taking the $L^c_p$ norm on the bulk and using 
\eqref{eq:bulk_gain_L2_Lc}, yields 
\begin{align} \label{eq:bulk_Duhamel_Lc}
\|\mathbf{1}_{\mathcal{B}}g(t)\|_{L^c_p}
    &\leq 
    e^{-a_{\mathcal{B}}t}
        \|\mathbf{1}_{\mathcal{B}} f\|_{L^c_p}
    + \int_0^t e^{-a_{\mathcal{B}}(t-s)}
  \|\mathbf{1}_{\mathcal{B}} Kg(s)\|_{L^c_p} ds
    \notag \\
    &\leq
    e^{-a_{\mathcal{B} }t}
        \|\mathbf1_{\mathcal{B}}f\|_{L^c_p}
 + C_{r_0,c}\int_0^t
  e^{-a_{\mathcal B}(t-s)}\|g(s)\|_{L^2_p} ds.
\end{align}

Now we may insert our small- and large-frequency estimates \eqref{eq:small_frequency_L2_decay} and \eqref{eq:large_frequency_L2_decay}, together with 
$ \|\omega^{-\theta} f \|_{L^2_p}   \lesssim_\rho\|\omega^{-1/2}f\|_{L^\infty_p}.
$
This last estimate we see if we write 
$\|\omega^{-\theta} f \|_{L^2_p}^2 \leq
    \|\omega^{-1/2}f\|_{L^\infty_p}^2
    \int_{\mathbb{T} }\omega^{1-2\theta} dp$, since  this last integral is finite.
And so for every $\xi \neq 0$,
\begin{equation} \label{eq:bulk_input_L2_decay}
    \| g(s) \|_{L^2_p}
    \lesssim \langle s \rangle^{-\rho}
    \left(
        \|\omega^{-1/2}f\|_{L^\infty_p}
        +\mathbf{1}_{\{0<|\xi|<\xi_*\}}
        |\xi|^{-2\rho/3}\|f\|_{\mathscr{Y}}
    \right).
\end{equation}
Altogether, taking into account the  elementary convolution bound
\begin{equation} \label{eq:bulk_exponential_convolution}
    \int_0^t e^{-a_{\mathcal B}(t-s)}
        \langle s\rangle^{-\rho}\,ds
    \lesssim_{\rho,a_{\mathcal B}}
        \langle t\rangle^{-\rho}, 
\end{equation}
and that  $ \|\mathbf{1}_{\mathcal{B}}f\|_{L^c_p}
     \lesssim_c \|\omega^{-1/2}f\|_{L^\infty_p}$, implies the claim. 
\end{proof}

\medskip

Now Lemmas \ref{lem:kernel_splitting_linear} and \ref{lem:bulkdecay_linear} imply a decay in the weighted pointwise topology: 
 
\begin{lemma}[Weighted pointwise decay] 
Let $\Pi f=0$ and $\|\omega^{-1/2} f\|_{L^\infty_p}<\infty$. We set $b_{1/6}:= 2/5$ and $b_{1/2}:= \rho$. Then for  $\mu\in \{1/6 , 1/2\}$ and every $\xi \neq 0$, we have   
\begin{equation} \label{eq:pointwise_mode_decay_linear}
\sup_{t\geq 0}\ \langle t \rangle^{b_\mu}
 \|\omega^\mu e^{tB(\xi)}f\|_{L^\infty_p}
 \lesssim_\rho 
 \|\omega^{-1/2}f\|_{L^\infty_p}+
\mathbf{1}_{\{0<|\xi|<\xi_*\}}
    |\xi|^{-2\rho/3}\|f\|_{\mathscr{Y}}.
\end{equation}
\end{lemma} 
\begin{proof} 
We write  $ x_\mu(t):=\|\omega^\mu g(t)\|_{L^\infty_p}.$  Since $|f(p)|\leq \omega(p)^{1/2}
\|\omega^{-1/2}f\|_{L^\infty_p}$ and the transport phase has modulus one,
$$ \|\omega^\mu e^{-t(a+i\xi v)}f\|_{L^\infty_p}
    \leq
    \|\omega^{-1/2}f\|_{L^\infty_p}
    \sup_p\omega(p)^{\mu+1/2}e^{-ta(p)} \lesssim 
    \langle t\rangle^{-\gamma_\mu}
    \|\omega^{-1/2}f\|_{L^\infty_p},
    \quad
    \gamma_\mu:=\frac{3}{5} \left(\mu+\frac{1}{2} \right).
$$

We are going to use the following uniform convolution estimate: for every $0< b< 1$, there is a constant $C_b$ such that
\begin{equation} \label{eq:weighted_time_convolution}
 \sup_{\alpha \geq 0}  \int_0^t
\alpha e^{-\alpha(t-s)}\langle s\rangle^{-b} ds
    \leq C_b\langle t \rangle^{-b},
    \qquad t\geq 0.
\end{equation}

We choose a bulk, i.e. $r_0$, as follows: 
Fix $2<c<\infty$ and choose $\eta>0$ sufficiently small so that the constants from
\eqref{eq:weighted_time_convolution} satisfy 
$$ \eta\max\{C_{2/5},C_\rho\}\leq\frac{1}{2}. 
$$

Then choose a single $r_0>0$ for which
Lemma \ref{lem:kernel_splitting_linear} holds for both
$\mu=1/6$ and $\mu=1/2$ with this coefficient $\eta$.  
We apply the preceding fixed-bulk argument for this particular choice of $r_0$:
\begin{equation}\label{eq:weighted_pointwise_bulk_input}
\|\mathbf{1}_{\mathcal B}g(s)\|_{L^c_p}
\leq C_{\rm bulk} \langle s \rangle^{-\rho}
 \Big[
    \|\omega^{-1/2}f\|_{L^\infty_p} +
        \mathbf{1}_{\{0<|\xi|<\xi_*\}}
   |\xi|^{-2\rho/3}\|f\|_{\mathscr{Y}}
    \Big].
\end{equation}
Then from Duhamel \eqref{eq:linear_free_Duhamel} and \eqref{eq:kernel_splitting_linear}, we get 
\begin{align}
    x_\mu(t)
    \leq 
    C\langle t\rangle^{-\gamma_\mu}&
        \|\omega^{-1/2}f\|_{L^\infty_p}
  + \eta\operatorname*{ess\,sup}_{p}
    \int_0^t
 a(p)e^{-a(p)(t-s)}x_\mu(s) ds
    \notag\\
    &\hspace{2cm}+
    C_{\eta,c,r_0}\operatorname*{ess\,sup}_{p}
    \int_0^t
        a(p)e^{-a(p)(t-s)}
\|\mathbf{1}_{\mathcal{B}}g(s)\|_{L^c_p} ds.
\label{eq:weighted_pointwise_Duhamel}
\end{align}

For $0\leq s\leq t\leq T$, we write 
$ x_\mu(s) \leq
    \langle s \rangle^{-b_\mu}
    \sup_{0\leq \tau\leq T}
        \langle \tau \rangle^{b_\mu} x_\mu(\tau)$. 
Using \eqref{eq:weighted_time_convolution}, we therefore get
\begin{align*}
    &\operatorname*{ess\,sup}_{p}
    \int_0^t a(p) e^{-a(p)(t-s)}x_\mu(s)\,ds \leq
    C_{b_\mu}\langle t\rangle^{-b_\mu}
    \sup_{0\leq \tau\leq T}
        \langle \tau\rangle^{b_\mu} x_\mu(\tau).
\end{align*}

For the bulk term,
\eqref{eq:weighted_pointwise_bulk_input} and
\eqref{eq:weighted_time_convolution}, with $b=\rho$, give
\begin{align*}
    &\operatorname*{ess\,sup}_{p}
    \int_0^t a(p)e^{-a(p)(t-s)}
        \|\mathbf1_{\mathcal B}g(s)\|_{L^c_p}\,ds \lesssim
    \langle t\rangle^{-\rho}
    \left(
        \|\omega^{-1/2}f\|_{L^\infty_p}
        + \mathbf{1}_{\{0<|\xi|<\xi_*\}}
        |\xi|^{-2\rho/3}\|f\|_{\mathscr{Y}}
    \right).
\end{align*}

Since $b_{1/6}=2/5$ and $b_{1/2}=\rho$ we have 
$b_\mu \leq \min\{\gamma_\mu,\rho\}$. We multiply  
\eqref{eq:weighted_pointwise_Duhamel} by
$\langle t \rangle^{b_\mu}$ and take  the supremum over
$0\leq t\leq T$, which yields
\begin{align*}
    \sup_{0\leq t\leq T}
        \langle t \rangle^{b_\mu}x_\mu(t)
    \leq &
    C\left(
        \|\omega^{-1/2}f\|_{L^\infty_p}
        +
        \mathbf{1}_{\{0<|\xi|<\xi_*\}}
    |\xi|^{-2\rho/3}\|f\|_{\mathscr{Y}}
    \right)\\
    &\hspace{4cm}+
    \eta C_{b_\mu}
    \sup_{0\leq t\leq T}
        \langle t \rangle^{b_\mu} x_\mu(t).
\end{align*}
Our choice $\eta C_{b_\mu}\leq 1/2$ allows us to subtract the last term and obtain a bound independent of $T$.
Letting $T \to\infty$, we conclude  \eqref{eq:pointwise_mode_decay_linear} for $\xi\neq 0$.    
Of course, in order for the last steps to make sense we need to make sure that $x_\mu(t)$ is finite on every bounded time interval. We leave below a proof of this fact. 

The kernel estimate in Lemma \ref{lem:kernel_splitting_linear} with $\mu=1/6$ gives
$$  \|\omega^{1/6}Kh\|_{L^\infty_p}
    \leq \|a\|_{L^\infty_p}
    \left( \eta\|\omega^{1/6}h\|_{L^\infty_p}
    +C_{\eta,c,r_0}
    \|\mathbf{1}_{\mathcal B}h\|_{L^c_p}
    \right).
$$ 
On the fixed bulk, 
$$
\|\mathbf{1}_{\mathcal{B}} h\|_{L^c_p}
    \leq |\mathbb{T}|^{1/c}r_0^{-1/6}
        \|\omega^{1/6}h\|_{L^\infty_p}.
$$
Thus $K$ is bounded on the Banach space with norm
$\|\omega^{1/6}h\|_{L^\infty_p}$. 
Multiplication by $a+i\xi v$ is also bounded there for each fixed $\xi$. So its bounded-operator evolution exists on every finite time interval. 
This weighted space embeds continuously into $L^2_p$ since 
$\|h\|_2^2 \leq
    \|\omega^{1/6}h\|_{L^\infty_p}^2
    \int_{\mathbb{T}}\omega^{-1/3} dp
    <\infty.$  
The weighted-space evolution therefore agrees with 
$e^{tB(\xi)}f$ by uniqueness in $L^2_p$.
The initial datum belongs to this space because 
$$
  \|\omega^{1/6}f\|_{L^\infty_p}
    \le\|\omega^{-1/2}f\|_{L^\infty_p}.
$$
so  $x_{1/6}(t)$ is bounded on every finite
time interval. Since $\omega^{1/2}\leq\omega^{1/6}$,
the same holds for $x_{1/2}(t)$.
\end{proof}

\medskip

\begin{proof}[Proof of Theorem \ref{thm:linear}] 
Let us first ensure that our assumptions imply that $g_0 \in L^2_{x,p}$ where our  semigroup $e^{tB(\xi)}$ is well-defined and continuous. 
First for $r=2$ we use the fact that  $\|f\|_{L^2_p}\lesssim\|\omega^{-1/2}f\|_{L^\infty_p}$
and Plancherel, since then for $f=\widehat{g}_0$ $$
\| \widehat{g}_0\|_{L^2_{\xi,p}}^2 \lesssim \int_{\xi \in \mathbb{R}} \|\omega^{-1/2}\widehat{g}_0(\xi)\|_{L^\infty_p}^2 d\xi := X_{-1/2,2}[g_0]^2 <\infty.
$$ 
 For $r=2$: we have 
$ \|f\|_{L^2_p}^2 \leq \|\omega^{-1/2}f\|_{L^\infty_p} \int_{\mathbb{T}}\omega^{1/2}|f| dp \lesssim \|\omega^{-1/2}f\|_{L^\infty_p} \|f\|_{\mathscr{Y}}.  $ Take $f = \widehat{g}_0(\xi,\cdot)$ and integrate over $\xi$: 
$$
\|\widehat{g}_0\|_{L^2_{\xi,p}}^2
\lesssim \int_{\xi \in \mathbb{R}} 
\|\omega^{-1/2}\widehat{g}_0(\xi)\|_{L^\infty_p}
\|\widehat{g}_0(\xi)\|_{\mathscr{Y}}\ d\xi \leq
\|\widehat{g}_0\|_{ L^\infty_\xi \mathscr{Y} }
\int_{\mathbb{R}} \|\omega^{-1/2}\widehat{g}_0 (\xi)\|_{L^\infty_p} d \xi 
=: \|\widehat{g}_0\|_{ L^\infty_\xi \mathscr{Y} } X_{-1/2,1}[g_0]<\infty. 
$$
thus again by Plancherel, $g_0 \in L^2_{x,p}$.

Since $\Pi$ acts only in $p$, it commutes with the spatial Fourier transform. Thus,
$$ \Pi\widehat g_0(\xi,\cdot) =
    \widehat{\Pi g_0}(\xi,\cdot) =0
$$
for almost every $\xi$. The orthogonality condition required by the mode-by-mode decay estimates is therefore satisfied. 
and the preceding estimates apply with
$f=\widehat g_0(\xi,\cdot)$: 
For $\mu \in \{1/6,1/2\}$, they give
\begin{align}
    \|\omega^\mu\widehat g(t,\xi)\|_{L^\infty_p}
    \lesssim_\rho &
    \langle t \rangle^{-b_\mu}
    \Big[
     \|\omega^{-1/2}\widehat g_0(\xi)\|_{L^\infty_p}
+  \mathbf{1}_{\{0<|\xi|<\xi_*\}}
        |\xi|^{-2\rho/3}
        \|\widehat g_0(\xi)\|_{\mathscr{Y} }
    \Big],
    \label{eq:Fourier_pointwise_input}
\end{align}
where
$b_{1/6}=\frac{2}{5}, b_{1/2}=\rho,$ and $\frac{1}{2}< \rho< \frac{3}{5}.$
\\

For $r=1,2$, we have $2\rho r/3<1$,  since $\rho < 3/5$. 
Taking now the $L^r_\xi$ norm in 
\eqref{eq:Fourier_pointwise_input}, using the triangle
inequality and that 
\begin{align}\label{eq:Fourier_singular_factor_bound}
 \left\| \mathbf{1}_{\{0<|\xi|<\xi_*\}}
 |\xi|^{-2\rho/3}
\|\widehat{g}_0(\xi)\|_{\mathscr{Y}}  \right\|_{L^r_\xi}
    \leq \|\widehat{g}_0 \|_{L^\infty_\xi\mathscr{Y}}
\left(\frac{2\xi_*^{1-2\rho r/3}}
 {1-2\rho r/3} \right)^{1/r} \lesssim_{\rho,\xi_*} \|\widehat{g}_0\|_{L^\infty_\xi \mathscr{Y}}, 
\end{align}
gives
\begin{equation} \label{eq:Fourier_weighted_decay}
 X_{\mu,r}[S(t)g_0]
 \lesssim_\rho 
    \langle t\rangle^{-b_\mu}
 \Big( X_{-1/2,r}[g_0]+\|\widehat{g}_0\|_{L^\infty_\xi\mathscr{Y}} \Big), \quad r=1,2.
\end{equation}

Now in order to get a weighted pointwise decay in physical space we use that for almost every $p$, Fourier inversion and the triangle
inequality give
$$ |\omega(p)^\mu g(t,x,p)| \leq
 \frac{1}{2\pi}\int_{\mathbb{R}}
 |\omega(p)^\mu\widehat g(t,\xi,p)| d\xi.
$$
Taking the essential supremum in $(x,p)$ yields 
\begin{align*}
    \|\omega^\mu g(t)\|_{L^\infty_{x,p}}
    \leq
    \frac1{2\pi}\int_{\mathbb{R}}
        \|\omega^\mu\widehat g(t,\xi)\|_{L^\infty_p}\ d\xi = \frac{1}{2\pi} X_{\mu,1}[g(t)].
\end{align*}
which is finite by
\eqref{eq:Fourier_weighted_decay}.
Thus 
\begin{equation}
    \|\omega^\mu S(t)g_0\|_{L^\infty_{x,p}}
    \lesssim_\rho
    \langle t\rangle^{-b_\mu}
    \big[ X_{-1/2,1}[g_0]+\|\widehat{g}_0\|_{L^\infty_\xi\mathscr{Y}}
    \big].
    \label{eq:physical_weighted_decay}
\end{equation}
For the unweighted decay estimate \eqref{eq:decay_L2}:   
We use the preceding $L^2_p$ bounds from the application of the abstract Lemma \ref{lem:abstract_energy} 
directly. Combining the small- and large-frequency 
estimates gives, for a.e. $\xi$,
\begin{align*}
    \|\widehat g(t,\xi)\|_{L^2_p}
    \lesssim_\rho
    \langle t\rangle^{-\rho}
    \Big[ 
    \|\omega^{-\theta} 
    \widehat{g}_0(\xi) \|_{L^2_p} +
        \mathbf{1}_{\{0<|\xi|<\xi_*\}}
        |\xi|^{-2\rho/3}
        \|\widehat g_0(\xi)\|_{\mathscr{Y}}
    \Big]. 
\end{align*}
Since $\theta < 1$,
$ \|\omega^{-\theta} 
 \widehat{g}_0(\xi) \|_{L^2_p} 
 \lesssim_\rho \|\omega^{-1/2} \widehat{g}_0 (\xi)\|_{L^\infty_p}.$ 
Taking now $L^2_\xi$ and applying
\eqref{eq:Fourier_singular_factor_bound} with $r=2$,
we obtain
$$ \|\widehat g(t)\|_{L^2_{\xi,p}}
 \lesssim_\rho
 \langle t\rangle^{-\rho}
\big[X_{-1/2,2}[g_0]+\|\widehat{g}_0\|_{L^\infty_\xi\mathscr{Y}} \big], 
$$ 
which by Plancherel implies the $L^2$-unweighted decay  
\begin{equation}   \label{eq:physical_L2_decay}
\|S(t)g_0\|_{L^2_{x,p}} \lesssim_\rho
 \langle t\rangle^{-\rho}\big[ 
 X_{-1/2,2}[g_0]+\|\widehat{g}_0\|_{L^\infty_\xi\mathscr{Y}}
 \big].
\end{equation}

Finally to derive the exact stated decay rates in \eqref{eq:decay_weight6}, \eqref{eq:decay_weight-2}:  
given $0<\delta<1/10$, choose
$\rho=\frac{3}{5}-\delta \in (1/2,3/5)$.  
Equations \eqref{eq:Fourier_weighted_decay} and
\eqref{eq:physical_weighted_decay} now give for $r=1,2$, 
\begin{align*}
X_{1/6,r}[S(t)g_0]
&\lesssim_\delta  \langle t \rangle^{-2/5}
\big( X_{-1/2,r}[g_0]+\|\widehat{g}_0\|_{L^\infty_\xi\mathscr{Y}}\big),\\
 X_{1/2,r}[S(t)g_0]
 &\lesssim_\delta
\langle t\rangle^{-3/5+\delta} \big( X_{-1/2,r}[g_0]+\|\widehat{g}_0\|_{L^\infty_\xi\mathscr{Y}} \big). 
\end{align*}
\end{proof}

\section{Nonlinear Stability} \label{sec:nonlinear} 
We recall the norm notation we are using: 
$\|h\|_{\mathcal{X}} : =X_{1/6,1}[h]+X_{1/6,2}[h]$ and $\mathscr{Y}=L^2(a^{-1}dp)$.
 In this section we are going to prove the following theorem:

\begin{theorem}[Nonlinear stability]
\label{thm:nonlinear_stability}
Fix $0<\delta<1/10$. Suppose that $g_0\in\mathcal{X}$ with 
$ \Pi g_0(x,\cdot)=0$ for almost every $x$, 
and
$$ \varepsilon :=X_{-1/2,1}[g_0]+X_{-1/2,2}[g_0]+\|\widehat{g}_0\|_{L^\infty_\xi\mathscr{Y}} 
$$
is sufficiently small. Then \eqref{eq:nonlinear_eq} has a
unique global mild solution $g\in C([0,\infty);\mathcal{X})$.
For $r=1,2$ and $t\geq 0$,
\begin{align}
    X_{1/6,r}[g(t)]
 \lesssim_\delta \varepsilon  \langle t\rangle^{-2/5}&,
  \qquad 
X_{1/2,r}[g(t)]
\lesssim_\delta \varepsilon\langle t\rangle^{-3/5+\delta}, \text{and } 
     \label{eq:nonlinear1/2}\\
    \|g(t)\|_{L^2_{x,p}}
       &\lesssim_\delta \varepsilon\langle t\rangle^{-3/5+\delta}.
       \label{eq:nonlinear_L2}
\end{align}
In particular,
$$ \|\omega^{1/6}g(t)\|_{L^\infty_{x,p}}
\lesssim_\delta\varepsilon\langle t\rangle^{-2/5},
    \quad \text{ and }\quad
    \|\omega^{1/2}g(t)\|_{L^\infty_{x,p}}
       \lesssim_\delta\varepsilon\langle t\rangle^{-3/5+\delta}.
$$
\end{theorem}

Before we proceed with the proof, we  recall that $g_j=g(p_j)$ and $\mathfrak{f}_j=\mathfrak{f} (p_j)$, where
$p_1,p_3$ are determined by the nontrivial branch in
\eqref{eq:parametrisation_nontrivial_branch}. The integral  kernel is 
$$ \sigma_{02} =\frac{\omega_0\omega_1\omega_2\omega_3} {\sqrt{F_+(p_0,p_2)}}, 
$$
and the nonlinearities are 
\begin{equation} \label{eq:nonlinearities_Gamma_Q}
\begin{split}
& \Gamma(g,g)(p_0) 
 = \frac{1}{\mathfrak{f}_0}
 \int_{\mathbb{T} }\sigma_{02}
 \big[ 
\mathfrak{f}_2\mathfrak{f}_3(
\mathfrak{f}_0 + \mathfrak{f}_1) g_2g_3
-\mathfrak{f}_0\mathfrak{f}_1
(\mathfrak{f}_2+\mathfrak{f}_3) g_0g_1 \\
 &\hspace{3.4cm}
   +\mathfrak{f}_0\mathfrak{f}_2(\mathfrak{f}_3-\mathfrak{f}_1)g_0g_2
   +\mathfrak{f}_0\mathfrak{f}_3(\mathfrak{f}_2-\mathfrak{f}_1)g_0g_3\\
 &\hspace{3.8cm}
   +\mathfrak{f}_1\mathfrak{f}_2(\mathfrak{f}_3-\mathfrak{f}_0)g_1g_2
   +\mathfrak{f}_1\mathfrak{f}_3(\mathfrak{f}_2-\mathfrak{f}_0)g_1g_3
 \big] dp_2,\\
 & Q(g,g,g)(p_0)
 = \frac{1}{\mathfrak{f}_0}
 \int_{\mathbb{T}}\sigma_{02}
 \big[
   \mathfrak{f}_1\mathfrak{f}_2\mathfrak{f}_3 g_1g_2g_3
+\mathfrak{f}_0\mathfrak{f}_2\mathfrak{f}_3g_0g_2g_3\\
 &\hspace{4cm}
 -\mathfrak{f}_0\mathfrak{f}_1\mathfrak{f}_3g_0g_1g_3 -\mathfrak{f}_0\mathfrak{f}_1\mathfrak{f}_2g_0g_1g_2
 \big] dp_2.
 \end{split}
\end{equation}
As we are going to take the spatial Fourier transform of these nonlinearities, we will need to understand $\Gamma, Q$ for distinct inputs\footnote{This is since for the quadratic nonlinearity for example it holds  $\widehat{\Gamma(g,g)}(\xi)
=(2\pi)^{-1} \int_{\mathbb{R}}
\Gamma\big(\widehat{g}(\eta),
     \widehat{g}(\xi-\eta)\big) d\eta$.}. 
For distinct inputs $f^1,f^2,f^3$, we understand $\Gamma$ and $Q$ as the symmetric bilinear and trilinear forms associated with these expressions.  In that case, in $\Gamma$ each product $g_jg_k$ in \eqref{eq:nonlinearities_Gamma_Q}  is replaced by the symmetrisation 
$(f_j^{1}f_k^{2}+f_j^{2} f_k^{1})/2$. In $Q$ similarly each cubic product is averaged over
the six permutations of its three inputs.  Also since $\widehat{g}(\xi, p)$ is generally complex-valued let us also note that we use the same formulas for complex-valued inputs. 

We have the following a-priori estimates:  
 \begin{lemma}[A priori bounds]
There exist finite constants $C_1, C_2 >0$ so that 
 \begin{align}
 & \|\omega^{-1/2}\Gamma(f,h)\|_{L^\infty_p}
 \leq C_1\|\omega^{1/2} f\|_\infty\|\omega^{1/2}h\|_{L^\infty_p},
\label{eq:Gamma-pointwise}\\
 & \|\omega^{-1/2}Q(f_1,f_2,f_3)\|_{L^\infty_p} \leq C_2 \prod_{j=1}^3\|\omega^{1/6}f_j \|_{L^\infty_p}.
\label{eq:pointwise_apriori}
\end{align}
\end{lemma}
\begin{proof} 
To prove these estimates, use the bounds
$  F_+(p_0,p_2)\geq 4\omega_0\omega_2,$,  and $
    \omega_1\omega_2\omega_3\leq 4\omega_0$. The first follows directly from the definition of $F_+$. The second we get from \cite[Lemmas 17 and 19]{GLM26}. These imply that 
\begin{equation}
\omega_0^{-1/2} \sigma_{02}
 \leq \tfrac{1}{2} \omega_1\omega_2^{1/2} \omega_3.
 \label{eq:nonlinear-kernel-bound}
\end{equation}

Now for the quadratic term, in order to conclude we need to estimate all the integrands of the form $\omega_0^{-1/2} \sigma_{02} \omega_j^{-1/2}\omega_k^{-1/2}$ 
for all $j, k \in \{0,1,2,3\}$, $j \neq k$.  
So we divide by the input weights
$\omega_j^{1/2}\omega_k^{1/2}$ for each pair $(j,k)$. Observe that for pairs not containing $0$, all the remaining powers in
\eqref{eq:nonlinear-kernel-bound} are nonnegative and thus bounded. 
For the pairs $(0,1)$ and $(0,3)$, the remaining factors are
$$ \left(\frac{\omega_1\omega_2\omega_3}{\omega_0}\right)^{1/2}
 \omega_3^{1/2}, \quad
    \left(\frac{\omega_1\omega_2\omega_3}{\omega_0}\right)^{1/2}  \omega_1^{1/2},
$$
respectively, which are both bounded. For $(0,2)$, we keep the coefficient difference in the integrand $|\mathfrak{f}_3-\mathfrak{f}_1|$ 
in \eqref{eq:nonlinearities_Gamma_Q}. This helps since 
$$|\mathfrak{f}_3-\mathfrak{f}_1|
\lesssim |\omega_3 - \omega_1| =
|\omega_0 -\omega_2|, 
$$ 
due to the resonance condition.  
Thus the remaining factor is bounded (up to constants) by
$$ \omega_0^{-1/2}\omega_1 \omega_3 
|\omega_0 - \omega_2| \lesssim \omega_0^{1/2}.
$$
where for the second inequality: If $\omega_2 \leq \omega_0$, we just use that $\omega_1 \omega_3\leq 1$, while when $\omega_2 \geq \omega_0$, we use the product bound: $\omega_1 \omega_3\leq 4\omega_0/\omega_2$. Integration in $p_2$ and taking the supremum over $p_0$ gives the claimed bound \eqref{eq:Gamma-pointwise}. 

Next for the cubic nonlinearity, its three inverse input weights are bounded
by $\prod_{\ell = 0}^3\omega_\ell^{-1/6}$, since the additional factor is
at least one. The remaining factor satisfies
\begin{align*} \omega_0^{-1/6}\omega_1^{5/6}\omega_2^{1/3}\omega_3^{5/6} &= \left(\frac{\omega_1\omega_2\omega_3}{\omega_0}\right)^{1/6}
\omega_1^{2/3}\omega_2^{1/6}\omega_3^{2/3}
\lesssim 1.
\end{align*}
This proves \eqref{eq:pointwise_apriori}. 
\end{proof}

We also record the property, as also in the classical Boltzmann equation: 
\begin{equation} \label{eq:proj_nonlinearities_zero}
\Pi\Gamma(f,h)=0,\qquad
    \Pi Q(f_1,f_2,f_3) =0.
\end{equation}
For this, first conservation of mass and energy gives that for a bounded function $F$, the full
collision operator satisfies
$$
\int_{\mathbb{T}}
\begin{bmatrix}1\\ \omega \end{bmatrix}\mathcal C[F](p) dp=0.
$$

For a bounded real $h$, whenever $F= \mathfrak{f} (1 + \lambda h)$ and $\phi =\mathfrak{f}$ or $\omega\mathfrak{f}$, we get 
\begin{align*}
 0 = \left\langle
\mathfrak{f}^{-1}\mathcal{C}[\mathfrak{f} (1+\lambda h)],\phi
    \right\rangle = \lambda\langle L h, \phi\rangle
     +\lambda^2 \langle\Gamma(h,h),\phi\rangle
     +\lambda^3 \langle Q(h,h,h),\phi\rangle.
\end{align*}

Since this holds for every sufficiently small $\lambda$, each coefficient has to be $0$. We conclude then \eqref{eq:proj_nonlinearities_zero} since these two choices of $\phi$ span $\operatorname{ker} L$. This works for the diagonal case, i.e. for the same function $h$. For distinct entries, for example in the quadratic form, we may replace $h$ by $\lambda_1 f+\lambda_2 h$ in \eqref{eq:proj_nonlinearities_zero}, and compare the coefficients of $\lambda_1\lambda_2$. Thanks to symmetry, this coefficient is $2 \Pi \Gamma(f,h)$ and so we may conclude that $\Pi \Gamma(f,h)=0$.  
Similarly with the cubic form. 

Analogues of \eqref{eq:proj_nonlinearities_zero} hold
for the weighted inputs of \eqref{eq:Gamma-pointwise} and \eqref{eq:pointwise_apriori}: We only need to be careful as the weighted bounds allow the inputs to be unbounded near the edges $0,2\pi$. Thus we truncate each $f$ by $f^{(M)}=f  
\mathbf{1}_{ \{|f|\leq M\} }$, so that they are bounded and converge pointwise to $f$. 
The estimates just proved in the proof of \eqref{eq:Gamma-pointwise}, 
\eqref{eq:pointwise_apriori}, dominate the integrands uniformly in these truncation $M$: they are bounded by $\omega_0^{1/2}$.   
Thus dominated convergence applies and yields convergence for the $p_2$-integral where the limit is bounded by  $C\omega_0^{1/2}$, with $C$ uniform in $M$. Then dominated convergence applies again and yields convergence in $L^2_p$. 
The projection $\Pi$ is bounded on $L^2_p$, and so the cancellation identities  \eqref{eq:proj_nonlinearities_zero} pass to the truncation limit, indeed.    

 Note that this is a property of the collision operator at each fixed $x$.
It does not assert that $\Pi g(t,x,\cdot)=0$.
In fact the local moments satisfy 
$$ \partial_t\langle g,\phi\rangle_p
 +\partial_x\langle vg,\phi\rangle_p=0,
 \qquad \phi\in\ker L,
$$
so the transport term can change them!  

We continue with a local-in-time well-posedness result on the space $\mathcal{X}$. 

\begin{lemma}[Local well-posedness in $ \mathcal{X}$]
For every $g_0 \in \mathcal{X}$, equation \eqref{eq:nonlinear_eq} 
has a unique maximal mild solution
$g\in C([0,T_0);\mathcal{X})$. If $T_0<\infty$, then
\begin{equation}
\limsup_{t \uparrow T_0}
\|g(t)\|_{\mathcal{X}}=\infty.
    \label{eq:NL_blowup criterion}
\end{equation}
\end{lemma}

\begin{proof}
Consider the semigroup 
$G(t)f(x,p)=e^{-ta(p)}f(x-tv(p),p)$ which consists of the transport and the loss term of the collisional operator.  
This is a strongly continuous contraction 
on $\mathcal{X}$:   
Indeed, since $\sup_p |e^{-t(a + i\xi v)}-1| \to 0 $ as $ t\downarrow 0$, for
each fixed $\xi$.  
Dominated convergence in $L^1_\xi$ and $L^2_\xi$ proves then strong
continuity. 

Continuity of the semigroup persists under $K$ since from Lemma \ref{lem:kernel_splitting_linear} we have that $K$ is bounded on $\mathcal{X}$: A direct application of the Lemma gives us for any $\eta>0$, there exists a $r_0>0$ so that 
  $$|\omega^\mu K h| \leq a(p) \Big( \eta \|\omega^\mu h\|_{L^\infty_p} + C \|\mathbf{1}_{\mathcal{B}_{r_0}} h  \|_{L^c_p} \Big),$$
  where $c>2$ and $\mu = 1/6$.  
   Since $a$ is bounded and $$\|\mathbf{1}_{\mathcal{B}_{r_0}} h \|_{L^c_p} \lesssim r_0^{-\mu} \|\omega^{\mu} h \|_{L^\infty_p}, 
  $$
  we have $$\|\omega^{\mu} K  h \|_{L^\infty_p} \lesssim \|\omega^{\mu} h \|_{L^\infty_p}$$ or that $K$
 is bounded on $\omega^{-1/6}L^\infty_p$ and thus on 
$\mathcal{X}$ after taking the $L^1_\xi$ and $L^2_\xi$ norms, implying $ \|Kh\|_{\mathcal{X}} \lesssim  \|h\|_{\mathcal{X}}$.   

 The bounded perturbation theorem therefore gives a
strongly continuous semigroup on $\mathcal{X}$  for the full linear operator $- v\partial_x + L$.
Since $\mathcal{X}$ embeds continuously into $L^2_{x,p}$, uniqueness of the linear equation in $L^2_{x,p}$ identifies this semigroup with $S(t)$ from Theorem \ref{thm:linear}. In particular, 
$$ \| S(t) \|_{\mathcal{X} \to \mathcal{X}} \leq e^{ \| K \|_{\mathcal{X} \to \mathcal{X}}t}.
$$
Next, we are going to control the nonlinearities on $\mathcal{X}$: 
$$
\| \Gamma(g,g)+Q(g,g,g) \|_\mathcal{X} \lesssim \|g\|_\mathcal{X}^2 + \|g\|_\mathcal{X}^3. 
$$
Since $\omega \leq 1$, we have the set of inequalities: 
$X_{1/2,r}[\cdot] \leq X_{1/6,r}[ \cdot ]\leq X_{-1/2,r}[ \cdot ]$, which will allow us to use our a-priori estimates \eqref{eq:Gamma-pointwise} and \eqref{eq:pointwise_apriori}.  

Denoting by  
$x_\mu(\xi)  =  \|\omega^\mu \widehat{g}(\xi)\|_{L^\infty_p}$, we take the $x$-Fourier transform of the collision operators, which are convolutions. 
Then the pointwise estimates in \eqref{eq:Gamma-pointwise} and \eqref{eq:pointwise_apriori} give
\begin{align*}
\|\omega^{-1/2}\widehat{\Gamma(g,g)}(\xi)\|_{L^\infty_p} 
       &\lesssim (x_{1/2} \ast x_{1/2})(\xi),\\
    \|\omega^{-1/2}\widehat{Q(g,g,g)}(\xi)\|_{L^\infty_p}
       &\lesssim(x_{1/6} \ast x_{1/6} \ast x_{1/6})(\xi).
\end{align*}
 
By Young's inequality for convolutions, for $r=1,2$, we get 
\begin{align} \label{eq:Fourier_estimate1}
 X_{-1/2,r}[\Gamma(g,g)]
&\lesssim X_{1/2,1}[g]X_{1/2,r}[g], \quad 
    \text{ and }\quad 
X_{-1/2,r}[Q(g,g,g)]
    \lesssim X_{1/6,1}[g]^2X_{1/6,r}[g]. 
\end{align}

By the triangle inequality now these  estimates give
\begin{align*}
&\|\Gamma(g,g)+Q(g,g,g)\|_{\mathcal{X}} \leq \sum_{r=1}^2 \left( X_{-1/2,r}[\Gamma(g,g)] + X_{-1/2,r}[Q(g,g,g)] \right)\\
 & \hspace{2cm} \lesssim X_{1/2,1}[g]
\left(X_{1/2,1}[g]+X_{1/2,2}[g]\right) +
  X_{1/6,1}[g]^2
       \left(X_{1/6,1}[g] + X_{1/6,2}[g]\right)\\
    &\hspace{2cm} \lesssim
    \|g\|_{\mathcal{X}}^2 + \|g\|_{\mathcal{X}}^3,
\end{align*}
where we used $X_{1/2,r}[g]\leq X_{1/6,r}[g]$ and
$\|g\|_{\mathcal{X}} = X_{1/6,1}[g]+X_{1/6,2}[g]$.

Multilinearity also gives, whenever
$\|g\|_{\mathcal X},\|h\|_{\mathcal X}\leq R$,

\begin{align*}
\| [\Gamma(g,g)-\Gamma(h,h)] 
  + [ Q(g,g,g)-Q(h,h,h)] \|_{\mathcal{X}} \lesssim(R+R^2)\|g-h\|_{\mathcal{X}}.
\end{align*}

We then close with a typical fixed-point argument:  
we construct the mild solution $g\in C([0,T];\mathcal{X})$ that satisfies 
\begin{equation} \label{eq:mildsolution_NL}
g(t)=S(t)g_0 +\int_0^t S(t-s)
\big[
\Gamma(g(s),g(s)) +
Q(g(s),g(s),g(s))\big]
ds.
\end{equation} 
 which is a fixed point of the mapping 
$$ \Phi[g](t):=S(t)g_0 + \int_0^t 
S(t-s)
 [\Gamma(g(s),g(s))+Q(g(s),g(s),g(s))] ds.
$$ 
Set $R = 2\|g_0\|_{\mathcal{X}}$. On the closed ball
$\left\{  g\in C([0,T];\mathcal{X}):
 \sup_{0\leq t\leq T}\|g(t)\|_{\mathcal{X}}
 \leq R \right\},
$ 
we have
$$ \sup_{0\leq t \leq T}
\| \Phi[g](t)\|_{\mathcal{X}}
\leq e^{\|K\|_{\mathcal{X} \to \mathcal{X}}  T}\|g_0\|_{\mathcal{X}} + C T e^{ \|K\|_{\mathcal{X}\to \mathcal{X} }  T}(R^2+R^3),
$$
and
\begin{align*}
\sup_{0\leq t\leq T}  \|\Phi[g](t)-\Phi[h](t)\|_{\mathcal{X}}
    \leq C T e^{ \|K\|_{\mathcal{X}\to \mathcal{X} }T } (R+R^2)
\sup_{0\leq t\leq T} \|g(t)-h(t)\|_{\mathcal{X}}.
\end{align*}
Choose $T>0$ sufficiently small that,  $\Phi$ maps the ball into itself and is a contraction. 
Its fixed point is our unique local solution. 
The construction can be restarted at any time, and thus it gives a unique maximal solution on 
$[0,T_0)$. If the maximal time $T_0 < \infty$ and the norm is  bounded near $T_0$, restarting sufficiently close to $T_0$
would extend the solution beyond $T_0$.
\end{proof}

\begin{proof}[Proof of Theorem \ref{thm:nonlinear_stability}]
For $T<T_0$ we define the weighted norm 
$$ 
\|g\|_{\mathcal{B}_T} :=\sum_{r=1,2} \sup_{0\leq t\leq T} \left[ 
\langle t \rangle^{2/5} X_{1/6,r}[g(t)] + 
\langle t \rangle^{3/5-\delta}X_{1/2,r}[g(t)]
 \right].
$$
This is finite and continuous in $T$. The Fourier estimates \eqref{eq:Fourier_estimate1} together with the estimate
\begin{align*}
\|g\|_{\mathscr{Y}}
\leq 
\left(\int_{\mathbb{T}} \frac{\omega}{a} dp\right)^{1/2}
   \|\omega^{-1/2}g \|_{L^\infty_p}
\lesssim\|\omega^{-1/2}g\|_{L^\infty_p},
\end{align*}
because $a \sim \omega^{5/3}$,  as well as using the convolution estimate $ \|f^1 \ast f^2 \|_{L^\infty_\xi} \leq \|f^1\|_{L^2_\xi}\|f^2\|_{L^2_\xi},$ (which follows from Cauchy-Schwarz), we obtain: 
\begin{align*}
\|\widehat{\Gamma(g,g)}\|_{L^\infty_\xi\mathscr{Y}}
&\lesssim X_{1/2,2}[g]^2, \quad 
    \text{ and } \quad 
\|\widehat{Q(g,g,g)}\|_{L^\infty_\xi\mathscr{Y}} \lesssim X_{1/6,1}[g]X_{1/6,2}[g]^2. 
\end{align*} 

These give, for $s \leq T$, 
\begin{align*}
X_{-1/2,r}[\Gamma(g,g)] + \|\widehat{\Gamma(g,g)}\|_{L^\infty_\xi\mathscr{Y}} \lesssim \|g\|_{\mathcal{B}_T}^2 \langle s\rangle^{-6/5+2\delta},\quad 
X_{-1/2,r}[Q(g,g,g)] + \|\widehat{Q(g,g,g)}\|_{L^\infty_\xi\mathscr{Y}} &\lesssim\|g\|_{\mathcal B_T}^3 \langle s\rangle^{-6/5}, 
x\end{align*}
which allows us now to apply our linear decay estimates.  
Indeed, by assumption: $\Pi g_0(x, \cdot)=0$ for almost every $x$, and also 
\eqref{eq:proj_nonlinearities_zero} gives $\Pi\Gamma(g(s),g(s))=0$ and $\Pi Q(g(s),g(s),g(s))=0$ at every time $s$, and a.e. fixed $x$, by the discussion above. 

Applying  the linear Theorem \ref{thm:linear} to the Duhamel formula  \eqref{eq:mildsolution_NL} 
and using Minkowski's
inequality, for $(\mu,b)=(1/6,2/5)$ or
$(\mu,b)=(1/2,3/5-\delta)$, gives
\begin{align*}
 X_{\mu,r}[g(t)]
    \lesssim_\delta &
 \varepsilon\langle t\rangle^{-b} + \|g\|_{\mathcal B_T}^2
 \int_0^t\langle t-s\rangle^{-b}
   \langle s\rangle^{-6/5+2\delta} ds + \|g\|_{\mathcal B_T}^3
 \int_0^t\langle t-s\rangle^{-b}
  \langle s\rangle^{-6/5} ds.
\end{align*}
For $0<b<1<q$, splitting at $t/2$ gives
\begin{align*}
 \int_0^t\langle t-s\rangle^{-b}\langle s\rangle^{-q} ds  &\lesssim \langle t\rangle^{-b}\int_0^{t/2}\langle s\rangle^{-q} ds
  +\langle t\rangle^{-q}\int_0^{t/2}\langle s\rangle^{-b} ds \lesssim_{b,q}\langle t\rangle^{-b}.
\end{align*}

Therefore
\begin{equation}\label{eq:nonlinear-bootstrap}
\|g\|_{\mathcal{B}_T}
 \leq C_\delta\big[
\varepsilon+\|g\|_{\mathcal{B}_T}^2 
+\|g\|_{\mathcal{B}_T}^3\big].
\end{equation}

Enlarge $C_\delta$ to bound the initial value of this norm by
$C_\delta\varepsilon$. Choose $\varepsilon$ so small that
$$ C_\delta\big(
(2C_\delta\varepsilon)^2
+(2C_\delta\varepsilon)^3 \big)
 \leq \tfrac{1}{2} C_\delta\varepsilon.
$$
At a first time when 
$\|g\|_{\mathcal{B}_T}$ reached
$2C_\delta \varepsilon$,  \eqref{eq:nonlinear-bootstrap} would
instead bound it by $\tfrac{3}{2} C_\delta \varepsilon$.
Continuity excludes such a time. Thus
$\|g\|_{\mathcal{B}_T}\leq 2C_\delta \varepsilon$ for every $T<T_0$. The blow-up criterion gives $T_0=\infty$.

This proves \eqref{eq:nonlinear1/2}, and Fourier inversion gives the
physical pointwise bounds. Applying the unweighted estimate
of Theorem \ref{thm:linear} to \eqref{eq:mildsolution_NL},
with the same source bounds and time convolution, proves 
\eqref{eq:nonlinear_L2}.
\end{proof}

\begin{remark} 
We also claim that spatial Sobolev regularity propagates under the assumptions of the global stability theorem. In particular, if
$g_0 \in H^M_x L^2_p$ for an integer $M\geq 0$, then 
$g\in C([0,\infty);H^M_xL^2_p)$ and
$$ \sup_{t\geq0}\mathcal{E}_M[g(t)]
+\int_0^\infty\mathcal D_M[g(t)]\,dt
    \lesssim_{M,\delta}\mathcal{E}_M[g_0],
$$
where
$$
\begin{aligned}
\mathcal E_M[g] &:=
 \int_{\mathbb{R}}\langle \xi\rangle^{2M}
    \|\widehat{g}(\xi)\|_{L^2_p}^2 d\xi, \qquad
 \mathcal{D}_M[g]
 &:= 
 \int_{\mathbb{R} }\langle\xi\rangle^{2M}
 \|a^{1/2}\Pi_a^\perp
 \widehat{g}(\xi)\|_{L^2_p}^2 d\xi.
\end{aligned} 
$$
The proof follows from an energy estimate using bounds analogous to \eqref{eq:Gamma-pointwise} and
\eqref{eq:pointwise_apriori} for $\Gamma$ and $Q$,
but with one input in $L^2_p$, together with our nonlinear decay estimates proved above and Gr\"onwall's inequality.
We omit the details. 
\end{remark}

\subsection{Acknowledgements}  
The authors thank Pierre Germain for useful discussions. AM acknowledges support from the Engineering and Physical Sciences Research Council fellowship with reference UKRI2025. HX acknowledges support from a Roth Scholarship at Imperial College London and from the Royal Society Wolfson Fellowship awarded to Pierre Germain.  
\\

\noindent
\textbf{AI declaration:} 
The authors used ChatGPT 5.6 Sol during the development and preparation of this paper, and ChatGPT Astra for reviewing the final proofs.   
This work combines ideas from the
authors and several useful exchanges with AI tools.
We give here a short account of the AI contributions that we regard as most important.

AI pointed us to the reduction to the Feshbach operator in Section \ref{sec: lowfrequency_spectrum}. 
It was also instrumental in refining the bounds involving the operator $K$ and the collision frequency $a$: in particular the smoothing estimate for $K$ in Lemma \ref{lem:app_Kernelbound_edge}. These allowed us to improve our earlier decay rates for the linear semigroup  in Section \ref{sec:linear}, by obtaining better control of the edges, and establishing the convergence and quantitative bounds for the eigenvectors in Lemma \ref{lem:estimates_linear}.   
ChatGPT also generated the TikZ code for 
Figure \ref{fig:lowfrequency_spectrum}, illustrating the low-frequency spectrum.  

All suggestions and computations provided by AI were independently verified and, wherever necessary, corrected and rewritten by the authors, who of course take full responsibility.

\appendix

\section{Some properties of the collision operator} 

\subsection{The weighted eigenspace $\mathscr{N}$} \label{app:appendix_eigenspace} 
We include here a short proof of the fact that 

$$\ker(I-\mathscr{K} )
= \mathscr{N}:=a^{1/2}\ker L.$$

First, if $h \in \ker L$, then $Kh=ah$, or 
$\mathscr{K} (a^{1/2}h) = a^{1/2}h$. Thus
$\mathscr{N} \subseteq \ker(I-\mathscr{K})$.

For the reverse inclusion, weighted coercivity
\cite[Proposition 4.1]{GLM26} and
\eqref{eq:relation projections} give, for every
$h\in L^2_p$,
$$
\big \langle (I-\mathscr{K} )a^{1/2}h, a^{1/2}h \big\rangle
 =-\langle Lh, h\rangle \geq c\|a^{1/2}(I-\Pi_a)h\|_{L^2_p}^2
 =c\|\mathsf{Q} a^{1/2}h\|_{L^2_p}^2.
$$
Since $a>0$ almost everywhere, multiplication by $a^{1/2}$
has dense range in $L^2_p$ and thus by boundedness of $\mathscr{K}$ and $\mathsf{Q}$, this inequality extends to every
$w\in L^2_p$:
$$ \langle(I-\mathscr{K} ) w, w \rangle
\geq c \|\mathsf{Q} w\|_{L^2_p}^2.
$$
Now if $\mathscr{K} w=w$, implies that $\mathsf{Q} w = 0$, 
so $w\in\mathscr{N}$. This proves the equality.

\subsection{Weighted $L^2$ to $L^R$ estimate for $K$} \label{app:Kernel_smoothing_edges}

In the following Lemma, used several times throughout the manuscript, we show that the Kernel $K$ improves  integrability form $L^2$ to $L^R$, $R>2$.  
The estimate includes a negative power of $\omega$, which gives additional control of $Kh$ near the edges (in a weighted $L^R_p$ sense). 

\begin{lemma}
\label{lem:app_Kernelbound_edge}
For every $\varepsilon>0$ and
$R>\max\{2,2/\varepsilon\}$,
$$
\|\omega^{-3/2+\varepsilon}Kh\|_{L^R_p}
\lesssim_{\varepsilon,R} \|h\|_{L^2_p}.
$$
\end{lemma}

\begin{proof}
Since $\gamma>0$, the RJ profile $\mathfrak{f}$ is bounded above and below by positive constants.  Its factors will be absorbed in the
implicit constants. We may write, as in \cite{GLM26}, that
$$ K=-K_1+2K_2.
$$
The estimate for $K_2$ in \cite[Proof of Lemma 8(i)]{GLM26} gives, for every $\varepsilon > 0$,
$$
|K_2h(p)|
    \lesssim_\varepsilon
    \omega(p)^{3/2-\varepsilon}\|h\|_{L^2_p}.
$$
which proves the required estimate for
$K_2$.  
After the parametrisation, the kernel $K_1$\footnote{We abuse notation and identify an
operator and its kernel} can be written in terms of the function $F_-$: 
$$ F_-(p,q):=\left(\cos\frac{p}{2}+\cos\frac{q}{2}\right)^2
      -4\sin\frac{p}{2} \sin\frac{q}{2}.
$$
The function $F_-(p,\cdot)$ has two roots $r_-(p)<r_+(p)$, \cite[Lemmas 15]{GLM26}.  Combined with \cite[Lemma 19]{GLM26} 
these give the $K_1$-bound: 
\begin{equation} \label{eq:appendix_K1_bound}
|K_1h(p)|\lesssim\omega(p)^{3/2}
\bigg[
 \int_0^{r_-(p)}
\frac{|h(q)|}{\sqrt{r_-(p)-q}} dq
 +\int_{r_+(p)}^{2\pi}
\frac{|h(q)|}{\sqrt{q-r_+(p)}} dq
\bigg].
\end{equation}

To estimate these integrals, we perform the following change of variables: 
$ A=\frac{p+q}{4}, B=\frac{p-q}{4}$, after which the equation $F_-(p,q)=0$ becomes $\tan^2(A) = \sec^2(B)$. 
Since $r_-(p)<2\pi-p<r_+(p)$, the two branches satisfy $
    \tan A=\sec B $ when $q=r_-(p)$, and 
$ \tan A=-\sec B$ when $q=r_+(p)$. 

Define $$H_\sigma(p,q):=\tan A-\sigma\sec B \ \text{ for } \ \sigma\in\{+1,-1\}.$$
Writing $P=\sec^2A$,
   $Q=\sec B\tan B$, 
we compute $\partial_p H_\sigma = \frac{1}{4} (P-\sigma Q),
$ and $\partial_q H_\sigma=\frac{1}{4} (P+\sigma Q)$.  On each branch we have that $\partial_q H_\sigma > 0$: Indeed for $t = \tan B$, we have
$ P=2+t^2$, $
    Q=t\sqrt{1+t^2},$ 
and thus 
$$ P^2-Q^2=(2+t^2)^2-t^2(1+t^2) =4+3t^2 > 0.
$$
Since $P>0$, this means that $P >|Q|$. 
 An application of the implicit function theorem now gives
$C^1$ root maps, with
\begin{equation}
    r_-'(p)=-\frac{P-Q}{P+Q} <0,
    \qquad
    r_+'(p)=-\frac{P+Q}{P-Q}<0.
\end{equation}

Thus both root maps are strictly decreasing and have no critical
point in $(0,2\pi)$.

Since $ F_-(0,q)=\left(1+\cos\frac{q}{2} \right)^2$, 
vanishes only at $q=2\pi$ on $[0,2\pi]$, continuity
implies that both roots tend to $2\pi$ as $p\to 0$. 
If we write $q = 2\pi-z$ on either branch, the equation becomes
$$  \left(\cos\frac{p}{2} - \cos\frac{z}{2} \right)^2
    =4\sin\frac{p}{2}\sin\frac{z}{2}, 
$$
which using the identity $ \cos\frac{p}{2}-\cos\frac{z}{2}
    = -2\sin\frac{p+z}{4}\sin\frac{p-z}{4}. $ This yields the asymptotic behaviour 
\begin{equation}
    (z^2-p^2)^2\sim pz \qquad p , z\to 0.
\end{equation}
Now observe that the ratio $z/p$ cannot remain in a fixed compact subset
of $(0,\infty)$: otherwise the left-hand side  would be $O(p^4)$,
while its right-hand side would be comparable to $p^2$. Now observe that the ratio $z/p$ cannot remain in a fixed compact subset
of $(0,\infty)$: otherwise the left-hand side  would be $O(p^4)$,
while its right-hand side would be comparable to $p^2$.
On the lower branch $q=r_-(p)$, we have $z>p$,
so $z/p \to \infty$. From the above asymptotics this gives $z^4\sim pz$, equivalently 
$z \sim p^{1/3}$. Now on the upper branch $q=r_+(p)$, we have $p>z$ and $z/p \to 0$, which from the above asymptotics we rather get $z\sim p^3$. 
Thus
$$ 2\pi-r_-(p)\sim p^{1/3}, \text{ and }
    \quad 2\pi-r_+(p)\sim p^3,
    \qquad p \to 0.$$ 
Reflection symmetry, i.e.  $F_-(2\pi-p,2\pi-q)=F_-(p,q)$ now implies $r_-(2\pi-p)=2\pi-r_+(p),$
 and $r_+(2\pi-p)=2\pi-r_-(p)$. Thus  
$$ r_-(p)\sim(2\pi-p)^3, \text{ and }
    \quad r_+(p)\sim(2\pi-p)^{1/3},
 \qquad p \to 2\pi.
$$
Combined with strict monotonicity, these limits show that each $r_\pm$ maps $(0,2\pi)$ onto $(0,2\pi)$. 
Since their derivatives do not vanish, in particular both maps are $C^1$-diffeomorphisms. Thus we can now find  the asymptotics of their derivatives. As $p \to 0$ along either root branch: 
$ t<0$ and $|t| \sim (p+z)^{-1} \uparrow \infty$. 
Using the formulas for $P,Q$ above: 
$$
    P-Q=2+t^2+|t|\sqrt{1+t^2} \sim t^2,\quad P+Q
    =\frac{P^2-Q^2}{P-Q}
    =\frac{4+3t^2}{P-Q}
    \sim 1.
$$ 
Collecting everything: 
\begin{equation}
    \begin{cases}
        \text{On the lower branch, } \  z\sim p^{1/3}, \qquad |r_-'(p)|
    = \frac{P-Q}{P+Q}  \sim(p+z)^{-2}
    \sim p^{-2/3} \\
    \text{On the upper branch, }\  z\sim p^3, \qquad |r_+'(p)|  = \frac{P+Q}{P-Q}
    \sim(p+z)^2  \sim p^2,  
    \end{cases}
\end{equation} 
with the analogous reflected estimates at $2\pi$. 
Therefore, if $p_\pm=r_\pm^{-1}$, from the inverse function theorem 
$$|p_\pm'(y)|   =
    \frac{1}{|r_\pm'(p_\pm(y))|}, $$
we obtain 
$$ \sup_{0 < y < 2\pi}
    \omega(p_\pm(y))^{\varepsilon R}| p_\pm'(y) |
    =  \sup_{0 < p < 2\pi}
    \frac{\omega(p)^{\varepsilon R}}{|r_\pm'(p)|}.
$$
Near $p=0$, the two quotients on the right are comparable to $ p^{\varepsilon R+2/3}$ and 
 $ p^{\varepsilon R-2}$, 
respectively. Near $p = 2\pi$, the roles are reversed with $p$ replaced by $2\pi-p$.  Both are bounded when $\varepsilon R>2$. 
We have therefore that whenever $\varepsilon R>2: $
\begin{equation} \label{eq:appendix_bounded_inverseFT}
\sup_{0< y <2\pi} \omega(p_\pm(y))^{\varepsilon R}|p_\pm'(y)|  < \infty, 
\end{equation}
since on every compact interval of $(0,2\pi)$, boundedness follows from continuity and nonvanishing of $r_\pm'$. 

Now define the integrals 
$$ I^-|h|(y) := \int_0^y \frac{|h(q)|}{\sqrt{y-q}} dq,
    \qquad
    I^+ |h|(y) :=
    \int_y^{2\pi} \frac{|h(q)|}{\sqrt{q-y}} dq, 
$$ 
and extend $h$ by zero outside $(0,2\pi)$, so that each of these are bounded by $\int_{ \mathbb{R} } |y-q|^{-1/2} |h(q)| dq$, for $y \in(0,2\pi)$.  
Let $2 <R < \infty$ and take $1<s<2$ so that $1/R=1/s-1/2$.   
From the 1-dim Hardy-Littlewood-Sobolev inequality we get  
$$ \|I^\pm |h|\|_{L^R(0,2\pi)}
\lesssim_R \|h\|_{L^s(0,2\pi)} \lesssim_R 
\| h\|_{L^2 (0,2\pi) }, 
$$
where the second inequality followed by Hölder's, since the interval has finite measure. 

Finally, multiply
\eqref{eq:appendix_K1_bound} by
$\omega^{-3/2 + \varepsilon}$ and write 
\begin{align*}
    \|\omega^{-3/2+\varepsilon}K_1 h \|_{L^R_p}
    &\lesssim
    \|\omega^\varepsilon ( (I^-|h|)\circ r_-)\|_{L^R_p} +
    \|\omega^\varepsilon
        ((I^+ |h|)\circ r_+)\|_{L^R_p}\\
    &\lesssim_{\varepsilon,R}
    \|I^- |h|\|_{L^R_p} +\|I^+ |h|\|_{L^R_p} \lesssim_{\varepsilon,R}\|h\|_{L^2_p}, 
\end{align*}
where in the second line we used that for any measurable function $F$, changing variables
$ y= r_\pm (p) $ and using  \eqref{eq:appendix_bounded_inverseFT}: 
\begin{align*}
    \| \omega^\varepsilon(F \circ r_\pm)\|_{L^R_p} &=  \left[ \int_0^{2\pi}
        \omega(p)^{\varepsilon R}|F(r_\pm(p))|^R dp \right]^{\frac{1}{R}} =  \left[ \int_0^{2\pi}
        \omega(p_\pm(y))^{\varepsilon R}
        |p_\pm'(y)| |F(y)|^R dy \right]^{\frac{1}{R}} \\
    &\lesssim_{\varepsilon,R}
    \|F\|_{L^R_y}, 
\end{align*}
and in  the last step the Hardy-Littlewood-Sobolev inequality. 
\end{proof}

\subsection{Sharp asymptotics of the collision frequency}
\label{app:sharp_collision_frequency}
Here we provide the proof of the optimal asymptotics of the collision frequency $a(p)$ for $p \to 0^+$ or $p \to  2\pi^-$, as was  found in \cite{LukkSpohn}, \cite{GLM26}. We revisit these estimates and refine them by  determining the prefactor $a_0$: 

\begin{lemma}
\label{lem:sharp_a_limit_appendix}
Assume $\gamma>0$. There exists $a_0>0$ such that
\begin{equation}
    a(p)=a_0\omega(p)^{5/3}
    +\smallO \left(\omega(p)^{5/3}\right)
\end{equation} 
as $p\to0^+$ or $p\to2\pi^-$. In particular, the two endpoint
constants agree by reflection symmetry.
\end{lemma}

\begin{proof}
We use the representation \cite[(4.2)]{GLM26} and the identities \cite[(A.11)-(A.12)]{GLM26}.
By $a(2\pi-p)= a(p)$, it is enough to let $p \to 0^+$ and set
$ \epsilon:=\omega(p)=\sin\frac{p}{2}$.  
Writing $p_1=h(p,p_2)$ and
$p_3=p+p_1-p_2$ modulo $2\pi$, the collision frequency is exactly
\begin{equation}\label{eq:collision_frequency_exact}
a(p) = \frac{\epsilon}{\mathfrak{f}(p)}
 \int_0^{2\pi} \frac{
\omega(p_1)\omega(p_2)\omega(p_3)
  \mathfrak{f}(p_1)\mathfrak{f}(p_2)\mathfrak{f}(p_3)}{\sqrt{F_+(p,p_2)} }dp_2.
\end{equation}

Fix a small $q_0\in(0,\pi)$. On
$0\leq p_2\leq2\pi-q_0$, identity \cite[(A.11)]{GLM26} gives
$$ \omega(p_1)\omega(p_2)\omega(p_3) =
\tan^2 \left(\frac{p_2-p}{4}\right)
\sin^2 \left(\frac{p_2}{2}\right)\epsilon
\lesssim_{q_0}\epsilon. 
$$
Moreover, for all sufficiently small $p$,
$\cos(p/2)+\cos(p_2/2)$ is bounded below by a positive constant 
depending only on $q_0$. Thus
$F_+(p,p_2)^{-1/2}\lesssim_{q_0}1$ there. Since $\gamma>0$, all factors  $\mathfrak{f}$, and $\mathfrak{f}^{-1}$  are uniformly
bounded and thus the contribution of this region to
\eqref{eq:collision_frequency_exact} is
\begin{equation}\label{eq:asymptot_on_(0,pi) q_0}
    \mathcal{O}_{q_0}(\epsilon^2)=\smallO(\epsilon^{5/3}).
\end{equation} 

In the remaining corner put
$ q:=2\pi-p_2$, $ s:=\sin\frac{q}{2}$, and $s_0:=\sin\frac{q_0}{2}$. 
The change of variables is:
\begin{equation}\label{eq:edge_jacobian}
    p_2=2\pi-2\arcsin s,
    \qquad
    |dp_2|=\frac{2 ds}{\sqrt{1-s^2}}.
\end{equation}
For brevity, set
$ c_\epsilon:=\sqrt{1-\epsilon^2}$, $
 r_s:=\sqrt{1-s^2}$. Now  \cite[Identity (A.12)]{GLM26} and the definition of $F_+$ give 
\begin{align} \label{eq:edge_product+exact}
\mathcal{P}(\epsilon,s)
 := \omega(p_1)\omega(p_2)\omega(p_3)
 = \frac{\epsilon s^2(1-\epsilon s+c_\epsilon r_s)
 }{1-c_\epsilon r_s+\epsilon s}, 
\end{align}
and 
\begin{align}  \label{eq:edge_denominator_exact}
 \mathcal{F} (\epsilon,s)
 := F_+(p,p_2) 
 = (c_\epsilon-r_s)^2+4 \epsilon s.
\end{align}
Consequently, the corner contribution is exactly
\begin{equation}  \label{eq:edge_integral_exact}
a_{\rm cor}(p) =
\frac{2\epsilon}{\mathfrak{f}(p)} \int_0^{s_0}
\frac{ \mathcal{P} (\epsilon,s)
 \mathfrak{f}(p_1)\mathfrak{f}(p_2)\mathfrak{f} (p_3) }{ \sqrt{\mathcal{F}(\epsilon,s)}\sqrt{1-s^2}} ds.
\end{equation} 
After decreasing $q_0$, if necessary, these exact formulas give the
uniform normal forms
\begin{equation}  \label{eq:edge_sharp_F_P}
    \mathcal{P} (\epsilon,s)
    \sim\frac{\epsilon s^2}{(\epsilon+s)^2},
    \qquad
    \mathcal{F} (\epsilon,s)
    \sim(s^2-\epsilon^2)^2+\epsilon s.
\end{equation}
Indeed,
$1-c_\epsilon r_s+\epsilon s\sim(\epsilon+s)^2$, the other factor
in the numerator of $\mathcal{P}$ stays bounded above and below by
positive constants, and
\[ 
    c_\epsilon-r_s
    =\frac{s^2-\epsilon^2}{c_\epsilon+r_s}.
\]

For $s=\epsilon^{1/3}y$, and extending  the rescaled
density by zero for $y\geq s_0\epsilon^{-1/3}$, 
\eqref{eq:edge_integral_exact} becomes
\begin{equation}
    \frac{a_{\rm cor}(p)}{\epsilon^{5/3}}
    =\int_0^\infty D_\epsilon(y) dy, 
\end{equation}
where
\begin{equation}
\begin{aligned}
    D_\epsilon(y)
    :=
    \frac{2}{\mathfrak{f}(p)}
    \frac{
        \epsilon^{-1}
        \mathcal P(\epsilon,\epsilon^{1/3}y)
}{
        \sqrt{\epsilon^{-4/3}
          \mathcal F(\epsilon,\epsilon^{1/3}y)}
    }
    \frac{\mathfrak{f}(p_1)\mathfrak{f}(p_2)\mathfrak{f}(p_3)}{
        \sqrt{1-\epsilon^{2/3}y^2}}
    \mathbf1_{(0,s_0\epsilon^{-1/3})}(y).
\end{aligned}
    \label{eq:edge_rescaleddensity}
\end{equation}
For every fixed $y>0$, direct expansion of
\eqref{eq:edge_product+exact} and 
\eqref{eq:edge_denominator_exact} gives
\begin{equation}
    \epsilon^{-1}
    \mathcal P(\epsilon,\epsilon^{1/3}y)
    \to 4, \quad 
    \epsilon^{-4/3}
    \mathcal F(\epsilon,\epsilon^{1/3}y)
    \to \frac{y^4}{4}+4y.
\end{equation}
Also $p_2$ approaches an endpoint. On the resonant manifold,
\[ \omega(p_1)-\omega(p_3)=s-\epsilon,
    \qquad
    \omega(p_1)\omega(p_3)
    =\frac{\mathcal P(\epsilon,s)}{s}.
\]
At $s=\epsilon^{1/3}y$, both right-hand sides tend to zero and so 
$\omega(p_1),\omega(p_3)\to 0$. Hence
$ \mathfrak{f}(p)\to \gamma^{-1},
$ $ \mathfrak{f}(p_j)\to \gamma^{-1},$ for $ j=1,2,3$. 
Therefore altogether 
\begin{equation}
    D_\epsilon(y)
    \longrightarrow
    \frac{8\gamma^{-2}}{\sqrt{y^4/4+4y}}.
\end{equation}

It remains to justify the limit under the integral. The
RJ factors $\mathfrak{f}$ are uniformly bounded, and the Jacobian in
\eqref{eq:edge_jacobian} is uniformly bounded on
$0<s<s_0<1$. From \eqref{eq:edge_sharp_F_P},
\begin{equation} \label{eq:edge_sharpdominator}
    0\leq D_\epsilon(y)\leq C\left(
      y^{-1/2}\mathbf1_{(0,1)}(y)
      +y^{-2}\mathbf1_{[1,\infty)}(y)
    \right).
\end{equation}
For $0<y<1$, use $\mathcal F\gtrsim\epsilon s$. If
$s\geq\epsilon$, then $\mathcal P\lesssim\epsilon$; if
$s\leq\epsilon$, then $\mathcal P\lesssim s^2/\epsilon$ and
$y\leq\epsilon^{2/3}$. These two cases give
$D_\epsilon(y)\lesssim y^{-1/2}$. If $y\geq1$, then
$s\gg\epsilon$, so $\mathcal P\lesssim\epsilon$ and
$\mathcal F\gtrsim s^4$, giving
$D_\epsilon(y)\lesssim y^{-2}$.

The majorant in \eqref{eq:edge_sharpdominator} is integrable and thus 
dominated convergence and \eqref{eq:asymptot_on_(0,pi) q_0} yield the 
exact positive limit
\begin{equation}
\lim_{p\to 0^+}\frac{a(p)}{\omega(p)^{5/3}}
 =a_0 := 8\gamma^{-2} \int_0^\infty
    \frac{dy}{\sqrt{y^4/4+4y}}
    \in(0,\infty).
\end{equation}
The integrand is comparable to
$y^{-1/2}$ at zero and to $y^{-2}$ at infinity, so the displayed
constant is finite and strictly positive. Reflection symmetry gives
the same limit as $p \to 2\pi^-$.
\end{proof}

\bibliographystyle{alpha}
\bibliography{reference}

\end{document}